\documentclass [a4paper,12pt]{amsart}
\usepackage{graphicx}
\usepackage[ansinew]{inputenc}    
\usepackage[all]{xy}

\usepackage{amssymb,amscd,empheq}
\usepackage{amsmath,amsfonts}
\usepackage{amsthm}

\usepackage[normalem]{ulem} 

\usepackage{hyperref}

\newtheorem{theorem}{Theorem}[section]
\newtheorem{definition}[theorem]{Definition}
\newtheorem{proposition}[theorem]{Proposition}
\newtheorem{lemma}[theorem]{Lemma}
\newtheorem{corollary}[theorem]{Corollary}
\newtheorem{remark}[theorem]{Remark}

\usepackage{graphicx}
\usepackage[ansinew]{inputenc}     
\usepackage[all]{xy}
\usepackage{amssymb,amscd}
\usepackage{color}

\numberwithin{equation}{section}

\title[Long-time behaviour for a non-autonomous Klein-Gordon-Zakharov system]{Long-time behaviour for a non-autonomous Klein-Gordon-Zakharov system}

\author[E. M. Bonotto]{Everaldo M. Bonotto$^{\dag}$}\thanks{$\dag$Research partially
supported by FAPESP \# 2019/03188-7 and CNPq  \#  310540/2019-4}
\address[E. M. Bonotto]{Instituto de Ci\^encias Matem\'{a}ticas e de Computa\c{c}\~ao, Universidade de S\~{a}o Paulo, Campus de S\~{a}o Carlos, Caixa Postal 668, 13560-970, S\~{a}o Carlos SP, Brazil.}
\email{ebonotto@icmc.usp.br}

\author[M. J. D. Nascimento]{Marcelo J. D. Nascimento$^\star$}\thanks{$^\star$Research partially
supported by FAPESP \# 2017/06582-2, Brazil}
\address[M. J. D. Nascimento]{Universidade Federal de S\~{a}o
Carlos, Departamento de Matem\'atica, 13565-905, S\~{a}o
Carlos SP, Brazil.}
\email{marcelo@dm.ufscar.br}

\author[E. B. Santiago]{Eric B. Santiago$^{\star\star}$}\thanks{$^{\star\star}$This study was financed in part by the Coordenação de Aperfeiçoamento de Pessoal de Nível Superior – Brasil (CAPES) – Finance Code 001}
\address[E. B. Santiago]{Universidade Federal de S\~{a}o
Carlos, Departamento de Matem\'atica, 13565-905, S\~{a}o
Carlos SP, Brazil.}
\email{eric.busatto@gmail.com}

\date{\today}

\begin{document}

\begin{abstract} 
The aim of this paper is to study the long-time dynamics of solutions of the evolution system
\[
\begin{cases}
u_{tt} - \Delta u + u + \eta(-\Delta)^{\frac{1}{2}}u_t + a_{\epsilon}(t)(-\Delta)^{\frac{1}{2}}v_t = f(u), & \;  (x, t) \in \Omega \times (\tau, \infty), \\ 
v_{tt} - \Delta v + \eta(-\Delta)^{\frac{1}{2}}v_t - a_{\epsilon}(t)(-\Delta)^{\frac{1}{2}}u_t = 0, & \; (x, t) \in \Omega \times (\tau, \infty), 
\end{cases}
\]
subject to boundary conditions
\[
u = v = 0,   \;\;   (x, t) \in   \partial\Omega\times (\tau, \infty), 
\]
where $\Omega$ is a bounded smooth domain in $\mathbb{R}^n$, $n \geq 3$, with the boundary $\partial\Omega$ assumed to be regular enough, $\eta > 0$ is constant, $a_{\epsilon}$ is a Hölder continuous function and $f$ is a dissipative nonlinearity. This problem is a non-autonomous version of the well known Klein-Gordon-Zakharov system. Using the uniform sectorial operators theory, we will show the local and global well-posedness of this problem in $H_0^1(\Omega) \times L^2(\Omega) \times H_0^1(\Omega) \times L^2(\Omega)$. Additionally, we prove existence, regularity and upper semicontinuity of pullback attractors.

\vskip .1 in \noindent {\it Mathematics Subject Classification 2020}: Primary: 35B41, 35B40. Secondary: 35B65, 35K40. 
\newline {\it Key words and phrases:} Klein-Gordon-Zakharov system, Non-autonomous problem, pullback attractor, Global well-posedness, upper semicontinuity.

\end{abstract}

\maketitle

\section{Introduction}

In this paper, we study a non-autonomous version of the well known Klein-Gordon-Zakharov system. We consider the following initial-boundary value problem
\begin{equation}\label{edp01}
\begin{cases}
u_{tt} - \Delta u + u + \eta(-\Delta)^{\frac{1}{2}}u_t + a_{\epsilon}(t)(-\Delta)^{\frac{1}{2}}v_t = f(u), & (x, t) \in \Omega \times (\tau, \infty), \\
v_{tt} - \Delta v + \eta(-\Delta)^{\frac{1}{2}}v_t - a_{\epsilon}(t)(-\Delta)^{\frac{1}{2}}u_t = 0, & (x, t) \in \Omega \times (\tau, \infty), 
\end{cases}
\end{equation}
where $\eta$ is a positive constant, subject to boundary conditions
\begin{equation}\label{boundary condition}
u = v = 0, \;  (x, t) \in   \partial\Omega\times (\tau, \infty), 
\end{equation}
and initial conditions
\begin{equation}\label{cond01}
u(\tau, x) = u_0(x), \ u_t(\tau, x) = u_1(x), \ v(\tau, x) = v_0(x), v_t(\tau, x) = v_1(x), \ x \in \Omega, \ \tau \in \mathbb{R},
\end{equation}
where $\Omega$ is a bounded smooth domain in $\mathbb{R}^n$ with $n \geq 3$, and the boundary $\partial\Omega$ is assumed to be regular enough. 

In the case that  $a_{\epsilon}(t)\equiv a$,  the system \eqref{edp01}  represents the autonomous version of the Klein-Gordon-Zakharov system. Within the autonomous case, if $n=3$ then the Klein-Gordon-Zakharov system arises to describe the interaction of a Langmuir wave and an ion acoustic wave in a plasma, see \cite{Bellan, Dency, OT} and references therein.

These types of systems have been considered by many researchers in recent years.  In what follows,  we recall some related results  for these kinds of systems. In \cite{OT},  the authors considered the following system (in dimension $2$ and $3$)
\[
\begin{cases}
u_{tt}- \Delta u + u +vu = 0, \\
v_{tt} -c_0^2 \Delta v = \Delta(|u|^2),
\end{cases}
\]
 and they proved instability of solutions in the sense that small perturbations of the initial data can make the perturbed solution blow up in finite time.

In \cite{Almeida and Santos}, it is considered the following coupled system of wave equations:
\[
\begin{cases}
u_{tt} - \Delta u + \int_{0}^{+\infty} g(s) \Delta u(t - s) ds + \alpha v = 0, \\
v_{tt} - \Delta v + \alpha u = 0,
\end{cases}
\]
where the authors showed  the dissipativeness of this system, and, moreover, they  
proved that the associated semigroup is not exponentially stable. Later in \cite{Jin-Liang-Xiao}, the authors studied a more general and abstract version of the previous system presented in \cite{Almeida and Santos}. In fact, they obtained existence of solutions and an optimal energy decay estimate for the following coupled system of second order abstract evolution equations:
\[
\begin{cases}
u_{tt}(t) + A_{1}u(t) - \int_{0}^{+\infty} g(s) Au(t - s) ds + Bv(t) = 0, \\
v_{tt}(t) + A_{2}v(t) + Bu(t) = 0,
\end{cases}
\]
where $A, A_1$ and $A_2$ are positive self-adjoint linear operators in a Hilbert space $H$, $B$ is a positive self-adjoint bounded linear operator in $H$, and $g$ is a non-increasing function satisfying some properties. With this formulation, this system covers the well-known Timoshenko system, which appears in mechanics and thermoelasticity, and  models the transverse vibrations of a beam.

For a deeper and more detailed discussion about systems consisting of wave equations and other types of physical models, we refer to \cite{GZ}, \cite{Gan},\cite{Ma-Qin}, \cite{MN} and \cite{OTT}.

The main purpose of this paper is to show the global well-posedness and to study the long-time dynamics of solutions of the evolution system \eqref{edp01}.  In order to do that,  we shall use the uniform sectorial operators theory to show the local and global well-posedness of system  \eqref{edp01} and
we will use the abstract evolution processes theory to prove existence, regularity and upper semicontinuity of pullback attractors.

We emphasize that, in general, to obtain existence of attractors we need some type of ``dissipation" and ``compactness" for the dynamical system associated with the problem. In the literature, for non-autonomous problems, the ``compactness" is the so called pullback asymptotic compactness, and this is obtained by decomposing the nonlinear process into two parts, where one part decays to zero and the other one is compact. See \cite{BCNS1}, \cite{BCNS2}, \cite{Livro Alexandre}, \cite{Rivero1} and \cite{Rivero2}  for more details. However, in this paper, we establish the compactness of the  nonlinear process in a direct way, see Proposition \ref{proc-c}. 

This paper is organized as follows: The main results are presented in Section \ref{MResults}. We recall some concepts associated to the theory of pullback attractors in Section \ref{Preli}. In Section \ref{global}, we obtain the global well-posedness of solutions. Section \ref{PullA} is devoted to the existence of pullback attractors. Regularity of pullback attractors is obtained in Section \ref{RegA}. Finally in Section \ref{UpSem}, we study the upper semicontinuity of pullback attractors.

\section{Main Results}\label{MResults}

 In this section, we present the statement of the main results which will be proved in the next sections.  We start by presenting the general conditions to obtain the local and global well-posedness of the problem $(\ref{edp01})-(\ref{cond01})$ in some appropriate space  which will be specified later. 
Assume that the function $a_{\epsilon}\colon  \mathbb{R} \to (0, \infty)$ is continuously differentiable in $\mathbb{R}$ and satisfies the following condition:
\begin{equation}\label{function a is bounded}
	0 < a_0 \leq a_{\epsilon}(t) \leq a_1,
\end{equation}
for all $\epsilon \in [0,1]$ and $t\in \mathbb{R}$, with positive constants $a_0$ and $a_1$, and we also assume that the first derivative of $a_{\epsilon}$ is uniformly bounded in $t$ and $\epsilon$, that is, there exists a constant $b_0 > 0$ such that
\begin{equation}\label{derivative-a-bounded}
	|a_{\epsilon}^{\prime}(t)| \leq b_0 \quad \text{for all} \quad t\in \mathbb{R}, \ \epsilon\in [0, 1].
\end{equation}
Furthermore, we assume that $a_{\epsilon}$ is $(\beta, C)$-H\"{o}lder continuous, for each $\epsilon \in [0,1]$; that is,
\begin{equation}\label{hol-a}
	|a_{\epsilon}(t) - a_{\epsilon}(s)| \leq C|t-s|^{\beta}
\end{equation}
for all $t, s\in \mathbb{R}$ and $\epsilon \in [0,1]$. Concerning the nonlinearity $f$, we assume that $f\in C^1(\mathbb{R})$ and it satisfies the dissipativeness condition
\begin{equation}\label{dissipativeness}
	\limsup\limits_{|s| \to \infty}\frac{f(s)}{s} \leq 0,
\end{equation}
and also satisfies the subcritical growth condition given by
\begin{equation}\label{Gcondition}
	|f^{\prime}(s)| \leq c(1 + |s|^{\rho - 1}),
\end{equation}
for all $s\in \mathbb{R}$, where $1<\rho<\frac{n}{n-2}$, with $n \geq 3$, and $c>0$ is a constant. 

In order to formulate the non-autonomous problem $(\ref{edp01})-(\ref{cond01})$ in  a nonlinear evolution process setting, we introduce some notations. Let $X = L^2(\Omega)$ and denote by $A\colon  D(A)\subset X \to X$ the negative Laplacian operator, that is, $Au = (-\Delta)u$ for all $u \in D(A)$, where
 $D(A) = H^2(\Omega) \cap H_0^1(\Omega)$. Thus $A$ is a positive self-adjoint operator in $X$ with compact resolvent and, therefore, $-A$ generates a compact analytic semigroup on $X$. Following Henry \cite{Henry}, $A$ is a sectorial operator in $X$. Now, denote by $X^{\alpha}$, $\alpha > 0$, the fractional power spaces associated with the operator $A$; that is, $X^{\alpha} = D(A^{\alpha})$ endowed with the graph norm. With this notation, we have $X^{-\alpha} = (X^{\alpha})^{\prime}$ for all $\alpha > 0$, see \cite{Amann}.

In this framework, the non-autonomous problem $(\ref{edp01})-(\ref{cond01})$ can be rewritten as an ordinary differential equation in the following abstract form
\begin{equation}\label{edp abstrata} 
\begin{cases}
W_t + \mathcal{A}(t)W = F(W), & \ t > \tau,\\
W(\tau) = W_0, & \ \tau \in \mathbb{R},
\end{cases}
\end{equation}
where $W = W(t)$, for all $t \in \mathbb{R}$, and $W_0 = W(\tau)$ are respectively given by
\begin{equation*}
W = \begin{bmatrix} u\\ u_t\\ v\\ v_t\\ \end{bmatrix} \ \text{and} \ \ W_0 = \begin{bmatrix} u_0\\ u_1\\ v_0\\ v_1\\ \end{bmatrix},
\end{equation*}
and, for each $t \in \mathbb{R}$, the unbounded linear operator $\mathcal{A}(t)\colon  D(\mathcal{A}(t)) \subset Y \to Y$ is defined by
\begin{equation}\label{abstract operator}
\begin{split}
\mathcal{A}(t)\!\! \begin{bmatrix} u\\ v\\ w\\ z\\ \end{bmatrix} \!\!=\!\!
\begin{bmatrix}
0     & -I                           & 0 & 0 \\
A+I & \eta A^{\frac{1}{2}} & 0 & a_{\epsilon}(t)A^{\frac{1}{2}} \\
0     & 0                            & 0 & -I \\
0     & -a_{\epsilon}(t)A^{\frac{1}{2}} & A & \eta A^{\frac{1}{2}}
\end{bmatrix}
\!\!
\begin{bmatrix} u\\ v\\ w\\ z\\ \end{bmatrix} \!\!=
\!\! \begin{bmatrix} -v\\ (A+I)u + \eta A^{\frac{1}{2}}v + a_{\epsilon}(t)A^{\frac{1}{2}}z\\ -z\\ -a_{\epsilon}(t)A^{\frac{1}{2}}v + Aw + \eta A^{\frac{1}{2}}z\\ \end{bmatrix}
\end{split}
\end{equation}
for each $\begin{bmatrix} u & v & w & z\end{bmatrix}^T$ in the domain $D(\mathcal{A}(t))$ defined by the space
\begin{equation}\label{domain abstract operator}
D(\mathcal{A}(t)) = X^1 \times X^{\frac{1}{2}} \times X^1 \times X^{\frac{1}{2}},
\end{equation}
where
$$
Y = Y_0 = X^{\frac{1}{2}} \times X \times X^{\frac{1}{2}} \times X
$$
is the phase space of the problem $(\ref{edp01})-(\ref{cond01})$. The nonlinearity $F$ is given by
\begin{equation}\label{nonlinearity F}
F(W) = \begin{bmatrix} 0\\ f^e(u)\\ 0\\ 0\\ \end{bmatrix},
\end{equation}
where $f^e(u)$ is the Nemitski\u i operator associated with $f(u)$; that is,
$$
f^e(u)(x) = f(u(x)), \quad \text{for all} \quad  x \in \Omega.
$$

Now, we observe that the norms
\[
\|(x,y,z,w)\|_{1} = \|x\|_{X^{\frac{1}{2}}} + \|y\|_{X} + \|z\|_{X^{\frac{1}{2}}} + \|w\|_{X}
\]
and
\[
\|(x,y,z,w)\|_{2} = (\|x\|_{X^{\frac{1}{2}}}^2 + \|y\|_{X}^2 + \|z\|_{X^{\frac{1}{2}}}^2 + \|w\|_{X}^2)^{\frac{1}{2}}
\]
are equivalent in $Y_0$. In this way, we shall use the same notation $\|(x,y,z,w)\|_{Y_0}$ for both norms and the choice will be as convenient.

In the next lines, we describe the main results of this paper. \vspace{.2cm}

Let $\alpha \in (0, 1)$ and consider $1 < \rho < \frac{n + 2(1-\alpha)}{n-2}$. Recall that $Y_{\alpha - 1} = [Y_{-1}, Y_0]_{\alpha} $, where $Y_{-1}$ denotes the extrapolation space of $Y_0$ and $[\cdot, \cdot]_{\alpha}$ denotes the complex interpolation functor, see \cite{Triebel}. Under these conditions, we obtain the following result on well-posedness, which is proved in Section \ref{global}.

\vspace{.2cm}

\begin{theorem}\label{global-sol} \rm \textbf{[Well-Posedness]} \it \,  
Let $f \in C^1(\mathbb{R})$ be a function satisfying \eqref{dissipativeness}-\eqref{Gcondition},
assume conditions \eqref{function a is bounded}-\eqref{hol-a} hold and let $F\colon  Y_0 \to Y_{\alpha - 1}$ be defined in $(\ref{nonlinearity F})$. Then for any initial data $W_0 \in Y_0$ the problem \eqref{edp abstrata} has a unique global solution $W(t)$ such that
\[
W(t) \in C([\tau,\infty),Y_0).
\]
Moreover, such solutions are continuous with respect to the initial data on $Y_0$.
\end{theorem}

The existence of a pullback attractor, see Theorem \ref{teo-pullback} below, is presented in Section \ref{PullA}.

\begin{theorem}\label{teo-pullback}  \rm \textbf{[Pullback Attractors]} \it \, 
Under the conditions of Theorem \ref{global-sol}, the problem $(\ref{edp01})-(\ref{cond01})$ has a pullback attractor $\{ \mathbb{A}(t): t\in\mathbb{R} \}$ in $Y_0$ and
$$
\bigcup\limits_{t \in \mathbb{R}}\mathbb{A}(t) \subset Y_0
$$
is bounded.
\end{theorem}

Theorem \ref{RegPA} deals with the regularity of the pullback attractor  obtained in Theorem \ref{teo-pullback}. This result is proved in Section \ref{RegA}.

\begin{theorem}\label{RegPA}  \rm \textbf{[Regularity of Pullback Attractors]} \it \, 
Assume that $\frac{n-1}{n-2} \leq \rho < \frac{n}{n-2}$. The pullback attractor $\{ \mathbb{A}(t): t\in\mathbb{R} \}$ for the problem $(\ref{edp01})-(\ref{cond01})$, obtained in Theorem \ref{teo-pullback}, lies in a more regular space than $Y_0$. More precisely,
\[
\bigcup\limits_{t \in \mathbb{R}}\mathbb{A}(t)
\]
is a bounded subset of $X^1 \times X^{\frac{1}{2}} \times X^1 \times X^{\frac{1}{2}}$.
\end{theorem}

Lastly, in Section \ref{UpSem}, we show a result on the upper semicontinuity of the pullback attractor, which is stated in 
 Theorem \ref{T-UpSemi}.

\begin{theorem}\label{T-UpSemi}   \rm \textbf{[Upper Semicontinuity]} \it \, 
For each $\eta > 0$ and $\epsilon\in [0, 1],$ let $W^{(\epsilon)} (\cdot) = S_{(\epsilon)} (\cdot, \tau)W_0$ be the solution of $(\ref{edp01})$ in $Y_0$. Assume that $\Vert a_{\epsilon} - a_0 \Vert_{L^{\infty}(\mathbb{R})} \rightarrow 0$ as $\epsilon \rightarrow 0^{+}$. Then, for each $T > 0$, $W^{(\epsilon)}$ converges to $W^{(0)}$ in $C([0, T], Y_0)$ as $\epsilon \rightarrow 0^{+}$. Moreover, the family of pullback attractors $\{ \mathbb{A}_{(\epsilon)}(t): t \in \mathbb{R} \}$ is upper semicontinuous at $\epsilon = 0$.
\end{theorem}

\section{Preliminaries}\label{Preli}

The main purpose of this section is to briefly introduce the reader to some terminology and facts about the theory of abstract parabolic problems, as well as to present definitions and existence results that appear in the study of pullback attractors for non-autonomous dynamical systems. Let $X$ be a Banach space. As a standard notation, we denote by $\mathcal{L}(X)$ the space of all bounded linear operators from $X$ into itself. Let $\{ \mathcal{B}(t): t \in \mathbb{R} \}$ be a family of unbounded closed linear operators, where each $\mathcal{B}(t)$ has the same dense subspace $D$ of $X$ as domain.

\subsection{Non-autonomous abstract linear problem.}
Consider the singularly non-autono\-mous abstract linear parabolic problem of the form
\begin{equation}\label{Ididit}
\begin{cases}
\dfrac{du}{dt} = -\mathcal{B}(t)u, \ t > \tau, \\
u(\tau) = u_0 \in D.
\end{cases}
\end{equation}

The term \textit{singularly non-autonomous} is used to evidence the fact that the unbounded operator $\mathcal{B}(t)$ has explicit dependence with the time. When it comes to the parabolic structure of the above problem, we assume the following conditions:
\begin{enumerate}
\item[$(A1)$] The family of operators $\mathcal{B}(t)\colon  D\subset X \to X$ is \textit{uniformly sectorial} in $X$; that is, $\mathcal{B}(t)$ is closed and densely defined  for every $t \in \mathbb{R}$, with domain $D$ fixed, and for all $T \in \mathbb{R}$ there exists a constant $C_1 > 0$, independent of $T$, such that
$$
\Vert (\mathcal{B}(t) +\lambda I)^{-1} \Vert_{\mathcal{L}(X)} \leq \frac{C_1}{|\lambda| + 1}
$$
for all $\lambda \in \mathbb{C}$ with $\text{Re}(\lambda) \geq 0$ and for all $t \in [-T, T]$. 

\item[$(A2)$] The map $\mathbb{R}\ni t\mapsto \mathcal{B}(t)$ is \textit{uniformly Hölder continuous} in $X$; that is, for all $T \in \mathbb{R}$ there are constants $C_2 > 0$ and $0 < \epsilon_0 \leq 1$, independent of $T$, such that
$$
\Vert [\mathcal{B}(t)-\mathcal{B}(s)]\mathcal{B}^{-1}(\tau) \Vert_{\mathcal{L}(X)} \leq C_2|t-s|^{\epsilon_0}
$$
for every $t, s, \tau \in [-T, T]$.
\end{enumerate}

Denote by $\mathcal{B}_0$ the operator $\mathcal{B}(t_0)$ for some $t_0 \in \mathbb{R}$ fixed. If $X^{\alpha}$ denotes the domain of $\mathcal{B}_0^{\alpha}$, $\alpha > 0$, with the graph norm, and $X^0 = X$, then $\{ X^{\alpha}: \alpha\geq 0\}$ is the fractional power scale associated with $\mathcal{B}_0$. For more details about fractional powers of operators, see Henry \cite{Henry}.

From $(A1)$, $-\mathcal{B}(t)$ is the generator of an analytic semigroup $\{ e^{-\tau\mathcal{B}(t)}: \tau\geq 0\} \subset \mathcal{L}(X)$. Using this and the fact that $0\in \rho(\mathcal{B}(t))$, one can obtain  a constant $C>0$ such that the following estimates hold:
$$
\Vert e^{-\tau\mathcal{B}(t)} \Vert_{\mathcal{L}(X)} \leq C, \ \tau\geq 0, \ t \in \mathbb{R},
$$
and
$$
\Vert \mathcal{B}(t) e^{-\tau\mathcal{B}(t)} \Vert_{\mathcal{L}(X)} \leq C\tau^{-1}, \ \tau > 0, \ t \in \mathbb{R}.
$$

For a given bounded set $I \subset \mathbb{R}^2$, it follows from $(A2)$, that there exists a constant $K = K(I) > 0$ such that
$$
\Vert \mathcal{B}(t)\mathcal{B}^{-1}(\tau) \Vert_{\mathcal{L}(X)} \leq K,
$$
for all $(t, \tau) \in I$.

Also, the semigroup $\{ e^{-\tau\mathcal{B}(t)}: \tau\geq 0\}$ satisfies
$$
\Vert e^{-\tau\mathcal{B}(t)} \Vert_{\mathcal{L}(X^{\beta}, X^{\alpha})} \leq C(\alpha, \beta)\tau^{\beta - \alpha}, \ \tau > 0, \ t \in \mathbb{R},
$$
where $0\leq\beta\leq\alpha < 1 + \epsilon_0$. 

\subsection{Abstract results on pullback attractors.}
In this subsection, we will present basic definitions and results of the theory of pullback attractors for nonlinear evolution processes. For more details we refer to \cite{Rivero2}, \cite{Livro Alexandre} and \cite{Chepyzhov and Vishik}.

Let $(Z, d)$ be a metric space. An \textit{evolution process} in $Z$ is a two-parameter family $\{ S(t, \tau): t\geq\tau \in\mathbb{R} \}$ of  maps from $Z$ into itself such that:
\begin{enumerate}
\item[$(a)$] $S(t, t) = I$ for all $t \in \mathbb{R}$, ($I$ is the identity operator in $Z$),
\item[$(b)$] $S(t, \tau) = S(t, s)S(s, \tau)$ for all $t\geq s\geq\tau$, and
\item[$(c)$] the map $\{ (t, \tau)\in \mathbb{R}^2: t\geq\tau \} \times Z \ni (t, \tau, x) \mapsto S(t, \tau)x \in Z$ is continuous.
\end{enumerate}

If $\{ S(t, \tau): t\geq\tau \in\mathbb{R}\} \subset  \mathcal{L}(X)$ is an evolution process, then we will call this process as a \textit{ linear evolution process.}

\begin{remark} If the operator $\mathcal{B}(t)\colon  D\subset X \to X$ of equation \eqref{Ididit}
 is uniformly sectorial and uniformly Hölder continuous, then there exists a linear evolution process $\{ L(t, \tau): t\geq\tau \in \mathbb{R} \}$ associated with $\mathcal{B}(t)$, which is given by
$$
L(t, \tau) = e^{-(t-\tau) \mathcal{B}(\tau)} + \int_{\tau}^t L(t, s)[\mathcal{B}(\tau) - \mathcal{B}(s)] e^{-(s-\tau) \mathcal{B}(\tau)} ds, \quad t\geq\tau.
$$

The process $\{ L(t, \tau): t\geq\tau \in \mathbb{R} \}$ satisfies the following condition:
$$
\Vert L(t, \tau) \Vert_{\mathcal{L}(X^{\beta}, X^{\alpha})} \leq C(\alpha, \beta)(t-\tau)^{\beta - \alpha},
$$
where $0\leq\beta\leq\alpha < 1 + \epsilon_0$. The reader may consult \cite{Nascimento} and \cite{Sobo} for more details.
\end{remark}

Now, let us consider the following singularly non-autonomous abstract parabolic problem
\begin{equation} \label{naapp}
\begin{cases}
\dfrac{du}{dt} = -\mathcal{B}(t)u + g(u), \ t > \tau, \\
u(\tau) = u_0 \in D,
\end{cases}
\end{equation}
where the operator $\mathcal{B}(t)\colon D\subset X \to X$ is uniformly sectorial and uniformly Hölder continuous, and the nonlinearity $g$ satisfies some suitable conditions that will be specified later. 
 The nonlinear evolution process $\{ S(t, \tau): t\geq\tau \in\mathbb{R} \}$ associated with $\mathcal{B}(t)$ is given by
$$
S(t, \tau) = L(t, \tau) + \int_{\tau}^t L(t, s) g(S(s, \tau)) ds, \quad  t\geq\tau.
$$

\begin{definition}
Let $g\colon  X^{\alpha} \to X^{\beta}$, $\alpha\in [\beta, \beta + 1)$, be a continuous function. A continuous function $u\colon  [\tau, \tau + t_0] \to X^{\alpha}$ is said to be a \textit{local solution} of the problem \eqref{naapp}, starting at $u_0 \in X^{\alpha}$, if the following conditions hold:
\begin{enumerate}
\item[$(a)$] $u \in C([\tau, \tau + t_0], X^{\alpha}) \cap C^1((\tau, \tau + t_0], X^{\alpha})$;
\item[$(b)$] $u(\tau) = u_0$;
\item[$(c)$] $u(t) \in D(\mathcal{B}(t))$ for all $t \in (\tau, \tau + t_0]$;
\item[$(d)$] $u(t)$ satisfies \eqref{naapp} for all $t \in (\tau, \tau + t_0]$.
\end{enumerate}
\end{definition}

Now we state the following abstract local well-posedness result, whose proof can be found in \cite{Rivero1}.  The reader may consult \cite{Nascimento} for a more general version that includes the critical growth case.

\begin{theorem}\label{abstract existence result}
Assume that the family of operators $\{\mathcal{B}(t): t \in \mathbb{R}\}$ is uniformly sectorial and uniformly Hölder continuous in $X^{\beta}$. If $g\colon  X^{\alpha} \to X^{\beta}$, $\alpha\in [\beta, \beta + 1)$, is a Lipschitz continuous map in bounded subsets of $X^{\alpha}$, then given $r>0$ there exists a time $t_0 > 0$ such that for all $u_0 \in B_{X^{\alpha}}(0, r)$ there exists a unique solution of the problem \eqref{naapp} starting in $u_0$ and defined  on $[\tau, \tau + t_0]$. Moreover, such solutions are continuous with respect to the initial data in $B_{X^{\alpha}}(0, r)$.
\end{theorem}

In the sequel, we present the concepts concerning the theory of pullback attractors which we will use in this work.
Further details can be found in \cite{Rivero2}, \cite{Livro Alexandre} and \cite{Chepyzhov and Vishik}. 
  Recall that, if $(Z, d)$ is a metric space, then the Hausdorff semidistance between two nonempty subsets $A$ and $B$ of $Z$ is defined by
$$d_H(A, B) = \sup\limits_{a\in A} \inf\limits_{b\in B} d(a, b).$$

\begin{definition}\label{pullback attraction}
Let $\{ S(t, \tau): t\geq\tau \in\mathbb{R} \}$ be an evolution process in $Z$. Given $t\in\mathbb{R}$ and $A, B$ subsets of $Z$, we say that $A$ \textit{pullback attracts} $B$ at time $t$ if
\begin{equation}\label{atrai}
\lim\limits_{\tau \to -\infty} d_H(S(t, \tau)B, A) = 0,
\end{equation}
where $S(t, \tau)B = \{ S(t, \tau)x : x \in B \}$. The set $A$ pullback attracts bounded sets at time $t$, if $(\ref{atrai})$ holds for every bounded subset $B$ of $Z$. Moreover, we say that a time-dependent family $\{ A(t): t\in\mathbb{R} \}$ of subsets of $Z$ pullback attracts bounded subsets of $Z$, if $A(t)$ pullback attracts bounded sets at time $t$, for each $t\in\mathbb{R}$.
\end{definition}

\begin{definition} 
A family of compact subsets $\{ \mathbb{A}(t): t\in\mathbb{R} \}$ of $Z$ is a \textit{pullback attractor} for the evolution process $\{ S(t, \tau): t\geq\tau \in\mathbb{R} \}$ if
\begin{enumerate}
\item[$(i)$] $\{ \mathbb{A}(t): t\in\mathbb{R} \}$ is invariant; that is, $S(t, \tau) \mathbb{A}(\tau) = \mathbb{A}(t)$ for all $t\geq\tau$,
\item[$(ii)$] $\{ \mathbb{A}(t): t\in\mathbb{R} \}$ pullback attracts bounded subsets of $Z$, in the sense of Definition \ref{pullback attraction}, and
\item[$(iii)$] $\{ \mathbb{A}(t): t\in\mathbb{R} \}$ is the minimal family of closed sets satisfying property $(ii)$.
\end{enumerate}
\end{definition}

\begin{definition}\label{def pullback asymptotically compact}  
An evolution process $\{ S(t, \tau): t\geq\tau \in\mathbb{R} \}$ in $Z$ is said to be \textit{pullback asymptotically compact} if, for each $t\in\mathbb{R}$, each sequence $\{\tau_k\}_{k\in\mathbb{N}}$ with $\tau_k \leq t$ for all $k\in\mathbb{N}$ and $\tau_k \xrightarrow{k\rightarrow +\infty} -\infty$, and each bounded sequence $\{x_k\}_{k\in\mathbb{N}} \subset Z$, then the sequence $\{S(t, \tau_k)x_k\}_{k\in\mathbb{N}}$ has a convergent subsequence.
\end{definition}

\begin{definition} 
We say that a set $B\subset Z$ \textit{pullback absorbs} bounded sets at time $t\in\mathbb{R}$ if, for each bounded subset $D$ of $Z$, there exists a time $T = T(t, D) \leq t$ such that $S(t, \tau)D \subset B$ for all $\tau\leq T$. Moreover, we say that a time-dependent family $\{ B(t): t\in\mathbb{R} \}$ of subsets of $Z$ pullback absorbs bounded sets of $Z$, if $B(t)$ pullback absorbs bounded sets at time $t$, for each $t\in\mathbb{R}$.
\end{definition}

\begin{definition}\label{def pullback strongly bounded dissipative}  
We say that an evolution process $\{ S(t, \tau): t\geq\tau \in\mathbb{R} \}$ in $Z$ is:
\begin{enumerate}
\item[(i)] \textit{pullback strongly bounded} if, for each bounded subset $B$ of $Z$ and each $t\in\mathbb{R}$, then the set $\bigcup\limits_{s \leq t} \gamma_p (B, s)$ is bounded, where $\gamma_p (B, t) = \bigcup\limits_{\tau \leq t} S(t, \tau)B$ 
 is  the \textit{pullback orbit} of $B\subset Z$ at time $t\in\mathbb{R}$. 
\item[(ii)]  \textit{pullback strongly bounded dissipative} if, for each $t\in\mathbb{R}$, then there is a bounded subset $B(t)$ of $Z$ which pullback absorbs bounded subsets of $Z$ at time $s$ for each $s\leq t$; that is, given a bounded subset $D$ of $Z$ and $s\leq t$, there exists $\tau_0(s, D)$ such that $S(s, \tau)D \subset B(t)$ for all $\tau \leq \tau_0(s, D)$.
 \end{enumerate}
\end{definition}

\begin{theorem}\label{existence of the pullback attractor}  
If an evolution process $\{ S(t, \tau): t\geq\tau \in\mathbb{R} \}$ in $Z$ is pullback strongly bounded dissipative and pullback asymptotically compact, then $\{ S(t, \tau): t\geq\tau \in\mathbb{R} \}$ has a pullback attractor $\{ \mathbb{A}(t): t\in\mathbb{R} \}$, such  that $\bigcup\limits_{\tau\leq t} \mathbb{A}(\tau)$ is bounded for each $t\in\mathbb{R}$.
\end{theorem}

\begin{definition}
A \textit{global solution} for an evolution process $\{ S(t, \tau): t\geq\tau \in\mathbb{R} \}$ in $Z$ is a function $\xi: \mathbb{R} \to Z$ such that $S(t, \tau) \xi(\tau) = \xi(t)$ for all $t\geq\tau$.
\end{definition}

It is well-known that if a semigroup has a global attractor, then it is characterized as the union of all bounded global solutions. In the non-autonomous case, an equivalent characterization is given by the following result.

\begin{theorem}\label{globalsolution} If a pullback attractor $\{ \mathbb{A}(t): t\in\mathbb{R} \}$ is bounded in the past then
\[
\mathbb{A}(t) = \{\xi(t): \, \xi(\cdot) \; \text{is a backward-bounded global solution}\}.
\]
\end{theorem}

\medskip

We end this section with the concept of upper semicontinuity.

\begin{definition}
A family $\{\mathbb{A}_\epsilon(t)\}_{\epsilon \in [0, 1]}$, $t \in \mathbb{R}$, of subsets of $Z$ is upper semicontinuous at $\epsilon=0$ if, for each $t \in \mathbb{R}$, then
\[
\lim_{\epsilon\to 0^+} d_H (\mathbb{A}_\epsilon(t),\mathbb{A}_0(t))=0.
\]
\end{definition}

\section{Local and global well-posedness}\label{global}
This section concerns the investigation of the existence of global solutions for \eqref{edp abstrata}. We are going to present 
auxiliary results to conclude Theorem \ref{global-sol}. We start by obtaining some spectral properties for the unbounded linear operator $\mathcal{A}(t)$ given in $(\ref{abstract operator})$ and $(\ref{domain abstract operator})$.

It is not difficult to see that $\det(\mathcal{A}(t)) = A(A+I)$, and therefore that $0 \in \rho(\mathcal{A}(t))$, for all $t \in\mathbb{R}$. Moreover, for each $t \in \mathbb{R}$, the operator $\mathcal{A}^{-1}(t)\colon  Y_0 \to Y_0$ is defined by
\begin{equation}\label{inverse of the abstract operator}
\begin{split}
\mathcal{A}^{-1}(t) \begin{bmatrix} u\\ v\\ w\\ z\\ \end{bmatrix} &=
\begin{bmatrix}
\eta A^{\frac{1}{2}}(A+I)^{-1} & (A+I)^{-1} & a_{\epsilon}(t)A^{\frac{1}{2}}(A+I)^{-1} & 0 \\
-I                                           & 0               & 0                                          & 0 \\
-a_{\epsilon}(t)A^{-\frac{1}{2}}               & 0               & \eta A^{-\frac{1}{2}}            & A^{-1} \\
0                                            & 0               & -I                                         & 0
\end{bmatrix}
\begin{bmatrix} 
u\\ v\\ w\\ z
\end{bmatrix}.
\end{split}
\end{equation}

\begin{proposition}\label{acretivo} 
For each fixed $t \in \mathbb{R}$, the operator $\mathcal{A}(t)$ defined in $(\ref{abstract operator})-(\ref{domain abstract operator})$ is maximal accretive.
\end{proposition}

\begin{proof}  The proof is analogous to the proof of \cite[Proposition 4.3]{BCNS2}. We include here only the proof of accretivity of $\mathcal{A}(t)$. Let $t \in \mathbb{R}$ be fixed but arbitrary, and let $x =\begin{bmatrix} u & v & w & z\end{bmatrix}^T \in D(\mathcal{A}(t))$. At first, we note that $ \langle v, u\rangle_{X^{\frac{1}{2}}} =  \left\langle (A+I)^{\frac{1}{2}}v, (A+I)^{\frac{1}{2}}u \right\rangle_X$, because from \cite[Corollary 1.3.5]{Cholewa}, we have $D((A+I)^{\frac{1}{2}}) = D(A^{\frac{1}{2}})$. Thus,
\[
\begin{split}
\langle \mathcal{A}(t)x, x\rangle_{Y_0} &=  \langle -v, u\rangle_{X^{\frac{1}{2}}} + \langle (A+I)u + \eta A^{\frac{1}{2}}v + a_{\epsilon}(t)A^{\frac{1}{2}}z, v\rangle_X + \langle -z, w\rangle_{X^{\frac{1}{2}}} \\
& \quad + \langle -a_{\epsilon}(t)A^{\frac{1}{2}}v + Aw + \eta A^{\frac{1}{2}}z, z\rangle_X \\
& = \left\langle (A+I)^{\frac{1}{2}}u, (A+I)^{\frac{1}{2}}v \right\rangle_X - \left\langle (A+I)^{\frac{1}{2}}v, (A+I)^{\frac{1}{2}}u \right\rangle_X \\
& \quad + a_{\epsilon}(t)\left(\langle A^{\frac{1}{2}}z, v\rangle_X - \langle v, A^{\frac{1}{2}}z\rangle_X\right) \\
& \quad + \langle Aw, z\rangle_X - \langle z, Aw\rangle_X + \eta\Vert A^{\frac{1}{4}}v \Vert_X^2 + \eta\Vert A^{\frac{1}{4}}z \Vert_X^2.
\end{split}
\]

Hence, 
\begin{equation}\label{real part}
\text{Re}(\langle \mathcal{A}(t)x, x\rangle_{Y_0}) = \eta\Vert A^{\frac{1}{4}}v \Vert_X^2 + \eta\Vert A^{\frac{1}{4}}z \Vert_X^2 \geq 0,
\end{equation}
which proves the accretivity of $\mathcal{A}(t)$.
\end{proof}

\begin{proposition}
If $Y_{-1}$ denotes the extrapolation space of $Y_0 = X^{\frac{1}{2}} \times X \times X^{\frac{1}{2}} \times X$ generated by the operator $\mathcal{A}^{-1}(t)$, then
\[
Y_{-1} = X \times X^{-\frac{1}{2}} \times X \times X^{-\frac{1}{2}}.
\]
\end{proposition}

\begin{proof}
This proof is analogous to the proof of \cite[Proposition 2]{BCNS1} and so we omit it. See also \cite{BCNS2}.
\end{proof}

\begin{proposition}
The operator $\mathcal{A}^{-1}(t)$ given in $(\ref{inverse of the abstract operator})$ is a compact map for each $t \in \mathbb{R}$.
\end{proposition}

\begin{proof}
Let $B \subset Y_0$ be a bounded set and denote $Y_1 = X^1 \times X^{\frac{1}{2}} \times X^1 \times X^{\frac{1}{2}}$.
At first, note that since $(A+I)(A+I)^{-1} = I$ and $A$ is uniformly sectorial, we have
\begin{equation*}
\Vert A(A+I)^{-1} \Vert_{\mathcal{L}(X)} \leq 1 + \Vert (A+I)^{-1} \Vert_{\mathcal{L}(X)} \leq 1+M,
\end{equation*}
for some constant $M>0$. Thus,  for $x =\begin{bmatrix} u & v & w & z\end{bmatrix}^T \in B$, we have
\[
\begin{split}
&\Vert \mathcal{A}^{-1}(t)x \Vert_{Y_1} \\
& = \Vert \eta A^{\frac{1}{2}}(A+I)^{-1}u + (A+I)^{-1}v + a_{\epsilon}(t)A^{\frac{1}{2}}(A+I)^{-1}w \Vert_{X^1} + \Vert -u\Vert_{X^{\frac{1}{2}}} \\
& + \Vert -a_{\epsilon}(t)A^{-\frac{1}{2}}u + \eta A^{-\frac{1}{2}}w + A^{-1}z \Vert_{X^1} + \Vert -w\Vert_{X^{\frac{1}{2}}} \\
& \leq \eta\Vert A(A+I)^{-1}\Vert_{\mathcal{L}(X)}\Vert A^{\frac{1}{2}}u\Vert_X + \Vert A(A+I)^{-1}\Vert_{\mathcal{L}(X)}\Vert v\Vert_X \\
& + a_1\Vert A(A+I)^{-1}\Vert_{\mathcal{L}(X)}\Vert A^{\frac{1}{2}}w\Vert_X + (1 + a_1)\Vert u\Vert_{X^{\frac{1}{2}}} + (1 + \eta)\Vert w\Vert_{X^{\frac{1}{2}}} + \Vert z\Vert_X \\
& \leq [\eta(1 + M) + 1 + a_1]\Vert u\Vert_{X^{\frac{1}{2}}} + (1 + M)\Vert v\Vert_X + [a_1(1 + M) + 1 + \eta]\Vert w\Vert_{X^{\frac{1}{2}}} + \Vert z\Vert_X \\
& \leq C\left(\Vert u\Vert_{X^{\frac{1}{2}}} + \Vert v\Vert_X + \Vert w\Vert_{X^{\frac{1}{2}}} + \Vert z\Vert_X\right),
\end{split}
\]
where $C$ is a positive constant, that is,
$$
\Vert \mathcal{A}^{-1}(t)x \Vert_{Y_1} \leq C\Vert x\Vert_{Y_0}.
$$

Thus,  $\mathcal{A}^{-1}(t)B$ is bounded in $Y_1$. Using the compact embedding $Y_1 \hookrightarrow Y_0$, 
we conclude that the operator $\mathcal{A}^{-1}(t)$ is compact.
\end{proof}

\begin{proposition}
The family of operators $\{\mathcal{A}(t): t \in \mathbb{R}\}$, defined in $(\ref{abstract operator})-(\ref{domain abstract operator})$, is uniformly Hölder continuous in $Y_{-1}$.
\end{proposition}

\begin{proof}
Using \eqref{hol-a}, this result follows immediately from $(\ref{abstract operator})$ and $(\ref{domain abstract operator})$.
\end{proof}

The next step is to show the analyticity of the semigroup $\{ e^{-\tau\mathcal{A}(t)}: \tau\geq 0\}$. For that, we will make use of the following result whose  proof can be found in \cite{Liu-Zheng}.

\begin{theorem}\label{analiticidade}
Let $\{ T(\tau)\colon  \tau \geq 0 \}$ be a $C_0$-semigroup of contractions in a Hilbert space $H$ with infinitesimal generator $\mathcal{B}$. Suppose that $i\mathbb{R}\subset \rho(\mathcal{B})$. Then $\{ T(\tau)\colon  \tau \geq 0 \}$ is analytic if, and only if
\[
\limsup\limits_{|\beta| \rightarrow \infty} \Vert \beta(i\beta I - \mathcal{B})^{-1} \Vert_{\mathcal{L} (H)} < \infty.
\]
\end{theorem}

The next lemma shows that $i\mathbb{R} \subset \rho(-\mathcal{A}(t))$ for all $t \in \mathbb{R}$. 

\begin{lemma}\label{An1}
The semigroup $\{ e^{-\tau\mathcal{A}(t)} \colon \tau\geq 0\}$, generated by $-\mathcal{A}(t)$, satisfies
\[
i\mathbb{R} \subset \rho(-\mathcal{A}(t))
\]
for all $t \in \mathbb{R}$.
\end{lemma}

\begin{proof}
Arguing by contradiction, suppose that there exists $0\neq\beta\in\mathbb{R}$ such that $i\beta$ is in the spectrum of $-\mathcal{A}(t)$ for some $t \in \mathbb{R}$. Then $i\beta$ must be an eigenvalue of $-\mathcal{A}(t)$, since the operator $\mathcal{A}^{-1}(t)$ is compact. Consequently, there exists
$$
U = \begin{bmatrix} u & v & w & z\end{bmatrix}^T \in D(\mathcal{A}(t)), \ \Vert U\Vert_{Y_0} = 1,
$$
such that $i\beta U - (-\mathcal{A}(t))U = 0$ or, equivalently,
\[
\begin{array}{r}
i\beta u - v = 0,  \\
i\beta v + Au + u + \eta A^{\frac{1}{2}}v + a_{\epsilon}(t)A^{\frac{1}{2}}z = 0,  \\
i\beta w - z = 0,  \\
i\beta z - a_{\epsilon}(t)A^{\frac{1}{2}}v + Aw + \eta A^{\frac{1}{2}}z = 0.
\end{array}
\]

Now, taking the real part of the inner product of $i\beta U + \mathcal{A}(t)U$ with $U$ in $Y_0$, we have
\[
\begin{split}
\langle i\beta U + \mathcal{A}(t)U, U \rangle_{Y_0} = \langle 0, U \rangle_{Y_0} = 0 &\implies i\beta\Vert U\Vert_{Y_0}^2 + \langle\mathcal{A}(t)U, U\rangle_{Y_0} = 0 \\
&\implies \text{Re}(\langle\mathcal{A}(t)U, U\rangle_{Y_0}) = 0 \\
&\implies \eta\Vert A^{\frac{1}{4}}v \Vert_X^2 + \eta\Vert A^{\frac{1}{4}}z \Vert_X^2 = 0 \\
&\implies \Vert A^{\frac{1}{4}}v \Vert_X^2 = \Vert A^{\frac{1}{4}}z \Vert_X^2 = 0 \\
&\implies v = z = 0.
\end{split}
\]

Consequently, $u = w = 0$. Therefore, $U = 0$, which is a contradiction. This proves our claim.
\end{proof}

Now, we are in position to prove that the semigroup generated by $-\mathcal{A}(t)$ is analytic. 

\begin{theorem}
The semigroup $\{ e^{-\tau\mathcal{A}(t)}: \tau\geq 0\}$, generated by $-\mathcal{A}(t)$, is analytic for each $t \in \mathbb{R}$.
\end{theorem}

\begin{proof} We are going to use Theorem \ref{analiticidade}. Let $t \in \mathbb{R}$. In view of Lemma \ref{An1}, it is enough to prove that there exists a positive constant $C$ such that
\[
|\beta| \Vert U\Vert_{Y_0} \leq C \Vert F\Vert_{Y_0},
\]
for all $F \in Y_0$ and all $\beta \in \mathbb{R}$, where
\[
U = (i\beta I + \mathcal{A}(t))^{-1}F \in D(\mathcal{A}(t)).
\]

In fact, denoting $U = \begin{bmatrix} u & v & w & z\end{bmatrix}^T$ and $F = \begin{bmatrix} f & g & h & k\end{bmatrix}^T$, we can write the resolvent equation
\begin{equation}\label{resolvent equation}
(i\beta I + \mathcal{A}(t))U = F
\end{equation}
in $Y_0$ in terms of its components, obtaining the following scalar equations
\[
i\beta u - v = f,
\]
\begin{equation}\label{eq aux01} 
Au + u + i\beta v + \eta A^{\frac{1}{2}}v + a_{\epsilon}(t)A^{\frac{1}{2}}z = g,
\end{equation}
\begin{equation} \label{eq aux02}
i\beta w - z = h,
\end{equation}
\[
- a_{\epsilon}(t)A^{\frac{1}{2}}v + Aw + i\beta z + \eta A^{\frac{1}{2}}z = k.
\]

Taking the inner product of $(\ref{resolvent equation})$ with $U$ in $Y_0$, we obtain
\begin{equation}\label{resolvent equation2}
i\beta\Vert U\Vert_{Y_0}^2 + \langle\mathcal{A}(t)U, U\rangle_{Y_0} = \langle F, U\rangle_{Y_0}.
\end{equation}

By the proof of Proposition \ref{acretivo}, see $(\ref{real part})$, we get
\[
\text{Re}(\langle \mathcal{A}(t)U, U\rangle_{Y_0}) = \eta\Vert A^{\frac{1}{4}}v \Vert_X^2 + \eta\Vert A^{\frac{1}{4}}z \Vert_X^2 \geq 0.
\]

It follows by the Cauchy-Schwartz inequality that
\[
\begin{split}
\eta\Vert A^{\frac{1}{4}}v \Vert_X^2 + \eta\Vert A^{\frac{1}{4}}z \Vert_X^2 &= |\text{Re}(\langle\mathcal{A}(t)U, U\rangle_{Y_0})| = |\text{Re}(\langle F, U\rangle_{Y_0})| \leq |\langle F, U\rangle_{Y_0}| \leq \Vert F\Vert_{Y_0}\Vert U\Vert_{Y_0}
\end{split}
\]
and, therefore, we obtain
\begin{equation}\label{normas de v e z}
\Vert A^{\frac{1}{4}}v \Vert_X^2 \leq \frac{1}{\eta}\Vert F\Vert_{Y_0}\Vert U\Vert_{Y_0} \ \ \text{and} \ \ \Vert A^{\frac{1}{4}}z \Vert_X^2 \leq \frac{1}{\eta}\Vert F\Vert_{Y_0}\Vert U\Vert_{Y_0}.
\end{equation}

Now, taking the inner product of $(\ref{resolvent equation})$ with $x_1 = \begin{bmatrix} A^{-\frac{1}{2}}v & 0 & 0 & 0\end{bmatrix}^T$ in $Y_0$, it leads to
\[
\begin{split}
\langle (i\beta I + \mathcal{A}(t))U, x_1\rangle_{Y_0} = \langle F, x_1\rangle_{Y_0} &\iff \langle i\beta u - v, A^{-\frac{1}{2}}v\rangle_{X^{\frac{1}{2}}} = \langle f, A^{-\frac{1}{2}}v\rangle_{X^{\frac{1}{2}}} \\
&\iff \langle A^{\frac{1}{2}}u, -i\beta v\rangle_X - \Vert A^{\frac{1}{4}}v\Vert_X^2 = \langle A^{\frac{1}{2}}f, v\rangle_X
\end{split}
\]
and then, using $(\ref{eq aux01})$, we conclude that
\[
\langle A^{\frac{1}{2}}u, Au + u + \eta A^{\frac{1}{2}}v + a_{\epsilon}(t)A^{\frac{1}{2}}z - g\rangle_X - \Vert A^{\frac{1}{4}}v\Vert_X^2 = \langle A^{\frac{1}{2}}f, v\rangle_X.
\]

Thus, from Cauchy-Schwartz and Young inequalities and $(\ref{normas de v e z})$, we obtain
\[
\begin{split}
&\Vert A^{\frac{3}{4}}u\Vert_X^2 \\
&= -\Vert A^{\frac{1}{4}}u\Vert_X^2 - \eta\langle A^{\frac{3}{4}}u, A^{\frac{1}{4}}v\rangle_X - a_{\epsilon}(t)\langle A^{\frac{3}{4}}u, A^{\frac{1}{4}}z\rangle_X + \langle A^{\frac{1}{2}}u, g\rangle_X + \langle A^{\frac{1}{2}}f, v\rangle_X + \Vert A^{\frac{1}{4}}v\Vert_X^2 \\
&\leq \eta\Vert A^{\frac{3}{4}}u\Vert_X\Vert A^{\frac{1}{4}}v\Vert_X + a_1\Vert A^{\frac{3}{4}}u\Vert_X\Vert A^{\frac{1}{4}}z\Vert_X + \Vert A^{\frac{1}{2}}u\Vert_X\Vert g\Vert_X + \Vert A^{\frac{1}{2}}f\Vert_X\Vert v\Vert_X + \Vert A^{\frac{1}{4}}v\Vert_X^2 \\
&\leq \frac{\epsilon_1}{2}\eta^2\Vert A^{\frac{3}{4}}u\Vert_X^2 + \frac{1}{2\epsilon_1}\Vert A^{\frac{1}{4}}v\Vert_X^2 + \frac{\epsilon_2}{2}a_1^2\Vert A^{\frac{3}{4}}u\Vert_X^2 + \frac{1}{2\epsilon_2}\Vert A^{\frac{1}{4}}z\Vert_X^2 + \left(\frac{1}{\eta} + 2\right)\Vert F\Vert_{Y_0}\Vert U\Vert_{Y_0} \\
&\leq \left(\frac{\epsilon_1}{2}\eta^2 + \frac{\epsilon_2}{2}a_1^2\right)\Vert A^{\frac{3}{4}}u\Vert_X^2 + \left(\frac{1}{2\eta\epsilon_1} + \frac{1}{2\eta\epsilon_2} + \frac{1}{\eta} + 2\right)\Vert F\Vert_{Y_0}\Vert U\Vert_{Y_0},
\end{split}
\]
for all $\epsilon_1 > 0$ and $\epsilon_2 > 0$. Now, it is enough to choose $\epsilon_1 = \frac{1}{2\eta^2}$ and $\epsilon_2 = \frac{1}{2a_1^2}$, and so we get
\begin{equation}\label{norma de u}
\Vert A^{\frac{3}{4}}u\Vert_X^2 \leq \left(2\eta + \frac{2(a_1^2 + 1)}{\eta} + 4\right)\Vert F\Vert_{Y_0}\Vert U\Vert_{Y_0}.
\end{equation}

Next, taking the inner product of $(\ref{resolvent equation})$ with $x_2 = \begin{bmatrix} 0 & 0 & 0 & A^{\frac{1}{2}}w\end{bmatrix}^T$, we have
\[
\langle (i\beta I + \mathcal{A}(t))U, x_2\rangle_{Y_0} = \langle F, x_2\rangle_{Y_0} \iff \langle - a_{\epsilon}(t)A^{\frac{1}{2}}v + Aw + i\beta z + \eta A^{\frac{1}{2}}z, A^{\frac{1}{2}}w\rangle_X = \langle k, A^{\frac{1}{2}}w\rangle_X.
\]

That is,
\[
- a_{\epsilon}(t)\langle A^{\frac{1}{2}}v, A^{\frac{1}{2}}w\rangle_X + \Vert A^{\frac{3}{4}}w\Vert_X^2 + \langle A^{\frac{1}{2}}z, -i\beta w\rangle_X + \eta\langle A^{\frac{1}{2}}z, A^{\frac{1}{2}}w\rangle_X = \langle k, A^{\frac{1}{2}}w\rangle_X
\]
and then, using $(\ref{eq aux02})$, we have
\[
\Vert A^{\frac{3}{4}}w\Vert_X^2 = a_{\epsilon}(t)\langle A^{\frac{1}{4}}v, A^{\frac{3}{4}}w\rangle_X - \eta\langle A^{\frac{1}{4}}z, A^{\frac{3}{4}}w\rangle_X + \Vert A^{\frac{1}{4}}z\Vert_X^2 + \langle A^{\frac{1}{2}}z, h\rangle_X + \langle k, A^{\frac{1}{2}}w\rangle_X.
\]

Using again the Cauchy-Schwartz and Young inequalities, and $(\ref{normas de v e z})$, we obtain
\[
\begin{split}
&\Vert A^{\frac{3}{4}}w\Vert_X^2 \\
&\leq a_1\Vert A^{\frac{1}{4}}v\Vert_X\Vert A^{\frac{3}{4}}w\Vert_X + \eta\Vert A^{\frac{1}{4}}z\Vert_X\Vert A^{\frac{3}{4}}w\Vert_X + \Vert A^{\frac{1}{4}}z\Vert_X^2 + \Vert z\Vert_X\Vert A^{\frac{1}{2}}h\Vert_X + \Vert k\Vert_X\Vert A^{\frac{1}{2}}w\Vert_X \\
&\leq \frac{\epsilon_3}{2}a_1^2\Vert A^{\frac{3}{4}}w\Vert_X^2 + \frac{1}{2\epsilon_3}\Vert A^{\frac{1}{4}}v\Vert_X^2 + \frac{\epsilon_4}{2}\eta^2\Vert A^{\frac{3}{4}}w\Vert_X^2 + \frac{1}{2\epsilon_4}\Vert A^{\frac{1}{4}}z\Vert_X^2 + \left(\frac{1}{\eta} + 2\right)\Vert F\Vert_{Y_0}\Vert U\Vert_{Y_0} \\
&\leq \left(\frac{\epsilon_3}{2}a_1^2 + \frac{\epsilon_4}{2}\eta^2\right)\Vert A^{\frac{3}{4}}w\Vert_X^2 + \left(\frac{1}{2\eta\epsilon_3} + \frac{1}{2\eta\epsilon_4} + \frac{1}{\eta} + 2\right)\Vert F\Vert_{Y_0}\Vert U\Vert_{Y_0},
\end{split}
\]
for all $\epsilon_3 > 0$ and $\epsilon_4 > 0$. Choosing $\epsilon_3 = \frac{1}{2a_1^2}$ and $\epsilon_4 = \frac{1}{2\eta^2}$, we get
\begin{equation}\label{norma de w}
\Vert A^{\frac{3}{4}}w\Vert_X^2 \leq \left(2\eta + \frac{2(a_1^2 + 1)}{\eta} + 4\right)\Vert F\Vert_{Y_0}\Vert U\Vert_{Y_0}.
\end{equation}

By \cite[Corollary 1.3.5]{Cholewa}, we have $D((A+I)^{\frac{1}{2}}) = D(A^{\frac{1}{2}})$, consequently,
\[
\langle A^{\frac{1}{2}}u, A^{\frac{1}{2}} v\rangle_{X} =   \langle u, v\rangle_{X^{\frac{1}{2}}}  = \left\langle (A+I)^{\frac{1}{2}}u, (A+I)^{\frac{1}{2}}v \right\rangle_X.
\]

Using this fact and the proof of Proposition \ref{acretivo}, we obtain
\[
\begin{split}
\langle \mathcal{A}(t)U, U\rangle_{Y_0} &= \left\langle A^{\frac{1}{2}}u, A^{\frac{1}{2}}v \right\rangle_X - \left\langle A^{\frac{1}{2}}v, A^{\frac{1}{2}}u \right\rangle_X \\
& + a_{\epsilon}(t)\left(\langle A^{\frac{1}{2}}z, v\rangle_X - \langle v, A^{\frac{1}{2}}z\rangle_X\right) \\
& + \langle Aw, z\rangle_X - \langle z, Aw\rangle_X + \eta\Vert A^{\frac{1}{4}}v \Vert_X^2 + \eta\Vert A^{\frac{1}{4}}z \Vert_X^2,
\end{split}
\]
and, taking the imaginary part, we have
\[
\begin{split}
\text{Im}(\langle \mathcal{A}(t)U, U\rangle_{Y_0}) & =  2\text{Im}(\langle A^{\frac{1}{2}}u, A^{\frac{1}{2}}v\rangle_X) + 2a_{\epsilon}(t)\text{Im}(\langle A^{\frac{1}{4}}z, A^{\frac{1}{4}}v\rangle_X) \\
&+ 2\text{Im}(\langle A^{\frac{3}{4}}w, A^{\frac{1}{4}}z\rangle_X)\\
&  =2\text{Im}(\langle A^{\frac{3}{4}}u, A^{\frac{1}{4}}v\rangle_X) + 2a_{\epsilon}(t)\text{Im}(\langle A^{\frac{1}{4}}z, A^{\frac{1}{4}}v\rangle_X) \\
&+ 2\text{Im}(\langle A^{\frac{3}{4}}w, A^{\frac{1}{4}}z\rangle_X).
\end{split}
\]

With this last equality and taking the imaginary part in $(\ref{resolvent equation2})$, it follows by the Cauchy-Schwartz and Young inequalities that
\[
\begin{split}
&\beta\Vert U\Vert_{Y_0}^2  = \text{Im}(\langle F, U\rangle_{Y_0}) - \text{Im}(\langle\mathcal{A}(t)U, U\rangle_{Y_0}) \\
& \leq \Vert F\Vert_{Y_0}\Vert U\Vert_{Y_0} + 2\Vert A^{\frac{3}{4}}u\Vert_X\Vert A^{\frac{1}{4}}v\Vert_X + 2a_1\Vert A^{\frac{1}{4}}z\Vert_X\Vert A^{\frac{1}{4}}v\Vert_X + 2\Vert A^{\frac{3}{4}}w\Vert_X\Vert A^{\frac{1}{4}}z\Vert_X \\
& \leq \Vert F\Vert_{Y_0}\Vert U\Vert_{Y_0} + \Vert A^{\frac{3}{4}}u\Vert_X^2 + (1 + a_1)\Vert A^{\frac{1}{4}}v\Vert_X^2 + \Vert A^{\frac{3}{4}}w\Vert_X^2 + (a_1 +1)\Vert A^{\frac{1}{4}}z\Vert_X^2
\end{split}
\]
and, using the estimates obtained in $(\ref{normas de v e z})$, $(\ref{norma de u})$ and $(\ref{norma de w})$, we get
\[
\beta\Vert U\Vert_{Y_0}^2 \leq \left( 1 + 2\left(2\eta + \frac{2(a_1^2 + 1)}{\eta} + 4\right) + \frac{2a_1 + 2}{\eta} \right)\Vert F\Vert_{Y_0}\Vert U\Vert_{Y_0},
\]
that is, there exists a positive constant $C$, independent of $\beta$, such that
\[
\beta\Vert (i\beta I + \mathcal{A}(t))^{-1}F \Vert_{Y_0} \leq C\Vert F\Vert_{Y_0}
\]
for all $F \in Y_0$ and all $\beta \in \mathbb{R}$. Since this holds for $\beta \in \mathbb{R}$ arbitrary,
\[
|\beta|\Vert (i\beta I + \mathcal{A}(t))^{-1} \Vert_{\mathcal{L}(Y_0)} \leq C, \quad \text{for all} \quad \beta \in \mathbb{R},
\]
and, therefore, we conclude that
\[
\limsup\limits_{|\beta| \rightarrow +\infty} \Vert\beta (i\beta I + \mathcal{A}(t))^{-1} \Vert_{\mathcal{L}(Y_0)} < \infty.
\]

By Theorem \ref{analiticidade}, the semigroup $\{ e^{-\tau\mathcal{A}(t)}: \tau\geq 0\}$ is analytic.
\end{proof}

\begin{remark}
We have the following description of the fractional power scale for the operator $\mathcal{A}(t)$, given as follows
\[
Y_0 \hookrightarrow Y_{\alpha - 1} \hookrightarrow Y_{-1}, \quad \text{for all} \quad 0 < \alpha < 1,
\]
where
\[
\begin{split}
Y_{\alpha - 1} &= [Y_{-1}, Y_0]_{\alpha} = [ X \times X^{-\frac{1}{2}} \times X \times X^{-\frac{1}{2}}, X^{\frac{1}{2}} \times X \times X^{\frac{1}{2}} \times X ]_{\alpha} \\
&= [ X, X^{\frac{1}{2}} ]_{\alpha} \times [ X^{-\frac{1}{2}}, X ]_{\alpha} \times [ X, X^{\frac{1}{2}} ]_{\alpha} \times [ X^{-\frac{1}{2}}, X ]_{\alpha} \\
&= X^{\frac{\alpha}{2}} \times X^{\frac{\alpha - 1}{2}} \times X^{\frac{\alpha}{2}} \times X^{\frac{\alpha - 1}{2}},
\end{split}
\]
where $[\cdot, \cdot]_{\alpha}$ denotes the complex interpolation functor, see \cite{Triebel}. The first equality follows from Proposition \ref{acretivo} (recall that $0\in \rho(\mathcal{A}(t))$), see \cite[Example 4.7.3 (b)]{Amann} and the others equalities follow from \cite[Proposition 2]{CC}.
\end{remark}

Proposition \ref{F is Lipschitz} gives us sufficient conditions  for $F\colon  Y_0 \to Y_{\alpha - 1}$ to be Lipschitz continuous in bounded subsets of $Y_0$. For a proof, the reader may consult \cite[Corollary 2.7]{Carbone_M_K_R} and \cite[Corollary 3.6]{BCNS1}.

\begin{proposition}\label{F is Lipschitz}
Assume that $1 < \rho < \frac{n + 2(1-\alpha)}{n-2}$, with $\alpha \in (0, 1)$. Then the map $F\colon  Y_0 \to Y_{\alpha - 1}$, defined in $(\ref{nonlinearity F})$, is Lipschitz continuous in bounded subsets of $Y_0$.
\end{proposition}

Proposition \ref{F is Lipschitz} and Theorem \ref{abstract existence result} ensure the local well-posedness of \eqref{edp abstrata} in the phase space $Y_0$, and this allows us to establish the following existence result.

\begin{corollary}\label{existence of solutions}
Let $1 < \rho < \frac{n + 2(1-\alpha)}{n-2}$, with $\alpha \in (0, 1)$, 
 $f \in C^1(\mathbb{R})$ be a function satisfying \eqref{dissipativeness}-\eqref{Gcondition},
assume conditions \eqref{function a is bounded}-\eqref{hol-a} hold and let $F\colon  Y_0 \to Y_{\alpha - 1}$ be defined in $(\ref{nonlinearity F})$.
Then given $r>0$, there exists a time $t_0 = t_0(r) > 0$ such that for all $W_0 \in B_{Y_0}(0, r)$, there exists a unique solution $W\colon  [\tau, \tau + t_0] \to Y_0$ of the problem \eqref{edp abstrata} starting in $W_0$. Moreover, such solutions are continuous with respect to the initial data in $B_{Y_0}(0, r)$.
\end{corollary}

Before to present the proof of Theorem \ref{global-sol}, we give an auxiliary result. 

\begin{lemma}\label{Lem_Aux_Int}
Let $f \in C^1(\mathbb{R})$ be a function satisfying \eqref{dissipativeness}-\eqref{Gcondition}. The following conditions hold:
\begin{enumerate}
\item[$(i)$] There is a  constant $k >0$ such that  $|f(s)| \leq |f(0)| + k(|s| + |s|^{\rho})$ for all $s \in \mathbb{R}$. Consequently, $|f(s)| \leq c(1 + |s|^{\rho})$
for all $s \in \mathbb{R}$ and some $c > 0$.

\item[$(ii)$] Given $\delta > 0$, there exists a constant $C_{\delta} > 0$ such that
\[
\int_{\Omega} f(u)u dx \leq C_{\delta} + \delta\Vert u\Vert_X^2
\quad \text{and} \quad \int_{\Omega}\int_0^u f(s)dsdx \leq C_{\delta} + \delta\Vert u\Vert_X^2,
\]
for all $u \in X$.

\item[$(iii)$] Given $r> 0$, there exist constants $C_{r} > 0$ and $C\geq 0$ $($which does not depend on $r)$ such that
\[
\left| \int_{\Omega} f(u)u dx \right| \leq C_{r}\Vert u\Vert_{X^{\frac{1}{2}}}^2 + C \;\; \text{and} \;\;
\left| \int_{\Omega}\int_0^u f(s)ds dx \right| \leq C_{r}\Vert u\Vert_{X^{\frac{1}{2}}}^2 + C
\]
for all $u \in X^{\frac{1}{2}}$ with $\|u\|_{X^{\frac{1}{2}}} \leq r$. If $f(0) = 0$ then the constant $C$ can be chosen zero.
\end{enumerate}
\end{lemma}

\begin{proof}
Condition $(i)$ is a consequence of \eqref{Gcondition} and \cite[Lemma 2.4]{Carbone_M_K_R}. Condition $(ii)$ follows by the ideas presented in \cite{Hale} (see page 76).

Let us prove condition $(iii)$. Let $u \in X^{\frac{1}{2}}$. Using the H\"older's inequality, the Poincar\'e inequality $\Vert u\Vert_X^2 \leq \lambda_1^{-1} \Vert u\Vert_{X^{\frac{1}{2}}}^2$ $(\lambda_1 > 0$ is the first eigenvalue of the negative Laplacian operator with homogeneous Dirichlet boundary condition$)$ and item $(i)$, we have
\begin{equation*}
\begin{split}
\left| \int_{\Omega} f(u)udx \right| & \leq \left( \int_{\Omega}|u|^2 dx\right)^{\frac{1}{2}} \left( \int_{\Omega}|f(u)|^2 dx\right)^{\frac{1}{2}} \\
& \leq  \kappa_0\Vert u\Vert_X \left( \int_{\Omega}\left(|f(0)|^2 + |u|^2 + |u|^{2\rho}\right) dx\right)^{\frac{1}{2}} \\
& \leq   \kappa_{1}\Vert u\Vert_{X}\left(|f(0)||\Omega|^{\frac{1}{2}} + \|u\|_X + \|u\|_{L^{2\rho}(\Omega)}^{\rho} \right) \\
& \leq         \frac{\kappa_{1}}{2}\Vert u\Vert_{X}^2 + \frac{\kappa_{1}}{2}\left(|f(0)||\Omega|^{\frac{1}{2}} + \|u\|_X + \|u\|_{L^{2\rho}(\Omega)}^{\rho} \right)^2 \\
&\leq  \kappa_{2}\left(|f(0)|^2|\Omega| + \|u\|_X^2 +  \Vert u \Vert_{L^{2\rho}(\Omega)}^{2\rho}  \right)\\
& \leq \kappa_{2}\left(|f(0)|^2|\Omega| + \lambda_1^{-1}\|u\|_{X^{\frac{1}{2}}}^2 +  \Vert u \Vert_{L^{2\rho}(\Omega)}^{2\rho}  \right) 
\end{split}
\end{equation*}
with $\kappa_{2} > 0$ being a constant. Thanks to our assumption on the exponent $\rho$, we have $2\rho < \frac{2n}{n-2}$ and, moreover, since we know that the embedding $H^1(\Omega) \hookrightarrow L^p(\Omega)$ holds if and only if $p\leq\frac{2n}{n-2}$, it follows that $H^1(\Omega) \hookrightarrow L^{2\rho}(\Omega)$. Thus, there exists $\kappa_{3} > 0$ such that $\Vert u \Vert_{L^{2\rho}(\Omega)} \leq \kappa_{3} \Vert u\Vert_{X^{\frac{1}{2}}}$ and, hence,
$$
\left| \int_{\Omega} f(u)udx \right| \leq  \kappa_{2}\left(|f(0)|^2|\Omega| + \lambda_1^{-1}\|u\|_{X^{\frac{1}{2}}}^2 +   \kappa_{3}^{2\rho}\Vert u \Vert_{X^{\frac{1}{2}}}^{2\rho} \right). 
$$
Now, given $r > 0$, if $\Vert u\Vert_{X^{\frac{1}{2}}} \leq r$, then we get
\begin{equation}\label{3rd integral assertion}
\left| \int_{\Omega} f(u)udx \right| \leq \kappa_{2}(\lambda^{-1} + \kappa_{3}^{2\rho} r^{2\rho-2}) \Vert u\Vert_{X^{\frac{1}{2}}}^2 + \kappa_{2}|f(0)|^2|\Omega|.
\end{equation}

Next, we show the other inequality. At first, note that
\[
\left|\int_0^s f(\theta)d\theta\right| \leq |f(0)||s| + k\left(\frac{|s|^2}{2} + \frac{|s|^{\rho +1}}{\rho +1}  \right) \quad \text{for all} \quad s \in \mathbb{R}.
\]
Now, let $u \in X^{\frac{1}{2}}$. Using  the Poincar\'e inequality $\Vert u\Vert_X^2 \leq \lambda_1^{-1} \Vert u\Vert_{X^{\frac{1}{2}}}^2$, we obtain

\begin{equation*}
\begin{split}
\left| \int_{\Omega} \int_0^u f(s) dsdx\right| & \leq \int_{\Omega} \left[\frac{|f(0)|^2}{2} +\frac{|u|^2}{2} + k\left(\frac{|u|^2}{2} +  \frac{|u|^{\rho +1}}{\rho +1}  \right)\right]dx \\
&\leq  \kappa_{4}\left(|f(0)|^2|\Omega| + \|u\|_X^2 +  \Vert u \Vert_{L^{\rho +1}(\Omega)}^{\rho +1} \right) \\
& \leq \kappa_{4}\left(|f(0)|^2|\Omega| + \lambda^{-1}\|u\|_{X^{\frac{1}{2}}}^2 +  \Vert u \Vert_{L^{\rho +1}(\Omega)}^{\rho +1} \right)\\
& \leq \kappa_{5}\left(|f(0)|^2|\Omega| + \|u\|_{X^{\frac{1}{2}}}^2 +  \Vert u \Vert_{L^{\rho +1}(\Omega)}^{\rho +1} \right),
\end{split}
\end{equation*}
with $\kappa_{5} > 0$ being a constant. Since $1<\rho<\frac{n}{n-2}$, with $n\geq 3$, we have $2 < \rho + 1 < \frac{2n - 2}{n-2} < \frac{2n}{n-2},$ which ensures that $H^1(\Omega) \hookrightarrow L^{\rho + 1}(\Omega)$. Thus, there exists a constant $\kappa_{6} > 0$ such that $
\Vert u \Vert_{L^{\rho +1}(\Omega)} \leq \kappa_{6} \Vert u\Vert_{X^{\frac{1}{2}}}
$ and, hence,
\begin{equation*}
\begin{split}
\left| \int_{\Omega} \int_0^u f(s) dsdx\right| &\leq   \kappa_{5}\left(|f(0)|^2|\Omega| + \|u\|_{X^{\frac{1}{2}}}^2 + \kappa_{6}^{\rho+1} \Vert u \Vert_{X^{\frac{1}{2}}}^{\rho +1}\right).
\end{split}
\end{equation*}
Now, for a given $r > 0$, if $\Vert u\Vert_{X^{\frac{1}{2}}} \leq r$, then we get
\begin{equation}\label{4th integral assertion}
\left| \int_{\Omega} \int_0^u f(s) dsdx\right| \leq  \kappa_{5}|f(0)|^2|\Omega| + \kappa_{5}\left(1 + \kappa_{6}^{\rho+1}r^{\rho -1}\right)\Vert u \Vert_{X^{\frac{1}{2}}}^2. 
\end{equation}

Therefore, we conclude that, for all $r>0$ given, and for all $u\in X^{\frac{1}{2}}$ with $\|u\|_{X^{\frac{1}{2}}} \leq r$, taking
\[
C_r = \max\left\{\kappa_{2}(\lambda^{-1} + \kappa_{3}^{2\rho} r^{2\rho-2}),   \kappa_{5}\left(1 + \kappa_{6}^{\rho+1}r^{\rho -1}\right)\right\} > 0 \;\; \text{and} \;\;
C = |\Omega||f(0)|^2\max\{\kappa_{2}, \kappa_{5}\}
\]
it follows by \eqref{3rd integral assertion} and \eqref{4th integral assertion} that
\[
\left| \int_{\Omega} f(u)u dx \right| \leq C_{r}\Vert u\Vert_{X^{\frac{1}{2}}}^2 + C\quad \text{and} \quad
\left| \int_{\Omega}\int_0^u f(s)ds dx \right| \leq C_{r}\Vert u\Vert_{X^{\frac{1}{2}}}^2 + C.
\]

\end{proof}

\medskip

\noindent {\bf Proof of Theorem \ref{global-sol}:} 
By Corollary \ref{existence of solutions}, the problem \eqref{edp01}-\eqref{cond01} has a local solution $(u(t), u_t(t), v(t), v_t(t))$ in $Y_0$
defined on some interval $[\tau, \tau + t_0]$. Consider the original system \eqref{edp01}. Multiplying the first equation in \eqref{edp01} by $u_t$, and the second by $v_t$, we obtain
\begin{equation}\label{edp03} 
	\begin{split}
		&\frac{1}{2}\frac{d}{dt}\int_{\Omega}|u_t|^2dx + \frac{1}{2}\frac{d}{dt}\int_{\Omega}|\nabla u|^2dx + \frac{1}{2}\frac{d}{dt}\int_{\Omega}|u|^2dx + \eta\Vert (-\Delta)^{\frac{1}{4}}u_t\Vert_X^2 \\
		&+ a_{\epsilon}(t)\langle (-\Delta)^{\frac{1}{2}}v_t, u_t\rangle_X = \frac{d}{dt}\int_{\Omega}\int_0^u f(s)dsdx,
	\end{split}
\end{equation}
and
\begin{equation}\label{edp04}
	\frac{1}{2}\frac{d}{dt}\int_{\Omega}|v_t|^2dx + \frac{1}{2}\frac{d}{dt}\int_{\Omega}|\nabla v|^2dx + \eta\Vert (-\Delta)^{\frac{1}{4}}v_t\Vert_X^2 - a_{\epsilon}(t)\langle (-\Delta)^{\frac{1}{2}}u_t, v_t\rangle_X = 0, 
\end{equation}
for all $\tau < t \leq \tau + t_0$. Combining $(\ref{edp03})$ and $(\ref{edp04})$, we get
\begin{equation}\label{derivada da energia}
	\frac{d}{dt} \mathcal{E}(t) = -\eta\Vert (-\Delta)^{\frac{1}{4}}u_t\Vert_X^2 - \eta\Vert (-\Delta)^{\frac{1}{4}}v_t\Vert_X^2
\end{equation}
for all $\tau < t \leq \tau + t_0$, where
\begin{equation}\label{funcional de energia}
	\begin{split}
		\mathcal{E}(t) & = \frac{1}{2}\Vert u(t)\Vert_{X^{\frac{1}{2}}}^2 + \frac{1}{2}\Vert u(t)\Vert_X^2 + \frac{1}{2}\Vert u_t(t)\Vert_X^2 + \frac{1}{2}\Vert v(t)\Vert_{X^{\frac{1}{2}}}^2 + \frac{1}{2}\Vert v_t(t)\Vert_X^2 \\
		&\quad - \int_{\Omega}\int_0^u f(s) dsdx
	\end{split}
\end{equation}
is the total energy associated with the solution $(u(t), u_t(t), v(t), v_t(t))$ of the problem \eqref{edp01}-\eqref{cond01} in $Y_0$. The identity \eqref{derivada da energia} means that the map $t \mapsto \mathcal{E}(t)$ is monotone decreasing along solutions. Moreover, using the property $\mathcal{E}(t) \leq \mathcal{E}(\tau)$ for all $\tau \leq t \leq \tau + t_0$, we can obtain a priori estimate of the solution $(u(t),u_t(t), v(t),v_t(t))$ in $Y_0$. In fact, given $\delta > 0$, it follows by Lemma \ref{Lem_Aux_Int}, item $(ii)$, that there is $C_{\delta} > 0$ such that 
\[
\int_{\Omega}\int_0^u f(s)dsdx \leq  C_{\delta} + \delta\Vert u\Vert_X^2.
\]
Thus, for all $\tau < t \leq \tau + t_0$, we have
\[
\begin{split}
&\Vert u\Vert_{X^{\frac{1}{2}}}^2 + \Vert u_t\Vert_X^2 + \Vert v\Vert_{X^{\frac{1}{2}}}^2 + \Vert v_t\Vert_X^2  \leq \Vert u\Vert_{X^{\frac{1}{2}}}^2 + \Vert u\Vert_X^2 + \Vert u_t\Vert_X^2 + \Vert v\Vert_{X^{\frac{1}{2}}}^2 + \Vert v_t\Vert_X^2 \\
&= 2\mathcal{E}(t) + 2\int_{\Omega}\int_0^u f(s)dsdx \leq 2\mathcal{E}(\tau) + 2(\delta\Vert u\Vert_X^2 + C_{\delta})\\ 
&\leq 2(\mathcal{E}(\tau) + C_{\delta}) + 2\delta\lambda_1^{-1}\Vert u\Vert_{X^{\frac{1}{2}}}^2 \\
&\leq 2(\mathcal{E}(\tau) + C_{\delta}) + 2\delta\lambda_1^{-1}(\Vert u\Vert_{X^{\frac{1}{2}}}^2 + \Vert u_t\Vert_X^2 + \Vert v\Vert_{X^{\frac{1}{2}}}^2 + \Vert v_t\Vert_X^2), 
\end{split}
\]
where we have used the Poincar\'e inequality (recall that $\lambda_1 > 0$ is the first eigenvalue of the negative Laplacian operator with homogeneous Dirichlet boundary condition).

Now, choosing $\delta = \frac{\lambda_1}{4}$, we get
\[
\Vert u\Vert_{X^{\frac{1}{2}}}^2 + \Vert u_t\Vert_X^2 + \Vert v\Vert_{X^{\frac{1}{2}}}^2 + \Vert v_t\Vert_X^2 \leq 4\left(\mathcal{E}(\tau) + C_{\frac{\lambda_1}{4}}\right),
\]
that is,
\[
\Vert (u(t),u_t(t), v(t), v_t(t)) \Vert_{Y_0}^2 \leq 4\left(\mathcal{E}(\tau) + C_{\frac{\lambda_1}{4}}\right).
\]
This ensures that the problem $(\ref{edp01})-(\ref{cond01})$ has a global solution $W(t)$ in $Y_0$, which proves the result. \qed

\medskip

Since the problem $(\ref{edp01})-(\ref{cond01})$ has a global solution $W(t)$ in $Y_0$, we can define
an  evolution process $\{S(t, \tau): t\geq\tau\in\mathbb{R}\}$ in $Y_0$  by
\begin{equation}\label{evolution process of the problem}
S(t, \tau)W_0 = W(t), \quad  t\geq\tau\in\mathbb{R}.
\end{equation}

By  \cite{Nascimento}, we have
\begin{equation}\label{process_formulation}
S(t, \tau)W_0 = L(t, \tau)W_0 + U(t, \tau)W_0, \quad t\geq\tau\in\mathbb{R},
\end{equation}
where $\{L(t, \tau): t\geq\tau\in\mathbb{R}\}$ is the linear evolution process in $Y_0$ associated with the homogeneous problem
\begin{equation}\label{homogeneous pde} 
\begin{cases}
W_t + \mathcal{A}(t)W = 0, \ t > \tau, \\
W(\tau) = W_0, \ \tau \in \mathbb{R},
\end{cases}
\end{equation}
and
\begin{equation}\label{parte compacta do processo}
U(t, \tau)W_0 = \int_{\tau}^t L(t, s)F(S(s, \tau)W_0)ds.
\end{equation}

\section{Existence of the pullback attractor}\label{PullA}

In this section, we prove the existence of the pullback attractor of the problem \eqref{edp01}-\eqref{cond01}. To this end, we need to make a modification on the energy functional. More precisely, for $\gamma_1, \gamma_2 \in \mathbb{R}_+$, let us define $L_{\gamma_1, \gamma_2}\colon  Y_0 \to \mathbb{R}$ by the map
\begin{equation}\label{funcional de energia modificado}
	\begin{split}
		L_{\gamma_1, \gamma_2}(\phi, \varphi, \psi, \Phi) &=  \frac{1}{2}\Vert\phi\Vert_{X^{\frac{1}{2}}}^2 + \frac{1}{2}\Vert\phi\Vert_X^2 + \frac{1}{2}\Vert\varphi\Vert_X^2 + \frac{1}{2}\Vert\psi\Vert_{X^{\frac{1}{2}}}^2 + \frac{1}{2}\Vert\Phi\Vert_X^2 \\
		&+ \gamma_1\langle\phi,\varphi\rangle_X + \gamma_2\langle\psi,\Phi\rangle_X - \int_{\Omega}\int_0^{\phi}f(s)dsdx.
	\end{split}
\end{equation}

We start by noting that if
\[
\gamma_i < \frac{1}{2} \ \ \text{and} \ \ \frac{\gamma_i}{2}\lambda_1^{-1} < \frac{1}{4}, \ i = 1, 2,
\]
then
\begin{equation}\label{desig principal}
	\frac{1}{4}\Vert (\phi, \varphi, \psi, \Phi)\Vert_{Y_0}^2  \leq L_{\gamma_1, \gamma_2}(\phi, \varphi, \psi, \Phi) + \int_{\Omega}\int_0^{\phi} f(s)dsdx
\end{equation}
\[
\leq \frac{3}{4}(1 + \lambda_1^{-1})\Vert (\phi, \varphi, \psi, \Phi)\Vert_{Y_0}^2.
\]

Indeed, using the Cauchy-Schwartz and Young inequalities, we obtain
\[
\begin{split}
|\gamma_1\langle\phi,\varphi\rangle_X + \gamma_2\langle\psi,\Phi\rangle_X| & \leq \gamma_1\Vert\phi\Vert_X\Vert\varphi\Vert_X + \gamma_2\Vert\psi\Vert_X\Vert\Phi\Vert_X \\
& \leq \frac{\gamma_1}{2}(\Vert\phi\Vert_X^2 + \Vert\varphi\Vert_X^2) + \frac{\gamma_2}{2}(\Vert\psi\Vert_X^2 + \Vert\Phi\Vert_X^2) \\
& \leq \frac{\gamma_1}{2}\lambda_1^{-1}\Vert\phi\Vert_{X^{\frac{1}{2}}}^2 + \frac{\gamma_1}{2}\Vert\varphi\Vert_X^2 + \frac{\gamma_2}{2}\lambda_1^{-1}\Vert\psi\Vert_{X^{\frac{1}{2}}}^2 + \frac{\gamma_2}{2}\Vert\Phi\Vert_X^2 \\
& \leq \frac{1}{4}\Vert (\phi, \varphi, \psi, \Phi)\Vert_{Y_0}^2,
\end{split}
\]
which leads to
\begin{equation}\label{Eqneeded}
	\frac{1}{4}\Vert (\phi, \varphi, \psi, \Phi)\Vert_{Y_0}^2 \leq \frac{1}{2}\Vert (\phi, \varphi, \psi, \Phi)\Vert_{Y_0}^2
	+ \gamma_1\langle\phi,\varphi\rangle_X + \gamma_2\langle\psi,\Phi\rangle_X \leq  \frac{3}{4}\Vert (\phi, \varphi, \psi, \Phi)\Vert_{Y_0}^2.
\end{equation}

Consequently,
\[
\begin{split}
\frac{1}{4}\Vert (\phi, \varphi, \psi, \Phi)\Vert_{Y_0}^2 &\leq L_{\gamma_1, \gamma_2}(\phi, \varphi, \psi, \Phi) + \int_{\Omega}\int_0^{\phi} f(s)dsdx\leq  \frac{3}{4}\Vert (\phi, \varphi, \psi, \Phi)\Vert_{Y_0}^2 + \frac{1}{2}\Vert\phi\Vert_X^2.
\end{split}
\]

But since $\Vert\phi\Vert_X^2 \leq  \lambda_1^{-1}\Vert\phi\Vert_{X^{\frac{1}{2}}}^2,$ we have
\begin{equation}\label{Eqneeded1}
	\frac{3}{4}\Vert (\phi, \varphi, \psi, \Phi)\Vert_{Y_0}^2 + \frac{1}{2}\Vert\phi\Vert_X^2 \leq \frac{3(1 + \lambda_1^{-1})}{4} \Vert (\phi, \varphi, \psi, \Phi) \Vert_{Y_0}^2,
\end{equation}
and the claim is proved.

\begin{theorem}\label{the solution is exponentially dominated} There exists $R > 0$ such that for any bounded subset $B \subset Y_0$ one can find $t_0(B) > 0$ satisfying
\[
\Vert (u, u_t, v, v_t) \Vert_{Y_0}^2 \leq R \quad \text{for all} \quad t \geq  \tau + t_0(B).
\] 
In particular, the evolution process $\{S(t, \tau)\colon  t\geq\tau\in\mathbb{R}\}$ defined in \eqref{evolution process of the problem} is pullback strongly bounded dissipative.
\end{theorem}

\begin{proof}
At first, note that we can differentiate the expression $(\ref{funcional de energia modificado})$ along the solution $W(t) = (u(t), u_t(t), v(t),v_t(t))$ and, using $(\ref{derivada da energia})$ and $(\ref{funcional de energia})$, we get
\[
\begin{aligned}
&\frac{d}{dt}L_{\gamma_1, \gamma_2}(u, u_t, v, v_t) = \frac{d}{dt} \mathcal{E}(t) + \gamma_1\langle u_t, u_t\rangle_X + \gamma_1\langle u, u_{tt}\rangle_X + \gamma_2\langle v_t, v_t\rangle_X + \gamma_2\langle v, v_{tt}\rangle_X \\
&= -\eta\Vert A^{\frac{1}{4}}u_t\Vert_X^2 - \eta\Vert A^{\frac{1}{4}}v_t\Vert_X^2 + \gamma_1\Vert u_t\Vert_X^2 + \gamma_1\langle u, -Au - u - \eta A^{\frac{1}{2}}u_t - a_{\epsilon}(t)A^{\frac{1}{2}}v_t + f(u)\rangle_X \\
&+\gamma_2\Vert v_t\Vert_X^2 + \gamma_2\langle v, -Av - \eta A^{\frac{1}{2}}v_t + a_{\epsilon}(t)A^{\frac{1}{2}}u_t\rangle_X \\
& = -\eta\Vert u_t\Vert_{X^{\frac{1}{4}}}^2 - \eta\Vert v_t\Vert_{X^{\frac{1}{4}}}^2 + \gamma_1\Vert u_t\Vert_X^2 - \gamma_1(\Vert u\Vert_{X^{\frac{1}{2}}}^2 + \Vert u\Vert_X^2) - \gamma_1\eta\langle A^{\frac{1}{2}}u, u_t\rangle_X \\
& - \gamma_1a_{\epsilon}(t)\langle A^{\frac{1}{2}}u, v_t\rangle_X 
+\gamma_1\langle u, f(u)\rangle_X + \gamma_2\Vert v_t\Vert_X^2 - \gamma_2\Vert v\Vert_{X^{\frac{1}{2}}}^2 \\
& - \gamma_2\eta\langle A^{\frac{1}{2}}v, v_t\rangle_X + \gamma_2a_{\epsilon}(t)\langle A^{\frac{1}{2}}v, u_t\rangle_X.
\end{aligned}
\]

Now, if $c > 0$ is the embedding constant of $X^{\frac{1}{4}} \hookrightarrow X$, then one has
\begin{equation}\label{eq5.3A}
-\eta\Vert\cdot\Vert_{X^{\frac{1}{4}}}^2 \leq -\eta\frac{1}{c^2}\Vert\cdot\Vert_X^2.
\end{equation}

Moreover, by Lemma \ref{Lem_Aux_Int}, item $(ii)$, for each $\delta > 0$, there exists a constant $C_{\delta} > 0$ such that
\[
\int_{\Omega} f(u)u dx \leq \delta\Vert u\Vert_X^2 + C_{\delta},
\]
which implies
\begin{equation}\label{eq5.4A}
\gamma_1\langle u, f(u)\rangle_X  \leq \gamma_1\delta\Vert u\Vert_X^2 + \gamma_1C_{\delta} \leq \gamma_1\delta\lambda_1^{-1}\Vert u\Vert_{X^{\frac{1}{2}}}^2 + \gamma_1C_{\delta}.
\end{equation}

Thus, using \eqref{eq5.3A}, \eqref{eq5.4A} and the Cauchy-Schwartz and Young inequalities, we have
\[
\begin{aligned}
&\frac{d}{dt}L_{\gamma_1, \gamma_2}(u, u_t, v, v_t) \\
&\leq -\gamma_1(1 - \delta\lambda_1^{-1})\Vert u\Vert_{X^{\frac{1}{2}}}^2 - \left(\eta\frac{1}{c^2} - \gamma_1\right)\Vert u_t\Vert_X^2 - \gamma_2\Vert v\Vert_{X^{\frac{1}{2}}}^2 - \left(\eta\frac{1}{c^2} - \gamma_2\right)\Vert v_t\Vert_X^2 \\
&+ \gamma_1C_{\delta} +\gamma_1\eta\left(\frac{\epsilon_1}{2}\Vert u\Vert_{X^{\frac{1}{2}}}^2 + \frac{1}{2\epsilon_1}\Vert u_t\Vert_X^2\right) + \gamma_1a_1\left(\frac{1}{2\epsilon_2}\Vert u\Vert_{X^{\frac{1}{2}}}^2 + \frac{\epsilon_2}{2}\Vert v_t\Vert_X^2\right) \\
& +\gamma_2a_1\left(\frac{1}{2\epsilon_3}\Vert v\Vert_{X^{\frac{1}{2}}}^2 + \frac{\epsilon_3}{2}\Vert u_t\Vert_X^2\right) + \gamma_2\eta\left(\frac{\epsilon_4}{2}\Vert v\Vert_{X^{\frac{1}{2}}}^2 + \frac{1}{2\epsilon_4}\Vert v_t\Vert_X^2\right) \\
&= -\gamma_1\left(1 - \delta\lambda_1^{-1} - \eta\frac{\epsilon_1}{2} - a_1\frac{1}{2\epsilon_2}\right)\Vert u\Vert_{X^{\frac{1}{2}}}^2 - \left(\eta\frac{1}{c^2} - \gamma_1 - \gamma_1\eta\frac{1}{2\epsilon_1} - \gamma_2a_1\frac{\epsilon_3}{2}\right)\Vert u_t\Vert_X^2\\
&-\gamma_2\left(1 - a_1\frac{1}{2\epsilon_3} - \eta\frac{\epsilon_4}{2}\right)\Vert v\Vert_{X^{\frac{1}{2}}}^2 - \left(\eta\frac{1}{c^2} - \gamma_2 - \gamma_1a_1\frac{\epsilon_2}{2} - \gamma_2\eta\frac{1}{2\epsilon_4}\right)\Vert v_t\Vert_X^2 + \gamma_1C_{\delta}
\end{aligned}
\]
for all $\epsilon_1, \epsilon_2, \epsilon_3, \epsilon_4 > 0$. Choosing $\delta = \frac{\lambda_1}{8}$, $\epsilon_1=\epsilon_4=\frac{1}{\eta}$ and $\epsilon_2=\epsilon_3=2a_1$, we obtain
\[
\begin{aligned}
\frac{d}{dt}L_{\gamma_1, \gamma_2}(u, u_t, v, v_t) &\leq  -\frac{1}{8}\gamma_1\Vert u\Vert_{X^{\frac{1}{2}}}^2 - \left(\eta\frac{1}{c^2} - \gamma_1\left(1 + \frac{\eta^2}{2}\right) - \gamma_2a_1^2\right)\Vert u_t\Vert_X^2 - \frac{1}{4}\gamma_2\Vert v\Vert_{X^{\frac{1}{2}}}^2 \\
&-\left(\eta\frac{1}{c^2} - \gamma_1a_1^2 - \gamma_2\left(1 + \frac{\eta^2}{2}\right)\right)\Vert v_t\Vert_X^2 + \gamma_1C_{\frac{\lambda_1}{8}}.
\end{aligned}
\]

We may choose $\gamma_i > 0, i = 1, 2,$ sufficiently small such that
$$
\gamma_i < \frac{\eta}{4c^2}\min\left\{ \frac{1}{a_1^2}, \ \left(1 + \frac{\eta^2}{2}\right)^{-1} \right\}, \ i = 1, 2.
$$

Now, taking
$$
C_1 = \min \left\{ \frac{1}{8}\gamma_1, \ \eta\frac{1}{c^2} - \gamma_1\left(1 + \frac{\eta^2}{2}\right) - \gamma_2a_1^2, \ \frac{1}{4}\gamma_2, \ \eta\frac{1}{c^2} - \gamma_1a_1^2 - \gamma_2\left(1 + \frac{\eta^2}{2}\right) \right\} > 0,
$$
and $C_2 = \gamma_1C_{\frac{\lambda_1}{8}} > 0$, we obtain
\begin{equation}\label{derivada funcional norma}
\frac{d}{dt}L_{\gamma_1, \gamma_2}(u, u_t, v, v_t) \leq -C_1 \Vert (u, u_t, v, v_t) \Vert_{Y_0}^2 + C_2.
\end{equation}
Note that $C_1$ and $C_2$ are independent of $B$.

We claim that there exists $K > 0$ such that $L_{\gamma_1, \gamma_2} (u, u_t, v, v_t) \geq \frac{1}{8}\Vert (u, u_t, v, v_t) \Vert_{Y_0}^2 - K$.

In fact, by Lemma \ref{Lem_Aux_Int}, item $(ii)$, given $\tilde{\delta} > 0$, there exists a constant $C_{\tilde{\delta}} > 0$ such that
$$
\int_{\Omega} \int_0^u f(s) ds dx \leq \tilde{\delta} \Vert u\Vert_X^2 + C_{\tilde{\delta}},
$$
which, together with $\Vert u\Vert_X^2 \leq  \lambda_1^{-1}\Vert u\Vert_{X^{\frac{1}{2}}}^2$, implies
\[
\begin{split}
L_{\gamma_1, \gamma_2}(u, u_t, v, v_t) &\geq \frac{1}{4}\Vert (u, u_t, v, v_t)\Vert_{Y_0}^2 - \int_{\Omega}\int_0^u f(s)dsdx \\
&\geq \frac{1}{4}\Vert (u, u_t, v, v_t)\Vert_{Y_0}^2 - \tilde{\delta}\lambda_1^{-1}\Vert u\Vert_{X^{\frac{1}{2}}}^2 - C_{\tilde{\delta}} \\
\end{split}
\]
\[
\begin{split}
\hspace{2cm} &\quad \geq \left(\frac{1}{4} - \tilde{\delta}\lambda_1^{-1}\right)\Vert (u, u_t, v, v_t)\Vert_{Y_0}^2 - C_{\tilde{\delta}}.
\end{split}
\]

Choosing $\tilde{\delta} = \frac{\lambda_1}{8}$, we get
\begin{equation}\label{Ch2eqdecayL}
L_{\gamma_1, \gamma_2} (u, u_t, v, v_t) \geq \frac{1}{8}\Vert (u, u_t, v, v_t) \Vert_{Y_0}^2 - K,
\end{equation}
where $K = C_{\frac{\lambda_1}{8}} > 0$, which proves the claim.

Now, define the set
\[
\ell_r = \sup\{\Vert (u, u_t, v, v_t) \Vert_{Y_0}^2\colon t \geq \tau, \Vert (u(\tau), u_t(\tau), v(\tau), v_t(\tau)) \Vert_{Y_0}^2 \leq r\}.
\]
Note that $\ell_r < \infty$ for each $r > 0$. In fact, by the proof of Theorem \ref{global-sol}, we have
\[
\Vert W(t) \Vert_{Y_0}^2 = \Vert (u(t),u_t(t), v(t), v_t(t)) \Vert_{Y_0}^2 \leq 4\left(\mathcal{E}(\tau) + C_{\frac{\lambda_1}{4}}\right), \quad t \geq \tau,
\]
where
\[
\begin{split}
\mathcal{E}(\tau) & = \frac{1}{2}\Vert W(\tau)\Vert_{Y_0}^2 + \frac{1}{2}\Vert u(\tau)\Vert_X^2  - \int_{\Omega}\int_0^{u(\tau)} f(s)dsdx\\
& \leq   \frac{1}{2}\Vert W(\tau)\Vert_{Y_0}^2  + \lambda_1^{-1}\Vert u(\tau)\Vert_{X^{\frac{1}{2}}}^2  +\left|\int_{\Omega}\int_0^{u(\tau)} f(s)dsdx\right| \\
& \leq \left(\frac{1}{2} +\lambda_1^{-1}\right)\Vert W(\tau)\Vert_{Y_0}^2 +  C_{r}\Vert u(\tau)\Vert_{X^{\frac{1}{2}}}^2 + C\\
& \leq \left(\frac{1}{2} +\lambda_1^{-1}\right) r +  C_{r}r + C. \hspace{3.5cm}
\end{split}
\]
with $C_r$ and $C$ given by Lemma \ref{Lem_Aux_Int}, item $(iii)$. This shows that $\ell_r < \infty$.
 
Now, we claim that given a bounded set $B \subset Y_0$ there exists $t_0(B) > 0$ such that
\[
\Vert (u, u_t, v, v_t) \Vert_{Y_0}^2 \leq  \max\left\{8K, \ell_{\frac{C_2 +1}{C_1}}\right\} \quad \text{for all} \quad t \geq  \tau + t_0(B).
\] 
In fact, let $B \subset Y_0$ be a bounded set. Let $r_0 > 0$ be such that $B \subset B_{Y_0}(0, r_0)$. By \eqref{desig principal} and Lemma \ref{Lem_Aux_Int}, we obtain
\[
L_{\gamma_1, \gamma_2}(u(\tau), u_t(\tau), v(\tau), v_t(\tau)) \leq \frac{3}{4}(1 + \lambda_1^{-1})r_0+ r_0C_{r_0} +C  = T_{r_0},
\]
for all $(u(\tau), u_t(\tau), v(\tau), v_t(\tau)) \in B$.

Let $(u(\tau), u_t(\tau), v(\tau), v_t(\tau)) \in B$ be arbitrary. If $\Vert (u, u_t, v, v_t) \Vert_{Y_0}^2 > \frac{C_2 +1}{C_1}$ for all $t \geq \tau$ then
\[
\frac{d}{dt}L_{\gamma_1, \gamma_2}(u, u_t, v, v_t) \leq -C_1 \Vert (u, u_t, v, v_t) \Vert_{Y_0}^2 + C_2 \leq -1 \quad \text{for all} \quad t \geq \tau,
\]
which implies
\[
L_{\gamma_1, \gamma_2}(u, u_t, v, v_t) \leq L_{\gamma_1, \gamma_2}(u(\tau), u_t(\tau), v(\tau), v_t(\tau)) -(t - \tau) \quad \text{for all} \quad t \geq \tau.
\]
Thus, $L_{\gamma_1, \gamma_2}(u, u_t, v, v_t) \leq 0$ for all $t \geq \tau + T_{r_0}$. Consequently, using \eqref{Ch2eqdecayL}, we have
\[
\Vert (u, u_t, v, v_t) \Vert_{Y_0}^2 \leq 8K \quad \text{for all} \quad t \geq \tau + T_{r_0}.
\]

On the other hand, if there exists $t_u \geq \tau $ such that  $\Vert (u(t_u), u_t(t_u), v(t_u), v_t(t_u)) \Vert_{Y_0}^2 \leq \frac{C_2 +1}{C_1}$ (take the smallest $t_u$ with this property) then 
\[
\Vert (u, u_t, v, v_t) \Vert_{Y_0}^2 \leq \ell_{\frac{C_2 +1}{C_1}} \quad \text{for all} \quad t \geq  t_u.
\]
Set
\[
B^u = \left\{w_0 \in B\colon \text{there exists} \; t_u^{w_0} > \tau \; \text{such that} \; \|W(t_u^{w_0})w_0\|_{Y_0}^2 = \frac{C_2 + 1}{C_1} \; \text{and} \right.
\]
\[
\left. \|W(t)w_0\|_{Y_0}^2 > \frac{C_2 + 1}{C_1} \;\text{for all} \; \tau \leq t < t_u^{w_0}\right\}.
\]
We claim that $T_u(B) = \sup\{t_u^{w_0}\colon w_0 \in B^u\} < \infty$. In fact, suppose to the contrary that there exists a sequence $\{w_0^n\}_{n \in \mathbb{N}} \subset B^u$ such that $t_u^{w_0^n} \to \infty$ as $n \to \infty$.
Since $\Vert W(t)w_0^n \Vert_{Y_0}^2 \geq \frac{C_2 +1}{C_1}$ for all $\tau \leq t \leq t_u^{w_0^n}$, we conclude that
\[
L_{\gamma_1, \gamma_2}(W(t)w_0^n) \leq L_{\gamma_1, \gamma_2}(w_0^{n}) -(t - \tau) \leq T_{r_0} - t + \tau\quad \text{for all} \quad \tau \leq t \leq t_u^{w_0^n}.
\] 
This implies that $\displaystyle\lim_{n\to \infty} L_{\gamma_1, \gamma_2}(W(t_u^{w_0^n})w_0^n) = -\infty$. But, using \eqref{desig principal}, we obtain
\[
\begin{split}
\dfrac{1}{4}\|W(t_u^{w_0^n})w_0^n\|_{Y_0}^2 & \leq  L_{\gamma_1, \gamma_2}(W(t_u^{w_0^n})w_0^n) + \int_{\Omega}\int_0^{u(t_u^{w_0^n})} f(s)dsdx \\
& 
\leq  L_{\gamma_1, \gamma_2}(W(t_u^{w_0^n})w_0^n)  + \left| \int_{\Omega}\int_0^{u(t_u^{w_0^n})} f(s)dsdx\right|\\
& \leq  L_{\gamma_1, \gamma_2}(W(t_u^{w_0^n})w_0^n)   + C_{\frac{C_2+1}{C_1}}\|u(t_u^{w_0^n})\|_{X^{\frac{1}{2}}}^2 +C \\
& \leq  L_{\gamma_1, \gamma_2}(W(t_u^{w_0^n})w_0^n)   + C_{\frac{C_2+1}{C_1}}\|W(t_u^{w_0^n})w_0^{n}\|_{Y_0}^2  +C\\
& =  L_{\gamma_1, \gamma_2}(W(t_u^{w_0^n})w_0^n)   + C_{\frac{C_2+1}{C_1}}\frac{C_2+1}{C_1}  +C\\
\end{split}
\]
which contradicts the fact that $\displaystyle\lim_{n\to \infty} L_{\gamma_1, \gamma_2}(W(t_u^{w_0^n})w_0^n) = -\infty$.

Taking $t_0(B) = \max\{T_u(B), T_{r_0}\}$, we conclude that
\[
\Vert (u, u_t, v, v_t) \Vert_{Y_0}^2 \leq  \max\left\{8K, \ell_{\frac{C_2 +1}{C_1}}\right\} \quad \text{for all} \quad t \geq  \tau + t_0(B).
\]

This shows that, if $s \leq t$ and $B\subset Y_0$ is a bounded set then
\[
S(s, \tau)B \subset \overline{B_{Y_0}(0, R)} \quad \text{for all} \quad \tau \leq \tau_0(s, B),
\]
where $\tau_0(s, B) = s - t_0(B)$ and $R = \max\left\{8K, \ell_{\frac{C_2 +1}{C_1}}\right\}$. Therefore, the process given by $(\ref{evolution process of the problem})$ is  pullback strongly bounded dissipative.
\end{proof}

Next, we prove that the solutions of problem \eqref{edp abstrata} are uniformly exponentially dominated when the initial data are in bounded subsets of $Y_0$.

\begin{theorem}\label{the solution is exponentially dominatedP2}  Let $B \subset Y_0$ be a bounded set. If $W\colon [\tau, \infty) \to Y_0$ is the global solution of \eqref{edp abstrata} starting at $W_0 \in B$, then there are positive constants $\sigma = \sigma(B)$, $K_1 = K_1(B)$ and $K_2 = K_2(B)$ such that
\[
\Vert W(t)\Vert_{Y_0}^2 \leq K_1e^{-\sigma(t - \tau)} + K_2, \quad t\geq \tau.
\]
\end{theorem}

\begin{proof}
Let $r>0$ be such that $B \subset B_{Y_0}(0, r)$. 
We claim that there is $M_r > 0$ and $C>0$ such that $L_{\gamma_1, \gamma_2}(u, u_t, v, v_t) \leq M_r\Vert (u, u_t, v, v_t)\Vert_{Y_0}^2 + C$ for all $t \geq \tau$. In fact, by $(\ref{desig principal})$, we have
\[
L_{\gamma_1, \gamma_2}(u, u_t, v, v_t) + \int_{\Omega}\int_0^u f(s)dsdx \leq \frac{3}{4}(1 + \lambda_1^{-1})\Vert (u, u_t, v, v_t)\Vert_{Y_0}^2.
\]
By the proof of Theorem \ref{the solution is exponentially dominated}, the set
\[
\ell_r = \sup\{\Vert (u, u_t, v, v_t) \Vert_{Y_0}^2\colon t \geq \tau, \Vert (u(\tau), u_t(\tau), v(\tau), v_t(\tau)) \Vert_{Y_0}^2 \leq r\} < \infty.
\]
Now, using Lemma \ref{Lem_Aux_Int}, condition $(iii)$, there are constants $C_{\ell_r} > 0$ and $C$ such that
\[
\left| \int_{\Omega}\int_0^u f(s)dsdx \right| \leq C_{\ell_r}\Vert u\Vert_{X^{\frac{1}{2}}}^2 + C
\]
whenever $\Vert (u(\tau), u_t(\tau), v(\tau), v_t(\tau))\Vert_{Y_0}^2 \leq r$.  Hence, if $\Vert (u(\tau), u_t(\tau), v(\tau), v_t(\tau))\Vert_{Y_0}^2 \leq r$, then
\[
\begin{split}
L_{\gamma_1, \gamma_2}(u, u_t, v, v_t) &\leq \frac{3}{4}(1 + \lambda_1^{-1})\Vert (u, u_t, v, v_t)\Vert_{Y_0}^2 - \int_{\Omega}\int_0^u f(s)dsdx  \\
&\leq \frac{3}{4}(1 + \lambda_1^{-1})\Vert (u, u_t, v, v_t)\Vert_{Y_0}^2 + C_{\ell_r}\Vert u\Vert_{X^{\frac{1}{2}}}^2 +C \\
&\leq M_r\Vert (u, u_t, v, v_t)\Vert_{Y_0}^2 +C,
\end{split}
\]
where $M_r = \frac{3}{4}(1 + \lambda_1^{-1}) + C_{\ell_r} > 0$, which proves the claim.

Using the proof of Theorem \ref{the solution is exponentially dominated}, it follows by \eqref{derivada funcional norma}  that
$$
\frac{d}{dt}L_{\gamma_1, \gamma_2}(W(t)) \leq -\frac{C_1}{M_r} L_{\gamma_1, \gamma_2}(W(t)) +\dfrac{CC_1}{M_r} + C_2, \quad t\geq\tau,
$$
which implies
\[
L_{\gamma_1, \gamma_2}(W(t)) \leq L_{\gamma_1, \gamma_2}(W(\tau)) e^{-\frac{C_1}{M_r} (t - \tau)} + \left(C_2 + \dfrac{CC_1}{M_r} \right)\frac{M_r}{C_1}, \quad t\geq\tau,
\]
and, using the fact that $\frac{1}{8}\Vert W(t)\Vert_{Y_0}^2 - K \leq L_{\gamma_1, \gamma_2}(W(t))$ (see \eqref{Ch2eqdecayL}), we conclude that
$$
\Vert W(t)\Vert_{Y_0}^2 \leq 8L_{\gamma_1, \gamma_2}(W(\tau)) e^{-\frac{C_1}{M_r} (t - \tau)} + 8\left( C_2\frac{M_r}{C_1} + C+K \right), \quad t\geq\tau.
$$

Since $L_{\gamma_1, \gamma_2}(W(\tau)) \leq K_r M_r+C$, we get
$$
\Vert W(t)\Vert_{Y_0}^2 \leq 8(K_rM_r+C)e^{-\frac{C_1}{M_r} (t - \tau)} + 8\left( C_2\frac{M_r}{C_1} + C+K \right), \quad t \geq \tau,
$$
and the result follows by taking $\sigma = \frac{C_1}{M_r}$, $K_1 = 8(K_rM_r+C)$ and $K_2 = 8\left( C_2\frac{M_r}{C_1} + C+K \right)$.
\end{proof}

\begin{theorem}\label{the solution decays} 
Let $B \subset Y_0$ be a bounded set and denote by $L\colon [\tau, \infty) \to Y_0$ the solution of the homogeneous problem \eqref{homogeneous pde} starting in $W_0 \in B$. Then there exist positive constants $K = K(B)$ and $\zeta$ such that
$$
\Vert L(t)\Vert_{Y_0}^2 \leq K e^{-\zeta(t - \tau)}, \quad t\geq \tau.
$$
\end{theorem}

\begin{proof}
The proof is analogous to the proof of Theorem \ref{the solution is exponentially dominatedP2} taking $f\equiv 0$. 
\end{proof}

\begin{proposition}\label{proc-c}
For each $t>\tau\in\mathbb{R}$, the evolution process $S(t, \tau)\colon  Y_0 \to Y_0$ given in $(\ref{evolution process of the problem})$ is a compact map.
\end{proposition}

\begin{proof}
Using the identity $(\ref{derivada da energia})$, the energy functional $(\ref{funcional de energia})$ and the Cauchy-Schwartz and Young inequalities, we obtain
\[
\begin{split}
&\frac{1}{2} \frac{d}{dt} \left( \Vert u\Vert_{X^{\frac{1}{2}}}^2 + \Vert u\Vert_X^2 + \Vert u_t\Vert_X^2 + \Vert v\Vert_{X^{\frac{1}{2}}}^2 + \Vert v_t\Vert_X^2 \right) + \eta\Vert u_t\Vert_{X^{\frac{1}{4}}}^2 + \eta\Vert v_t\Vert_{X^{\frac{1}{4}}}^2 \\
&= \langle f(u), u_t \rangle_X \leq \Vert f(u)\Vert_X \Vert u_t\Vert_X \leq \tilde{c}\Vert f(u)\Vert_X \Vert u_t\Vert_{X^{\frac{1}{4}}} \leq \frac{1}{2\epsilon}\Vert f(u)\Vert_X^2 + \frac{\epsilon}{2}\tilde{c}^{2} \Vert u_t\Vert_{X^{\frac{1}{4}}}^2,
\end{split}
\]
for all $\epsilon > 0$, where $\tilde{c} > 0$ is the embedding constant of $X^{\frac{1}{4}} \hookrightarrow X$. Choosing $\epsilon = \frac{\eta}{\tilde{c}^{2}}$, we get
\begin{equation}\label{comp estimate1}
\begin{aligned}
&\frac{1}{2} \frac{d}{dt} \left( \Vert u\Vert_{X^{\frac{1}{2}}}^2 + \Vert u\Vert_X^2 + \Vert u_t\Vert_X^2 + \Vert v\Vert_{X^{\frac{1}{2}}}^2 + \Vert v_t\Vert_X^2 \right) + \frac{\eta}{2}\Vert u_t\Vert_{X^{\frac{1}{4}}}^2 + \eta\Vert v_t\Vert_{X^{\frac{1}{4}}}^2 \\
&\leq \frac{\tilde{c}^{2}}{2\eta} \Vert f(u)\Vert_X^2.
\end{aligned}
\end{equation}

Now, knowing that the embedding $X^{\frac{1}{2}} \hookrightarrow L^{2\rho}(\Omega)$ holds for $1 < \rho\leq \frac{n}{n-2}$, and using Lemma \ref{Lem_Aux_Int}, condition $(i)$, we get
\begin{equation}\label{comp estimate2}
\begin{split}
\Vert f(u)\Vert_X^2 &\leq \int_{\Omega} [c(1 + |u|^{\rho})]^2 dx \leq c_1 \int_{\Omega} (1 + |u|^{2\rho}) dx  \\
&= c_1|\Omega| + c_1\Vert u\Vert_{L^{2\rho}(\Omega)}^{2\rho} \leq c_1|\Omega| + c_2\Vert u\Vert_{X^{\frac{1}{2}}}^{2\rho} \leq c_1|\Omega| + c_2\Vert W\Vert_{Y_0}^{2\rho},
\end{split}
\end{equation}
where $c_1, c_2$ are positive constants and $W(t) = (u(t), u_t(t), v(t), v_t(t))$. Thus, combining $(\ref{comp estimate1})$ and $(\ref{comp estimate2})$, we obtain
$$
\frac{d}{dt} \left( \Vert W\Vert_{Y_0}^2 + \Vert u\Vert_X^2 \right) + \eta\Vert u_t\Vert_{X^{\frac{1}{4}}}^2 + 2\eta\Vert v_t\Vert_{X^{\frac{1}{4}}}^2 \leq \frac{\tilde{c}^{2}c_1|\Omega|}{\eta} + \frac{\tilde{c}^{2}c_2}{\eta} \Vert W\Vert_{Y_0}^{2\rho}.
$$

Integrating the previous inequality from $\tau$ to $t$, we obtain
\begin{equation}\label{comp estimate3}
\begin{aligned}
&\Vert W(t)\Vert_{Y_0}^2 + \Vert u(t)\Vert_X^2 + \eta\int_{\tau}^{t} \Vert u_t(r)\Vert_{X^{\frac{1}{4}}}^2 dr + 2\eta\int_{\tau}^{t} \Vert v_t(r)\Vert_{X^{\frac{1}{4}}}^2 dr \\
&\leq \frac{\tilde{c}^{2}c_1|\Omega|}{\eta}(t-\tau) + \frac{\tilde{c}^{2}c_2}{\eta} \int_{\tau}^{t} \Vert W(r)\Vert_{Y_0}^{2\rho} dr + \Vert W(\tau)\Vert_{Y_0}^2 + \Vert u(\tau)\Vert_X^2 \\
&\leq \frac{\tilde{c}^{2}c_1|\Omega|}{\eta}(t-\tau) + \frac{\tilde{c}^{2}c_2}{\eta} \int_{\tau}^{t} \Vert W(r)\Vert_{Y_0}^{2\rho} dr + (1 + \lambda_1^{-1})\Vert W(\tau)\Vert_{Y_0}^2,
\end{aligned}
\end{equation}
where we have used the Poincaré inequality $\Vert u(\tau)\Vert_X^2 \leq \lambda_1^{-1}\Vert u(\tau)\Vert_{X^{\frac{1}{2}}}^2$. Also, note that inequality $(\ref{comp estimate3})$ implies
\begin{equation}\label{comp estimate4}
\begin{aligned}
&\int_{\tau}^{t} \Vert u_t(r)\Vert_{X^{\frac{1}{4}}}^2 dr + \int_{\tau}^{t} \Vert v_t(r)\Vert_{X^{\frac{1}{4}}}^2 dr \\
&\leq \frac{\tilde{c}^{2}c_1|\Omega|}{\eta^2}(t-\tau) + \frac{\tilde{c}^{2}c_2}{\eta^2} \int_{\tau}^{t} \Vert W(r)\Vert_{Y_0}^{2\rho} dr + \frac{1 + \lambda_1^{-1}}{\eta} \Vert W(\tau)\Vert_{Y_0}^2.
\end{aligned}
\end{equation}

Now, consider the original system $(\ref{edp01})$. By taking the inner product of the first equation in \eqref{edp01} with $A^{\frac{1}{2}} u$, and also the inner product of the second equation in \eqref{edp01} with $A^{\frac{1}{2}} v$, and noticing the identity
$$\langle u_{tt}, A^{\frac{1}{2}} u\rangle_X = \frac{d}{dt} \langle u_t, A^{\frac{1}{2}} u\rangle_X - \Vert u_t\Vert_{X^{\frac{1}{4}}}^2,$$
we obtain,
\[
\begin{aligned}
&\frac{d}{dt} \langle u_t, A^{\frac{1}{2}} u\rangle_X - \Vert u_t\Vert_{X^{\frac{1}{4}}}^2 + \Vert u\Vert_{X^{\frac{3}{4}}}^2 + \Vert u\Vert_{X^{\frac{1}{4}}}^2 + \frac{\eta}{2} \frac{d}{dt} \Vert u\Vert_{X^{\frac{1}{2}}}^2 + a_{\epsilon}(t) \langle A^{\frac{1}{2}} v_t, A^{\frac{1}{2}} u\rangle_X \\
&+ \frac{d}{dt} \langle v_t, A^{\frac{1}{2}} v\rangle_X - \Vert v_t\Vert_{X^{\frac{1}{4}}}^2 + \Vert v\Vert_{X^{\frac{3}{4}}}^2 + \frac{\eta}{2} \frac{d}{dt} \Vert v\Vert_{X^{\frac{1}{2}}}^2 - a_{\epsilon}(t) \langle A^{\frac{1}{2}} u_t, A^{\frac{1}{2}} v\rangle_X \\
&= \langle f(u), A^{\frac{1}{2}} u\rangle_X.
\end{aligned}
\]

Once again, using the Cauchy-Schwartz and Young inequalities, we have
\[
\begin{split}
&\frac{d}{dt} \left( \langle u_t, A^{\frac{1}{2}} u\rangle_X + \langle v_t, A^{\frac{1}{2}} v\rangle_X \right) + \frac{\eta}{2} \frac{d}{dt} \left( \Vert u\Vert_{X^{\frac{1}{2}}}^2 + \Vert v\Vert_{X^{\frac{1}{2}}}^2 \right) + \Vert u\Vert_{X^{\frac{3}{4}}}^2 + \Vert u\Vert_{X^{\frac{1}{4}}}^2 + \Vert v\Vert_{X^{\frac{3}{4}}}^2 \\
\leq{}& \Vert u_t\Vert_{X^{\frac{1}{4}}}^2 + \Vert v_t\Vert_{X^{\frac{1}{4}}}^2 + a_1\Vert v_t\Vert_{X^{\frac{1}{4}}} \Vert u\Vert_{X^{\frac{3}{4}}} + a_1\Vert u_t\Vert_{X^{\frac{1}{4}}} \Vert v\Vert_{X^{\frac{3}{4}}} + \Vert f(u)\Vert_X \Vert u\Vert_{X^{\frac{1}{2}}} \\
\leq{}& \left( 1 + \frac{1}{2\epsilon_2} \right) \Vert u_t\Vert_{X^{\frac{1}{4}}}^2 + \left( 1 + \frac{1}{2\epsilon_1} \right) \Vert v_t\Vert_{X^{\frac{1}{4}}}^2 + \frac{\epsilon_1}{2} a_1^2\Vert u\Vert_{X^{\frac{3}{4}}}^2 + \frac{\epsilon_2}{2} a_1^2\Vert v\Vert_{X^{\frac{3}{4}}}^2 \\ 
& + \frac{1}{2}\Vert f(u)\Vert_X^2 + \frac{1}{2}\Vert u\Vert_{X^{\frac{1}{2}}}^2,
\end{split}
\]
for all $\epsilon_1, \epsilon_2 > 0$. Choosing $\epsilon_1 = \epsilon_2 = \frac{1}{a_1^2}$, and using $(\ref{comp estimate2})$, we get
\[
\begin{split}
&\frac{d}{dt} \left( \langle u_t, A^{\frac{1}{2}} u\rangle_X + \langle v_t, A^{\frac{1}{2}} v\rangle_X \right) + \frac{\eta}{2} \frac{d}{dt} \left( \Vert u\Vert_{X^{\frac{1}{2}}}^2 + \Vert v\Vert_{X^{\frac{1}{2}}}^2 \right) + \frac{1}{2} \Vert u\Vert_{X^{\frac{3}{4}}}^2 + \frac{1}{2} \Vert v\Vert_{X^{\frac{3}{4}}}^2 \\
&\leq \frac{2 + a_1^2}{2} \Vert u_t\Vert_{X^{\frac{1}{4}}}^2 + \frac{2 + a_1^2}{2} \Vert v_t\Vert_{X^{\frac{1}{4}}}^2 + \frac{c_1|\Omega|}{2} + \frac{c_2}{2}\Vert W\Vert_{Y_0}^{2\rho} + \frac{1}{2}\Vert W\Vert_{Y_0}^{2}.
\end{split}
\]

Integrating the previous inequality from $\tau$ to $t$, and using $(\ref{comp estimate4})$, we obtain

\[
\begin{split}
&\frac{\eta}{2} \left( \Vert u(t)\Vert_{X^{\frac{1}{2}}}^2 + \Vert v(t)\Vert_{X^{\frac{1}{2}}}^2 \right) + \frac{1}{2} \int_{\tau}^{t} \Vert u(r)\Vert_{X^{\frac{3}{4}}}^2 dr + \frac{1}{2} \int_{\tau}^{t} \Vert v(r)\Vert_{X^{\frac{3}{4}}}^2 dr \\
\leq{}& \frac{2 + a_1^2}{2} \left( \int_{\tau}^{t} \Vert u_t(r)\Vert_{X^{\frac{1}{4}}}^2 dr + \int_{\tau}^{t} \Vert v_t(r)\Vert_{X^{\frac{1}{4}}}^2 dr \right) + \frac{c_1|\Omega|}{2}(t-\tau) + \frac{c_2}{2} \int_{\tau}^{t} \Vert W(r)\Vert_{Y_0}^{2\rho} dr \\
&+ \frac{1}{2} \int_{\tau}^{t} \Vert W(r)\Vert_{Y_0}^{2} dr - \langle u_t(t), A^{\frac{1}{2}} u(t)\rangle_X - \langle v_t(t), A^{\frac{1}{2}} v(t)\rangle_X \\
&+ \langle u_t(\tau), A^{\frac{1}{2}} u(\tau)\rangle_X + \langle v_t(\tau), A^{\frac{1}{2}} v(\tau)\rangle_X + \frac{\eta}{2} \left( \Vert u(\tau)\Vert_{X^{\frac{1}{2}}}^2 + \Vert v(\tau)\Vert_{X^{\frac{1}{2}}}^2 \right) \\
\leq{}& \frac{2 + a_1^2}{2} \left( \frac{\tilde{c}^{2}c_1|\Omega|}{\eta^2}(t-\tau) + \frac{\tilde{c}^{2}c_2}{\eta^2} \int_{\tau}^{t} \Vert W(r)\Vert_{Y_0}^{2\rho} dr + \frac{1 + \lambda_1^{-1}}{\eta} \Vert W(\tau)\Vert_{Y_0}^2 \right) + \frac{c_1|\Omega|}{2}(t-\tau) \\
&+ \frac{c_2}{2} \int_{\tau}^{t} \Vert W(r)\Vert_{Y_0}^{2\rho} dr + \frac{1}{2} \int_{\tau}^{t} \Vert W(r)\Vert_{Y_0}^{2} dr + \frac{1}{2}\Vert W(t)\Vert_{Y_0}^{2} + \frac{1 + \eta}{2}\Vert W(\tau)\Vert_{Y_0}^{2},
\end{split}
\]
which implies
\begin{equation}\label{comp estimate5}
\begin{aligned}
&\int_{\tau}^{t} \Vert u(r)\Vert_{X^{\frac{3}{4}}}^2 dr + \int_{\tau}^{t} \Vert v(r)\Vert_{X^{\frac{3}{4}}}^2 dr \\
\leq{}& (2 + a_1^2) \left( \frac{\tilde{c}^{2}c_1|\Omega|}{\eta^2}(t-\tau) + \frac{\tilde{c}^{2}c_2}{\eta^2} \int_{\tau}^{t} \Vert W(r)\Vert_{Y_0}^{2\rho} dr + \frac{1 + \lambda_1^{-1}}{\eta} \Vert W(\tau)\Vert_{Y_0}^2 \right) \\
&+ c_1|\Omega|(t-\tau) + c_2\int_{\tau}^{t} \Vert W(r)\Vert_{Y_0}^{2\rho} dr + \int_{\tau}^{t} \Vert W(r)\Vert_{Y_0}^{2} dr \\
&+ \Vert W(t)\Vert_{Y_0}^{2} + (1 + \eta)\Vert W(\tau)\Vert_{Y_0}^{2}.
\end{aligned}
\end{equation}

On the other hand, taking the inner product of the first equation in \eqref{edp01} with $A^{\frac{1}{2}} u_t$, and also the inner product of the second equation in \eqref{edp01} with $A^{\frac{1}{2}} v_t$, and using $(\ref{comp estimate2})$, we have
\[
\begin{aligned}
&\frac{1}{2} \frac{d}{dt} \left( \Vert u_t\Vert_{X^{\frac{1}{4}}}^2 + \Vert u\Vert_{X^{\frac{3}{4}}}^2 + \Vert u\Vert_{X^{\frac{1}{4}}}^2 + \Vert v_t\Vert_{X^{\frac{1}{4}}}^2 + \Vert v\Vert_{X^{\frac{3}{4}}}^2 \right) + \eta\Vert u_t\Vert_{X^{\frac{1}{2}}}^2 + \eta\Vert v_t\Vert_{X^{\frac{1}{2}}}^2 \\
&= \langle f(u), A^{\frac{1}{2}} u_t \rangle_X \leq \Vert f(u)\Vert_X \Vert u_t\Vert_{X^{\frac{1}{2}}} \leq \frac{1}{2\eta} \Vert f(u)\Vert_X^2 + \frac{\eta}{2} \Vert u_t\Vert_{X^{\frac{1}{2}}}^2 \\
&\leq \frac{c_1|\Omega|}{2\eta} + \frac{c_2}{2\eta} \Vert W\Vert_{Y_0}^{2\rho} + \frac{\eta}{2} \Vert u_t\Vert_{X^{\frac{1}{2}}}^2,
\end{aligned}
\]
which yields
\[
\frac{d}{dt} \left( \Vert u\Vert_{X^{\frac{3}{4}}}^2 + \Vert u\Vert_{X^{\frac{1}{4}}}^2 + \Vert u_t\Vert_{X^{\frac{1}{4}}}^2 + \Vert v\Vert_{X^{\frac{3}{4}}}^2 + \Vert v_t\Vert_{X^{\frac{1}{4}}}^2 \right) \leq \frac{c_1|\Omega|}{\eta} + \frac{c_2}{\eta} \Vert W\Vert_{Y_0}^{2\rho}.
\]

Integrating the previous inequality from $r$ to $t$, for $\tau < r < t$, we have
\[
\begin{aligned}
&\Vert u(t)\Vert_{X^{\frac{3}{4}}}^2 + \Vert u(t)\Vert_{X^{\frac{1}{4}}}^2 + \Vert u_t(t)\Vert_{X^{\frac{1}{4}}}^2 + \Vert v(t)\Vert_{X^{\frac{3}{4}}}^2 + \Vert v_t(t)\Vert_{X^{\frac{1}{4}}}^2 \\
\leq{}&\frac{c_1|\Omega|}{\eta}(t-r) + \frac{c_2}{\eta} \int_{r}^{t} \Vert W(s)\Vert_{Y_0}^{2\rho} ds + \Vert u(r)\Vert_{X^{\frac{3}{4}}}^2 + \Vert u(r)\Vert_{X^{\frac{1}{4}}}^2 + \Vert u_t(r)\Vert_{X^{\frac{1}{4}}}^2 \\
&+ \Vert v(r)\Vert_{X^{\frac{3}{4}}}^2 + \Vert v_t(r)\Vert_{X^{\frac{1}{4}}}^2,
\end{aligned}
\]
consequently,
\begin{equation}\label{comp estimate6}
\begin{aligned}
&\Vert u(t)\Vert_{X^{\frac{3}{4}}}^2 + \Vert u_t(t)\Vert_{X^{\frac{1}{4}}}^2 + \Vert v(t)\Vert_{X^{\frac{3}{4}}}^2 + \Vert v_t(t)\Vert_{X^{\frac{1}{4}}}^2 \\
\leq{}&\frac{c_1|\Omega|}{\eta}(t-r) + \frac{c_2}{\eta} \int_{r}^{t} \Vert W(s)\Vert_{Y_0}^{2\rho} ds + \tilde{k} \Vert W(r)\Vert_{Y_0}^{2} + \Vert u(r)\Vert_{X^{\frac{3}{4}}}^2 \\
&+ \Vert u_t(r)\Vert_{X^{\frac{1}{4}}}^2 + \Vert v(r)\Vert_{X^{\frac{3}{4}}}^2 + \Vert v_t(r)\Vert_{X^{\frac{1}{4}}}^2,
\end{aligned}
\end{equation}
where we have used the embedding $X^{\frac{1}{2}} \hookrightarrow X^{\frac{1}{4}}$, i.e., $\Vert u(r)\Vert_{X^{\frac{1}{4}}}^2 \leq \tilde{k} \Vert u(r)\Vert_{X^{\frac{1}{2}}}^2$.

Now, by integrating inequality $(\ref{comp estimate6})$, with respect to $r$, from $\tau$ to $t$, we obtain
\begin{equation}\label{comp estimate7}
\begin{aligned}
&(t-\tau) \left( \Vert u(t)\Vert_{X^{\frac{3}{4}}}^2 + \Vert u_t(t)\Vert_{X^{\frac{1}{4}}}^2 + \Vert v(t)\Vert_{X^{\frac{3}{4}}}^2 + \Vert v_t(t)\Vert_{X^{\frac{1}{4}}}^2 \right) \\
\leq{}&\frac{c_1|\Omega|}{2\eta}(t-\tau)^2 + \frac{c_2}{\eta} \int_{\tau}^{t} \int_{r}^{t} \Vert W(s)\Vert_{Y_0}^{2\rho} ds dr + \tilde{k} \int_{\tau}^{t} \Vert W(r)\Vert_{Y_0}^{2} dr \\
&+ \int_{\tau}^{t} \Vert u(r)\Vert_{X^{\frac{3}{4}}}^2 dr + \int_{\tau}^{t} \Vert u_t(r)\Vert_{X^{\frac{1}{4}}}^2 dr \\
&+ \int_{\tau}^{t} \Vert v(r)\Vert_{X^{\frac{3}{4}}}^2 dr + \int_{\tau}^{t} \Vert v_t(r)\Vert_{X^{\frac{1}{4}}}^2 dr.
\end{aligned}
\end{equation}

Combining the inequalities obtained in $(\ref{comp estimate4})$, $(\ref{comp estimate5})$ and $(\ref{comp estimate7})$, we get
\begin{equation}\label{comp estimate8}
\begin{aligned}
&\Vert u(t)\Vert_{X^{\frac{3}{4}}}^2 + \Vert u_t(t)\Vert_{X^{\frac{1}{4}}}^2 + \Vert v(t)\Vert_{X^{\frac{3}{4}}}^2 + \Vert v_t(t)\Vert_{X^{\frac{1}{4}}}^2 \\
\leq{}&\frac{\tilde{c}^{2}c_1|\Omega|(3 + a_1^2)}{\eta^2} + c_1|\Omega| + \frac{c_1|\Omega|}{2\eta}(t-\tau) + \frac{c_2}{\eta(t-\tau)} \int_{\tau}^{t} \int_{r}^{t} \Vert W(s)\Vert_{Y_0}^{2\rho} ds dr \\
&+ \frac{1}{t-\tau} \left( \frac{\tilde{c}^{2}c_2(3 + a_1^2)}{\eta^2} + c_2 \right) \int_{\tau}^{t} \Vert W(r)\Vert_{Y_0}^{2\rho} dr + \frac{\tilde{k} + 1}{t-\tau} \int_{\tau}^{t} \Vert W(r)\Vert_{Y_0}^{2} dr \\
&+ \frac{1}{t-\tau} \Vert W(t)\Vert_{Y_0}^{2} + \frac{(1 + \lambda_1^{-1})(3 + a_1^2) + \eta(1 + \eta)}{\eta(t-\tau)} \Vert W(\tau)\Vert_{Y_0}^{2}.
\end{aligned}
\end{equation}

Now, if the global solution $W(t) = (u(t), u_t(t), v(t), v_t(t))$ of the problem $(\ref{edp01})-(\ref{cond01})$ starts in a bounded subset $B$ of $Y_0$, then
\begin{equation*}\label{comp estimate9}
\Vert W(\tau)\Vert_{Y_0} \leq M
\end{equation*}
for some positive constant $M$. Moreover, remember that from Theorem \ref{the solution is exponentially dominatedP2} there exist positive constants $\sigma = \sigma(B)$, $K_1 = K_1(B)$ and $K_2= K_2(B)$ such that
\begin{equation*}\label{comp estimate10}
\Vert W(t)\Vert_{Y_0}^2 \leq K_1e^{-\sigma(t - \tau)} + K_2, \quad  t \geq \tau.
\end{equation*}
With this, we can handle  with the three integrals that appear on the right hand side of inequality $(\ref{comp estimate8})$. In fact, first note that
\begin{equation*}
\begin{split}
\int_{\tau}^{t} \Vert W(r)\Vert_{Y_0}^{2} dr &\leq \int_{\tau}^{t} [K_1e^{-\sigma(r - \tau)} + K_2] dr \leq \frac{K_1}{\sigma} + K_2(t - \tau)
\end{split}
\end{equation*}
and
\begin{equation*}
\begin{split}
\int_{\tau}^{t} \Vert W(r)\Vert_{Y_0}^{2\rho} dr &\leq \int_{\tau}^{t} [\tilde{K_1} e^{-\rho\sigma(r - \tau)} + \tilde{K_2}] dr  \leq \frac{\tilde{K_1}}{\rho\sigma} + \tilde{K_2}(t - \tau),
\end{split}
\end{equation*}
where $\tilde{K_1}, \tilde{K_2}$ are positive constants. For the last integral remaining, note that
\[
\begin{split}
\int_{r}^{t} \Vert W(s)\Vert_{Y_0}^{2\rho} ds &\leq \int_{r}^{t} \left[ \tilde{\tilde{K_1}} e^{-\rho\sigma(s - \tau)} + \tilde{\tilde{K_2}} \right] ds \leq \frac{\tilde{\tilde{K_1}}}{\rho\sigma} e^{-\rho\sigma(r - \tau)} + \tilde{\tilde{K_2}}(t - r),
\end{split}
\]
for positive constants $\tilde{\tilde{K_1}}$ and $\tilde{\tilde{K_2}}$, and then it follows that
\begin{equation}\label{comp estimate13}
\begin{split}
\int_{\tau}^{t} \int_{r}^{t} \Vert W(s)\Vert_{Y_0}^{2\rho} ds dr &\leq \int_{\tau}^{t} \left[ \frac{\tilde{\tilde{K_1}}}{\rho\sigma} e^{-\rho\sigma(r - \tau)} + \tilde{\tilde{K_2}}(t - r) \right] dr  \\
&\leq \frac{\tilde{\tilde{K_1}}}{(\rho\sigma)^2} + \frac{\tilde{\tilde{K_2}}}{2}(t - \tau)^2.
\end{split}
\end{equation}

Finally, combining all the estimates in $(\ref{comp estimate8})-(\ref{comp estimate13})$, we conclude that there exist positive constants $k_1, k_2, k_3, k_4, k_5$ such that
\begin{equation}\label{comp estimate14}
\begin{aligned}
&\Vert u(t)\Vert_{X^{\frac{3}{4}}}^2 + \Vert u_t(t)\Vert_{X^{\frac{1}{4}}}^2 + \Vert v(t)\Vert_{X^{\frac{3}{4}}}^2 + \Vert v_t(t)\Vert_{X^{\frac{1}{4}}}^2 \\
&\leq k_{1} + k_{2}(t - \tau) + \frac{1}{t - \tau}[k_{3} e^{-k_{4} (t - \tau)} + k_{5}].
\end{aligned}
\end{equation}

Hence, $S(t, \tau)B$ is bounded in $X^{\frac{3}{4}} \times X^{\frac{1}{4}} \times X^{\frac{3}{4}} \times X^{\frac{1}{4}}$. Since $X^{\frac{3}{4}} \times X^{\frac{1}{4}} \times X^{\frac{3}{4}} \times X^{\frac{1}{4}} \hookrightarrow Y_0$,  and this embedding is compact, we conclude that $S(t, \tau)\colon  Y_0 \to Y_0$, given in $(\ref{evolution process of the problem})$, is compact for each $t>\tau$.
\end{proof}

We end this section with the proof of Theorem \ref{teo-pullback}.

\medskip

\noindent {\bf Proof of Theorem \ref{teo-pullback}:}  Theorem \ref{the solution is exponentially dominated} assures that the evolution process $S(t, \tau)\colon  Y_0 \to Y_0$ given by \eqref{evolution process of the problem} is  pullback strongly bounded dissipative. Additionally, it follows by Proposition \ref{proc-c} that  $S(t, \tau)\colon  Y_0 \to Y_0$ is compact, and, consequently,  it is pullback asymptotically compact. Now the result is a simple consequence of Theorem \ref{existence of the pullback attractor}. \qed

\section{Regularity of the pullback attractor}\label{RegA}

The purpose of this section is to show that the regularity of the pullback attractor can be improved, using energy estimates and progressive increases of regularity.

\vspace{.4cm}

\noindent {\bf Proof of Theorem \ref{RegPA}:} 
Let $\xi\colon  \mathbb{R}\to Y_0$ be a bounded global solution for the system \eqref{edp01}.  Since
$\bigcup\limits_{t \in \mathbb{R}}\mathbb{A}(t)$ is bounded in $Y_0$ (see Theorem \ref{teo-pullback}), we have $\{ \xi(t)\colon  t\in\mathbb{R} \}$ is a bounded subset of $Y_0$ by Theorem \ref{globalsolution}. Moreover, $\xi(\cdot) = (\mu(\cdot), \mu_t(\cdot), \nu(\cdot), \nu_t(\cdot))\colon  \mathbb{R}\to Y_0$ is such that $\xi(t) \in \mathbb{A}(t)$ for all $t\in\mathbb{R}$, and by \eqref{process_formulation},
$$
\xi(t) = L(t, \tau)\xi(\tau) + \int_{\tau}^t L(t, s)F(\xi(s)) ds, \quad t \geq \tau.
$$

Using the decay of $L(\cdot, \cdot)$, which was established in Theorem \ref{the solution decays}, and letting $\tau \rightarrow -\infty$, we get
$$
\xi(t) = \int_{-\infty}^t L(t, s)F(\xi(s)) ds, \quad t \in \mathbb{R}.
$$

Now, for $\tau\in\mathbb{R}$ fixed, we write $W_0 = (u_0,u_1,v_0,v_1) = \xi(\tau)$ and consider
$$
(u(t), u_t(t),  v(t),  v_t(t)) = U(t, \tau) W_0 = \int_{\tau}^t L(t, s)F(S(s, \tau)W_0) ds,
$$
where $U(\cdot, \cdot)$ is defined as in $(\ref{parte compacta do processo})$. Note that $(u(\cdot), v(\cdot))$ solves the system
\begin{equation} \label{edp05} 
\begin{cases}
u_{tt} - \Delta u + u + \eta(-\Delta)^{\frac{1}{2}}u_t + a_{\epsilon}(t)(-\Delta)^{\frac{1}{2}} v_t = f(u(t, \tau; u_0)), &\!\! (x, t) \in \Omega \times (\tau, \infty), \\
v_{tt} - \Delta v + \eta(-\Delta)^{\frac{1}{2}} v_t - a_{\epsilon}(t)(-\Delta)^{\frac{1}{2}}u_t = 0, &\!\! (x, t) \in \Omega \times (\tau, \infty),
\end{cases}
\end{equation}
with
\begin{equation}\label{cond03}
u(\tau, x) = 0, \  v(\tau, x) = 0, \ x\in\Omega.
\end{equation}

To estimate the solution of $(\ref{edp05})-(\ref{cond03})$ for $(u_0, u_1, v_0,  v_1)$ in a bounded subset $B \subset Y_0$, we again consider the maps
\[
\mathcal{E} (t) = \frac{1}{2}\Vert u(t)\Vert_{X^{\frac{1}{2}}}^2 + \frac{1}{2}\Vert u(t)\Vert_X^2 + \frac{1}{2}\Vert u_t(t)\Vert_X^2 + \frac{1}{2}\Vert v(t)\Vert_{X^{\frac{1}{2}}}^2 + \frac{1}{2}\Vert v_t(t)\Vert_X^2 - \int_{\Omega}\int_0^{ u(t)} f(s)dsdx,
\]
and
\[
\begin{split}
\mathcal{L}(t) &= \frac{1}{2}\Vert u(t)\Vert_{X^{\frac{1}{2}}}^2 + \frac{1}{2}\Vert u(t)\Vert_X^2 + \frac{1}{2}\Vert u_t(t)\Vert_X^2 + \frac{1}{2}\Vert v(t)\Vert_{X^{\frac{1}{2}}}^2 + \frac{1}{2}\Vert v_t(t)\Vert_X^2 \\
&\quad + \gamma_1\langle u(t),  u_t(t)\rangle_X + \gamma_2\langle v(t),  v_t(t)\rangle_X
\end{split}
\]
with $\gamma_1, \gamma_2 \in \mathbb{R}^+$. Using \eqref{edp05}, we can write
\[
\begin{split}
\frac{d}{dt}\mathcal{L} (t) & = \frac{d}{dt} \left( \mathcal{E}(t) + \gamma_1\langle u,  u_t\rangle_X + \gamma_2\langle v,  v_t\rangle_X + \int_{\Omega}\int_0^{ u} f(s)dsdx \right) \\
& = - \eta\Vert u_t\Vert_{X^{\frac{1}{4}}}^2 - \eta\Vert v_t\Vert_{X^{\frac{1}{4}}}^2 + \gamma_1\Vert u_t\Vert_X^2 - \gamma_1(\Vert u\Vert_{X^{\frac{1}{2}}}^2 + \Vert u\Vert_X^2) - \gamma_1\eta\langle A^{\frac{1}{2}} u,  u_t\rangle_X \\
& - \gamma_1a_{\epsilon}(t)\langle A^{\frac{1}{2}} u,  v_t\rangle_X 
+ \gamma_1\langle u, f( u)\rangle_X + \gamma_2\Vert v_t\Vert_X^2 - \gamma_2\Vert v\Vert_{X^{\frac{1}{2}}}^2 - \gamma_2\eta\langle A^{\frac{1}{2}} v,  v_t\rangle_X \\
& + \gamma_2a_{\epsilon}(t)\langle A^{\frac{1}{2}} v,  u_t\rangle_X + \langle f( u),  u_t\rangle_X.
\end{split}
\]

In the first place, let's deal with the nonlinearity $f$. By Lemma \ref{Lem_Aux_Int}, it follows that for each $\delta > 0$, there exists a constant $C_{\delta} > 0$ such that
$$
\int_{\Omega} f( u) u dx \leq \delta\Vert u\Vert_X^2 + C_{\delta}.
$$

Further, once the condition $1 < \frac{n-1}{n-2} \leq \rho < \frac{n}{n-2}$ implies $X^{\frac{1}{2}} \hookrightarrow L^{2\rho}(\Omega)$, and using again Lemma \ref{Lem_Aux_Int}, condition $(i)$, we have
\[
\begin{split}
\Vert f( u) \Vert_X & \leq \left( \int_{\Omega} [c(1 + | u|^{\rho})]^2 dx \right)^{\frac{1}{2}}  \leq \tilde{c} \left( |\Omega| + \int_{\Omega} | u|^{2\rho} dx \right)^{\frac{1}{2}} \\
& \leq \tilde{\tilde{c}} \left( |\Omega|^{\frac{1}{2}} + \Vert u\Vert_{L^{2\rho}(\Omega)}^{\rho} \right)  \leq \overline{c}\Vert u\Vert_{X^{\frac{1}{2}}}^{\rho} + \tilde{\tilde{c}}|\Omega|^{\frac{1}{2}}  \leq \overline{c}r^{\rho} + \tilde{\tilde{c}}|\Omega|^{\frac{1}{2}}  = \overline{C},
\end{split}
\]
whenever $\Vert u\Vert_{X^{\frac{1}{2}}} \leq r$.

Hence, using the Poincaré and Young Inequalities, one can obtain
\[
\begin{split}
& \frac{d}{dt}\mathcal{L}(t)  \leq - \eta\frac{1}{c^2}\Vert u_t\Vert_X^2 - \eta\frac{1}{c^2}\Vert v_t\Vert_X^2 + \gamma_1\Vert u_t\Vert_X^2 - \gamma_1\Vert u\Vert_{X^{\frac{1}{2}}}^2 + \gamma_1\eta\left( \frac{\epsilon_1}{2}\Vert u\Vert_{X^{\frac{1}{2}}}^2 + \frac{1}{2\epsilon_1}\Vert u_t\Vert_X^2 \right) \\
& + \gamma_1a_1\left( \frac{1}{2\epsilon_2}\Vert u\Vert_{X^{\frac{1}{2}}}^2 + \frac{\epsilon_2}{2}\Vert v_t\Vert_X^2 \right) + \gamma_1\delta\lambda_1^{-1}\Vert u\Vert_{X^{\frac{1}{2}}}^2 + \gamma_1C_{\delta} + \gamma_2\Vert v_t\Vert_X^2 - \gamma_2\Vert v\Vert_{X^{\frac{1}{2}}}^2 \\
& + \gamma_2a_1\left( \frac{1}{2\epsilon_3}\Vert v\Vert_{X^{\frac{1}{2}}}^2 + \frac{\epsilon_3}{2}\Vert u_t\Vert_X^2 \right) + \gamma_2\eta\left( \frac{\epsilon_4}{2}\Vert v\Vert_{X^{\frac{1}{2}}}^2 + \frac{1}{2\epsilon_4}\Vert v_t\Vert_X^2 \right) + \frac{1}{2\epsilon_5}\Vert f( u)\Vert_X^2 + \frac{\epsilon_5}{2}\Vert u_t\Vert_X^2 \\
&\leq -\gamma_1\left( 1 - \delta\lambda_1^{-1} - \eta\frac{\epsilon_1}{2} - a_1\frac{1}{2\epsilon_2} \right)\Vert u\Vert_{X^{\frac{1}{2}}}^2 - \left( \eta\frac{1}{c^2} - \gamma_1 - \gamma_1\eta\frac{1}{2\epsilon_1} - \gamma_2a_1\frac{\epsilon_3}{2} - \frac{\epsilon_5}{2} \right)\Vert u_t\Vert_X^2 \\
& - \gamma_2\left( 1 - a_1\frac{1}{2\epsilon_3} - \eta\frac{\epsilon_4}{2} \right)\Vert v\Vert_{X^{\frac{1}{2}}}^2 - \left( \eta\frac{1}{c^2} - \gamma_2 - \gamma_1a_1\frac{\epsilon_2}{2} - \gamma_2\eta\frac{1}{2\epsilon_4} \right)\Vert v_t\Vert_X^2 + \frac{1}{2\epsilon_5}\overline{C}^2 + \gamma_1C_{\delta}
\end{split}
\]
for all $\epsilon_1, \epsilon_2, \epsilon_3, \epsilon_4, \epsilon_5 > 0,$ where $c > 0$ is the embedding constant of $X^{\frac{1}{4}} \hookrightarrow X$. Choosing
$$
\delta = \frac{\lambda_1}{8}, \ \epsilon_1 = \frac{1}{\eta}, \ \epsilon_2 = 2a_1, \ \epsilon_3 = 2a_1, \ \epsilon_4 = \frac{1}{\eta}, \ \epsilon_5 = \frac{\eta}{c^2},
$$
it follows that
\[
\begin{split}
\frac{d}{dt}\mathcal{L}(t) &\leq -\frac{1}{8}\gamma_1\Vert u\Vert_{X^{\frac{1}{2}}}^2 - \left( \frac{\eta}{2c^2} - \gamma_1\left(1 + \frac{\eta^2}{2}\right) - \gamma_2a_1^2 \right)\Vert u_t\Vert_X^2 - \frac{1}{4}\gamma_2\Vert v\Vert_{X^{\frac{1}{2}}}^2 \\
&\quad -\left( \frac{\eta}{2c^2} - \gamma_1a_1^2 - \gamma_2\left(1 + \frac{\eta^2}{2}\right) \right)\Vert v_t\Vert_X^2 + \frac{c^2}{2\eta}\overline{C}^2 + \gamma_1C_{\frac{\lambda_1}{8}}.
\end{split}
\]

Now, note that one can take $\gamma_i > 0, i = 1, 2,$ sufficiently small such that
$$
\gamma_i < \frac{\eta}{8c^2}\min\left\{ \frac{1}{a_1^2}, \ \left(1 + \frac{\eta^2}{2}\right)^{-1} \right\}, \ i = 1, 2.
$$

Setting
$$
C_1 = \min\left\{ \frac{1}{8}\gamma_1, \frac{\eta}{2c^2} - \gamma_1\left(1 + \frac{\eta^2}{2}\right) - \gamma_2a_1^2, \ \frac{1}{4}\gamma_2, \ \frac{\eta}{2c^2} - \gamma_1a_1^2 - \gamma_2\left(1 + \frac{\eta^2}{2}\right) \right\} > 0
$$
and $C_2 =  \frac{c^2}{2\eta}\overline{C}^2 + \gamma_1C_{\frac{\lambda_1}{8}} > 0$, we obtain
\[
\frac{d}{dt}\mathcal{L}(t)  \leq -C_1 \Vert ( u,  u_t,  v,  v_t) \Vert_{Y_0}^2 + C_2.
\]

Using \eqref{Eqneeded} and \eqref{Eqneeded1}, we get 
\[
\frac{1}{4} \Vert ( u,  u_t,  v,  v_t) \Vert_{Y_0}^2 \leq \mathcal{L}(t)  \leq  \frac{3(1 + \lambda_1^{-1})}{4}\Vert ( u,  u_t,  v,  v_t) \Vert_{Y_0}^2,
\]
and putting $C_3 = C_1\left(\frac{3(1 + \lambda_1^{-1})}{4}\right)^{-1}$, one has
\[
\frac{1}{4} \Vert ( u,  u_t,  v,  v_t) \Vert_{Y_0}^2 \leq \mathcal{L}(t)  \leq \mathcal{L}(\tau)e^{-C_3(t - \tau)} + \frac{C_2}{C_3}, \quad t \geq \tau. 
\]

From this, we obtain
\[
\bigcup\limits_{\tau \leq s \leq t} U(s, \tau)B \ \text{is a bounded subset of} \ Y_0.
\]

On the other hand, note that $(\phi, \varphi) = (u_t, v_t)$ solves the system
\begin{equation}\label{edp07}
\begin{cases}
\phi_{tt} - \Delta\phi + \phi + \eta(-\Delta)^{\frac{1}{2}}\phi_t + a_{\epsilon}(t)(-\Delta)^{\frac{1}{2}}\varphi_t + a_{\epsilon}^{\prime}(t)(-\Delta)^{\frac{1}{2}}\varphi = f^{\prime}( u)\phi, \\
\varphi_{tt} - \Delta\varphi + \eta(-\Delta)^{\frac{1}{2}}\varphi_t - a_{\epsilon}(t)(-\Delta)^{\frac{1}{2}}\phi_t - a_{\epsilon}^{\prime}(t)(-\Delta)^{\frac{1}{2}}\phi = 0.
\end{cases}
\end{equation}

We want to estimate $(\phi, \phi_t, \varphi, \varphi_t)$ in $Y_0$, but our solutions are not regular enough for this to be done in a direct way. Thus, instead, the process will be done by progressive increases of regularity. For $\alpha > 0$, let us consider the fractional power spaces $X^{\alpha} = D(A^{\alpha})$ endowed with the graph norm, and let $X^{-\alpha} = (X^{\alpha})^{\prime}$. For
$$
(\phi, \phi_t, \varphi, \varphi_t) \in X^{\frac{1-\alpha}{2}} \times X^{-\frac{\alpha}{2}} \times X^{\frac{1-\alpha}{2}} \times X^{-\frac{\alpha}{2}},
$$
let us define
\[
\begin{split}
\mathcal{L}_{\alpha}(t) &= \frac{1}{2}\left( \Vert\phi(t)\Vert_{X^{\frac{1-\alpha}{2}}}^2 + \Vert\phi(t)\Vert_{X^{-\frac{\alpha}{2}}}^2 + \Vert\phi_t(t)\Vert_{X^{-\frac{\alpha}{2}}}^2 + \Vert\varphi(t)\Vert_{X^{\frac{1-\alpha}{2}}}^2 + \Vert\varphi_t(t)\Vert_{X^{-\frac{\alpha}{2}}}^2 \right) \\
&\quad + \gamma_1\langle \phi(t), \phi_t(t) \rangle_{X^{-\frac{\alpha}{2}}} + \gamma_2\langle \varphi(t), \varphi_t(t) \rangle_{X^{-\frac{\alpha}{2}}},
\end{split}
\]
with $\gamma_1, \gamma_2 \in \mathbb{R}^+$. Using $(\ref{edp07})$, we obtain
\begin{equation*}\label{derivada funcional regularidade}
\begin{split}
\frac{d}{dt} \mathcal{L}_{\alpha}(t) & = \langle\phi_t, \phi\rangle_{X^{\frac{1-\alpha}{2}}} + \langle\phi_t, \phi\rangle_{X^{-\frac{\alpha}{2}}} + \langle\phi_{tt}, \phi_t\rangle_{X^{-\frac{\alpha}{2}}} + \langle\varphi_t, \varphi\rangle_{X^{\frac{1-\alpha}{2}}} + \langle\varphi_{tt}, \varphi_t\rangle_{X^{-\frac{\alpha}{2}}} \\
& + \gamma_1\langle\phi_t, \phi_t\rangle_{X^{-\frac{\alpha}{2}}} + \gamma_1\langle\phi, \phi_{tt}\rangle_{X^{-\frac{\alpha}{2}}} + \gamma_2\langle\varphi_t, \varphi_t\rangle_{X^{-\frac{\alpha}{2}}} + \gamma_2\langle\varphi, \varphi_{tt}\rangle_{X^{-\frac{\alpha}{2}}} \\
\end{split}
\end{equation*}

\begin{equation*}\label{derivada funcional regularidade}
\begin{split}
&= \gamma_1\Vert\phi_t\Vert_{X^{-\frac{\alpha}{2}}}^2 + \gamma_2\Vert\varphi_t\Vert_{X^{-\frac{\alpha}{2}}}^2 - \gamma_1\Vert\phi\Vert_{X^{-\frac{\alpha}{2}}}^2 - \gamma_1\Vert\phi\Vert_{X^{\frac{1-\alpha}{2}}}^2 - \gamma_2\Vert\varphi\Vert_{X^{\frac{1-\alpha}{2}}}^2 - \eta\Vert\phi_t\Vert_{X^{\frac{1-2\alpha}{4}}}^2 \\
& -\eta\Vert\varphi_t\Vert_{X^{\frac{1-2\alpha}{4}}}^2 - \eta\gamma_1\langle\phi,A^{\frac{1}{2}}\phi_t\rangle_{X^{-\frac{\alpha}{2}}} - a_{\epsilon}(t)\gamma_1\langle\phi,A^{\frac{1}{2}}\varphi_t\rangle_{X^{-\frac{\alpha}{2}}} - \eta\gamma_2\langle\varphi,A^{\frac{1}{2}}\varphi_t\rangle_{X^{-\frac{\alpha}{2}}} \\
& + a_{\epsilon}(t)\gamma_2\langle\varphi,A^{\frac{1}{2}}\phi_t\rangle_{X^{-\frac{\alpha}{2}}} - a_{\epsilon}^{\prime}(t)\langle A^{\frac{1}{2}}\varphi,\phi_t\rangle_{X^{-\frac{\alpha}{2}}} + a_{\epsilon}^{\prime}(t)\langle A^{\frac{1}{2}}\phi,\varphi_t\rangle_{X^{-\frac{\alpha}{2}}} - a_{\epsilon}^{\prime}(t)\gamma_1\langle\phi,A^{\frac{1}{2}}\varphi\rangle_{X^{-\frac{\alpha}{2}}} \\
& + a_{\epsilon}^{\prime}(t)\gamma_2\langle\varphi,A^{\frac{1}{2}}\phi\rangle_{X^{-\frac{\alpha}{2}}} + \gamma_1\langle\phi,f^{\prime}( u)\phi\rangle_{X^{-\frac{\alpha}{2}}} + \langle f^{\prime}( u)\phi,\phi_t\rangle_{X^{-\frac{\alpha}{2}}}.
\end{split}
\end{equation*}

Next, we shall estimate  the terms that appear on the right hand side of the above expression, beginning with those in which the nonlinearity $f^{\prime}$ is explicit. To do this, consider
\[
\alpha_1 =  \frac{(\rho - 1)(n-2)}{2}.
\]
Since $\frac{n-1}{n-2} \leq \rho < \frac{n}{n-2}$, we have $\frac{1}{2} \leq \alpha_1 < 1$.

Noticing that
\begin{equation}\label{estimate0}
\langle f^{\prime}( u)\phi,\phi_t\rangle_{X^{-\frac{\alpha}{2}}} \leq \Vert f^{\prime}( u)\phi\Vert_{X^{-\frac{\alpha}{2}}}\Vert\phi_t\Vert_{X^{-\frac{\alpha}{2}}}
\end{equation}
and that the embedding $X^{\frac{\alpha}{2}} = H^{\alpha}(\Omega) \hookrightarrow L^p(\Omega)$ or, equivalently, $L^{\frac{p}{p-1}}(\Omega) \hookrightarrow X^{-\frac{\alpha}{2}}$, holds for any $2\leq p\leq\frac{2n}{n-2\alpha}$, one can obtain an estimate for the term $\Vert f^{\prime}( u)\phi\Vert_{X^{-\frac{\alpha}{2}}}$ using Hölder's Inequality and the growth condition, in the following way:
\[
\begin{split}
\Vert f^{\prime}( u)\phi \Vert_{X^{-\frac{\alpha}{2}}} & \leq c_1\Vert f^{\prime}( u)\phi\Vert_{L^{\frac{2n}{n+2\alpha}}(\Omega)} \leq c_1\Vert\phi\Vert_{L^2(\Omega)}\Vert f^{\prime}( u)\Vert_{L^{\frac{n}{\alpha}}(\Omega)} \\
& \leq c_1\Vert\phi\Vert_X \left( \int_{\Omega} [c(1 + | u|^{\rho - 1})]^{\frac{n}{\alpha}} dx \right)^{\frac{\alpha}{n}} \leq c_2\Vert\phi\Vert_X \left( |\Omega| + \int_{\Omega} | u|^{\frac{(\rho - 1)n}{\alpha}} dx \right)^{\frac{\alpha}{n}} \\
&\leq c_3\Vert\phi\Vert_X \left( 1 + \Vert u\Vert_{L^{\frac{(\rho - 1)n}{\alpha}}(\Omega)}^{\rho - 1} \right).
\end{split}
\]

Now, once the embedding $H^1(\Omega) \hookrightarrow L^{\frac{(\rho - 1)n}{\alpha}}(\Omega)$ holds, if and only if $\alpha\geq\frac{(\rho - 1)(n-2)}{2}$ and $\frac{(\rho - 1)n}{\alpha} \geq 2$, that is, $\frac{(\rho - 1)(n-2)}{2} \leq \alpha \leq \frac{(\rho - 1)n}{2}$,  then for $\alpha = \alpha_1$ we have
\begin{equation}\label{estimate1}
\begin{split}
\Vert f^{\prime}( u)\phi\Vert_{X^{-\frac{\alpha_{1}}{2}}} &\leq c_3\Vert\phi\Vert_X \left( 1 + \Vert u\Vert_{L^{\frac{(\rho - 1)n}{\alpha_{1}}}(\Omega)}^{\rho - 1} \right) \leq c_5\Vert\phi\Vert_X \left( 1 + \Vert u\Vert_{X^{\frac{1}{2}}}^{\rho - 1} \right) \leq c_6,
\end{split}
\end{equation}
since $ u$ and $\phi$ remain in bounded subsets of $X^{\frac{1}{2}}$ and $X$, respectively. Hence, from Young's inequality, and using $(\ref{estimate0})$ and $(\ref{estimate1})$, we get
\begin{equation*}\label{estimate2}
\begin{split}
\langle f^{\prime}( u)\phi,\phi_t\rangle_{X^{-\frac{\alpha_{1}}{2}}} & \leq \frac{1}{2\epsilon_0}\Vert f^{\prime}( u)\phi\Vert_{X^{-\frac{\alpha_{1}}{2}}}^2 + \frac{\epsilon_0}{2}\Vert\phi_t\Vert_{X^{-\frac{\alpha_{1}}{2}}}^2\\
&  \leq \frac{1}{2\epsilon_0}c_6^2 + \frac{\epsilon_0}{2}\Vert\phi_t\Vert_{X^{-\frac{\alpha_{1}}{2}}}^2
\end{split}
\end{equation*}
for all $\epsilon_0 > 0$. With this in mind, it is possible to obtain an estimate for the other term that has the nonlinearity $f^{\prime}$, that is,
\[
\begin{split}
\gamma_1\langle\phi, f^{\prime}( u)\phi\rangle_{X^{-\frac{\alpha_{1}}{2}}} 
\leq \frac{\epsilon_1}{2}\Vert\phi\Vert_{X^{-\frac{\alpha_{1}}{2}}}^2 + \frac{1}{2\epsilon_1}\gamma_1^2\Vert f^{\prime}( u)\phi\Vert_{X^{-\frac{\alpha_{1}}{2}}}^2 
\leq \frac{\epsilon_1}{2}\Vert\phi\Vert_{X^{-\frac{\alpha_{1}}{2}}}^2 + \frac{1}{2\epsilon_1}\gamma_1^2c_6^2
\end{split}
\]
for all $\epsilon_1 > 0$.

Next, from Cauchy-Schwartz and Young inequalities, we have
\[
\begin{split}
- \eta\gamma_1\langle\phi,A^{\frac{1}{2}}\phi_t\rangle_{X^{-\frac{\alpha_{1}}{2}}}   
\leq \eta\gamma_1\frac{\epsilon_2}{2}\Vert\phi\Vert_{X^{\frac{1-\alpha_{1}}{2}}}^2 + \eta\gamma_1\frac{1}{2\epsilon_2}\Vert\phi_t\Vert_{X^{-\frac{\alpha_{1}}{2}}}^2,
\end{split}
\]
\[
\begin{split}
- a_{\epsilon}(t)\gamma_1\langle\phi,A^{\frac{1}{2}}\varphi_t\rangle_{X^{-\frac{\alpha_{1}}{2}}}  
\leq a_1\gamma_1\frac{\epsilon_3}{2}\Vert\phi\Vert_{X^{\frac{1-\alpha_{1}}{2}}}^2 + a_1\gamma_1\frac{1}{2\epsilon_3}\Vert\varphi_t\Vert_{X^{-\frac{\alpha_{1}}{2}}}^2,
\end{split}
\]

\[
- \eta\gamma_2\langle\varphi,A^{\frac{1}{2}}\varphi_t\rangle_{X^{-\frac{\alpha_{1}}{2}}}  \leq \eta\gamma_2\frac{\epsilon_4}{2}\Vert\varphi\Vert_{X^{\frac{1-\alpha_{1}}{2}}}^2 + \eta\gamma_2\frac{1}{2\epsilon_4}\Vert\varphi_t\Vert_{X^{-\frac{\alpha_{1}}{2}}}^2
\]
and
\[
a_{\epsilon}(t)\gamma_2\langle\varphi,A^{\frac{1}{2}}\phi_t\rangle_{X^{-\frac{\alpha_{1}}{2}}} \leq a_1\gamma_2\frac{\epsilon_5}{2}\Vert\varphi\Vert_{X^{\frac{1-\alpha_{1}}{2}}}^2 + a_1\gamma_2\frac{1}{2\epsilon_5}\Vert\phi_t\Vert_{X^{-\frac{\alpha_{1}}{2}}}^2,
\]
for all $\epsilon_2 > 0$, $\epsilon_3 > 0$, $\epsilon_4 > 0$ and $\epsilon_5 > 0$.

Since $\frac{1}{2} \leq \alpha_{1} < 1$, we have the embedding $X \hookrightarrow X^{\frac{1-2\alpha_{1}}{4}}$, that is,
$$
\Vert\cdot\Vert_{X^{\frac{1-2\alpha_{1}}{4}}} \leq \tilde{c} \Vert\cdot\Vert_X
$$
for some constant $\tilde{c} > 0$. From this, and by condition \eqref{derivative-a-bounded}, and also using the fact that $\varphi$ remains in a bounded subset of $X$, we get
\begin{equation*}\label{estimate8}
\begin{split}
&- a_{\epsilon}^{\prime}(t)\langle A^{\frac{1}{2}}\varphi,\phi_t\rangle_{X^{-\frac{\alpha_{1}}{2}}} \leq b_0\Vert\varphi\Vert_{X^{\frac{1-2\alpha_{1}}{4}}}\Vert\phi_t\Vert_{X^{\frac{1-2\alpha_{1}}{4}}} \leq b_0\frac{1}{2\epsilon_6}\Vert\varphi\Vert_{X^{\frac{1-2\alpha_{1}}{4}}}^2 + b_0\frac{\epsilon_6}{2}\Vert\phi_t\Vert_{X^{\frac{1-2\alpha_{1}}{4}}}^2 \\
\leq{}& b_0\frac{1}{2\epsilon_6}\tilde{c}^2\Vert\varphi\Vert_X^2 + b_0\frac{\epsilon_6}{2}\Vert\phi_t\Vert_{X^{\frac{1-2\alpha_{1}}{4}}}^2  \leq \frac{1}{2\epsilon_6}b_0c_7 + b_0\frac{\epsilon_6}{2}\Vert\phi_t\Vert_{X^{\frac{1-2\alpha_{1}}{4}}}^2
\end{split}
\end{equation*}
for all $\epsilon_6 > 0$,
\begin{equation*}\label{estimate9}
\begin{split}
&a_{\epsilon}^{\prime}(t)\langle A^{\frac{1}{2}}\phi,\varphi_t\rangle_{X^{-\frac{\alpha_{1}}{2}}} \leq b_0\frac{1}{2\epsilon_7}\Vert\phi\Vert_{X^{\frac{1-2\alpha_{1}}{4}}}^2 + b_0\frac{\epsilon_7}{2}\Vert\varphi_t\Vert_{X^{\frac{1-2\alpha_{1}}{4}}}^2 \\
\leq{}& b_0\frac{1}{2\epsilon_7}\tilde{c}^2\Vert\phi\Vert_X^2 + b_0\frac{\epsilon_7}{2}\Vert\varphi_t\Vert_{X^{\frac{1-2\alpha_{1}}{4}}}^2 \leq \frac{1}{2\epsilon_7}b_0c_8 + b_0\frac{\epsilon_7}{2}\Vert\varphi_t\Vert_{X^{\frac{1-2\alpha_{1}}{4}}}^2
\end{split}
\end{equation*}
for all $\epsilon_7 > 0$,
\begin{equation*}\label{estimate10}
\begin{split}
&- a_{\epsilon}^{\prime}(t)\gamma_1\langle\phi,A^{\frac{1}{2}}\varphi\rangle_{X^{-\frac{\alpha_{1}}{2}}} \leq b_0\gamma_1\frac{\epsilon_8}{2}\Vert\phi\Vert_{X^{\frac{1-2\alpha_{1}}{4}}}^2 + b_0\gamma_1\frac{1}{2\epsilon_8}\Vert\varphi\Vert_{X^{\frac{1-2\alpha_{1}}{4}}}^2 \\
\leq{}& \frac{b_0\gamma_1\epsilon_8\tilde{c}^2}{2}\Vert\phi\Vert_X^2 + \frac{b_0\gamma_1\tilde{c}^2}{2\epsilon_8}\Vert\varphi\Vert_X^2  \leq c_9
\end{split}
\end{equation*}

\noindent and
\begin{equation*}\label{estimate11}
\begin{split}
&a_{\epsilon}^{\prime}(t)\gamma_2\langle\varphi,A^{\frac{1}{2}}\phi\rangle_{X^{-\frac{\alpha_{1}}{2}}}  \leq b_0\gamma_2\frac{\epsilon_9}{2}\Vert\varphi\Vert_{X^{\frac{1-2\alpha_{1}}{4}}}^2 + b_0\gamma_2\frac{1}{2\epsilon_9}\Vert\phi\Vert_{X^{\frac{1-2\alpha_{1}}{4}}}^2 \\
\leq{}& \frac{b_0\gamma_2\epsilon_9\tilde{c}^2}{2}\Vert\varphi\Vert_X^2 + \frac{b_0\gamma_2\tilde{c}^2}{2\epsilon_9}\Vert\phi\Vert_X^2  \leq c_{10}
\end{split}
\end{equation*}
for some constants $c_9 > 0$ and $c_{10} > 0$.

Finally, combining all the estimates obtained before, we get
\begin{equation*}\label{estimate12}
\begin{split}
\frac{d}{dt}\mathcal{L}_{\alpha_{1}} (t) & \leq - \left( \gamma_1 - \eta\gamma_1\frac{\epsilon_2}{2} - a_1\gamma_1\frac{\epsilon_3}{2} \right) \Vert\phi\Vert_{X^{\frac{1-\alpha_{1}}{2}}}^2 \\
&- \left( - \gamma_1 - \eta\gamma_1\frac{1}{2\epsilon_2} - a_1\gamma_2\frac{1}{2\epsilon_5} - \frac{\epsilon_0}{2} \right) \Vert\phi_t\Vert_{X^{-\frac{\alpha_{1}}{2}}}^2 \\
&- \left( \gamma_2 - \eta\gamma_2\frac{\epsilon_4}{2} - a_1\gamma_2\frac{\epsilon_5}{2} \right) \Vert\varphi\Vert_{X^{\frac{1-\alpha_{1}}{2}}}^2 \\
&- \left( - \gamma_2 - a_1\gamma_1\frac{1}{2\epsilon_3} - \eta\gamma_2\frac{1}{2\epsilon_4} \right) \Vert\varphi_t\Vert_{X^{-\frac{\alpha_{1}}{2}}}^2 \\
&- \left( \gamma_1 - \frac{\epsilon_1}{2} \right) \Vert\phi\Vert_{X^{-\frac{\alpha_{1}}{2}}}^2 + \left( b_0\frac{\epsilon_6}{2} - \eta \right) \Vert\phi_t\Vert_{X^{\frac{1-2\alpha_{1}}{4}}}^2 + \left( b_0\frac{\epsilon_7}{2} - \eta \right) \Vert\varphi_t\Vert_{X^{\frac{1-2\alpha_{1}}{4}}}^2 \\
&+ \frac{1}{2\epsilon_0}c_6^2 + \frac{1}{2\epsilon_1}\gamma_1^2c_6^2 + \frac{1}{2\epsilon_6}b_0c_7 + \frac{1}{2\epsilon_7}b_0c_8 + c_9 + c_{10}.
\end{split}
\end{equation*}

Now, by choosing $\epsilon_1 > 0$, $\epsilon_2 > 0$, $\epsilon_3 > 0$, $\epsilon_4 > 0$, $\epsilon_5 > 0$, $\epsilon_6 > 0$ and $\epsilon_7 > 0$, respectively, such that
$$
\epsilon_1 = 2\gamma_1, \ \epsilon_2 = \frac{1}{2\eta}, \ \epsilon_3 = \frac{1}{2a_1}, \ \epsilon_4 = \frac{1}{2\eta}, \ \epsilon_5 = \frac{1}{2a_1}, \ \epsilon_6 = \frac{\eta}{b_0} \ \text{and} \ \epsilon_7 = \frac{3\eta}{2b_0},
$$
we obtain
\begin{equation}\label{estimate13}
\begin{split}
\frac{d}{dt}\mathcal{L}_{\alpha_{1}} (t) & \leq -\frac{1}{2}\gamma_1\Vert\phi\Vert_{X^{\frac{1-\alpha_{1}}{2}}}^2 - \left( - \gamma_1 - \eta^2\gamma_1 - a_1^2\gamma_2 - \frac{\epsilon_0}{2} \right) \Vert\phi_t\Vert_{X^{-\frac{\alpha_{1}}{2}}}^2 - \frac{1}{2}\gamma_2\Vert\varphi\Vert_{X^{\frac{1-\alpha_{1}}{2}}}^2 \\
&- \left( - \gamma_2 - a_1^2\gamma_1 - \eta^2\gamma_2 \right) \Vert\varphi_t\Vert_{X^{-\frac{\alpha_{1}}{2}}}^2 - \frac{\eta}{2}\Vert\phi_t\Vert_{X^{\frac{1-2\alpha_{1}}{4}}}^2 - \frac{\eta}{4}\Vert\varphi_t\Vert_{X^{\frac{1-2\alpha_{1}}{4}}}^2 \\
& + \frac{1}{2\epsilon_0}c_6^2 + \frac{1}{4}\gamma_1c_6^2 + \frac{b_0^2 c_7 }{2\eta} + \frac{b_0^2 c_8 }{3\eta} +  c_9 + c_{10}.
\end{split}
\end{equation}

As $\frac{1-2\alpha_{1}}{4} = -\frac{\alpha_{1}}{2} +\frac{1}{4} > -\frac{\alpha_{1}}{2}$,  we have the embedding $X^{\frac{1-2\alpha_{1}}{4}} \hookrightarrow X^{-\frac{\alpha_{1}}{2}},$ and so
$$
\Vert\cdot\Vert_{X^{-\frac{\alpha_{1}}{2}}} \leq \tilde{\tilde{c}}\Vert\cdot\Vert_{X^{\frac{1-2\alpha_{1}}{4}}}
$$
for some constant $\tilde{\tilde{c}} > 0$, which implies
\begin{equation}\label{estimate14}
- \Vert\cdot\Vert_{X^{\frac{1-2\alpha_{1}}{4}}}^2 \leq - \frac{1}{\tilde{\tilde{c}}^{2}} \Vert\cdot\Vert_{X^{-\frac{\alpha_{1}}{2}}}^2.
\end{equation}

Hence, combining $(\ref{estimate13})$ and $(\ref{estimate14})$, we get
\begin{equation}\label{estimate15}
\begin{split}
\frac{d}{dt}\mathcal{L}_{\alpha_{1}} (t) & \leq -\frac{1}{2}\gamma_1\Vert\phi\Vert_{X^{\frac{1-\alpha_{1}}{2}}}^2 - \left( \frac{\eta}{2\tilde{\tilde{c}}^{2}} - \gamma_1 - \eta^2\gamma_1 - a_1^2\gamma_2 - \frac{\epsilon_0}{2} \right) \Vert\phi_t\Vert_{X^{-\frac{\alpha_{1}}{2}}}^2 \\
&- \frac{1}{2}\gamma_2\Vert\varphi\Vert_{X^{\frac{1-\alpha_{1}}{2}}}^2 - \left( \frac{\eta}{4\tilde{\tilde{c}}^{2}} - \gamma_2 - a_1^2\gamma_1 - \eta^2\gamma_2 \right) \Vert\varphi_t\Vert_{X^{-\frac{\alpha_{1}}{2}}}^2 \\
&+ \frac{1}{2\epsilon_0}c_6^2 +\frac{1}{4}\gamma_1c_6^2 + \frac{b_0^2 c_7 }{2\eta} + \frac{b_0^2 c_8 }{3\eta} +  c_9 + c_{10}.
\end{split}
\end{equation}

At last, choosing $\epsilon_0 > 0$ such that $\epsilon_0 = \frac{\eta}{2\tilde{\tilde{c}}^{2}}$, expression $(\ref{estimate15})$ turns into
\begin{equation*}\label{estimate16}
\begin{split}
\frac{d}{dt}\mathcal{L}_{\alpha_{1}} (t) & \leq -\frac{1}{2}\gamma_1\Vert\phi\Vert_{X^{\frac{1-\alpha_{1}}{2}}}^2 - \left( \frac{\eta}{4\tilde{\tilde{c}}^{2}} - (1 + \eta^2)\gamma_1 - a_1^2\gamma_2 \right) \Vert\phi_t\Vert_{X^{-\frac{\alpha_{1}}{2}}}^2 \\
&- \frac{1}{2}\gamma_2\Vert\varphi\Vert_{X^{\frac{1-\alpha_{1}}{2}}}^2 - \left( \frac{\eta}{4\tilde{\tilde{c}}^{2}} - a_1^2\gamma_1 - (1 + \eta^2)\gamma_2 \right) \Vert\varphi_t\Vert_{X^{-\frac{\alpha_{1}}{2}}}^2 \\
&+ \frac{\tilde{\tilde{c}}^{2}}{\eta}c_6^2 + \frac{1}{4}\gamma_1c_6^2 + \frac{b_0^2 c_7 }{2\eta} + \frac{b_0^2 c_8 }{3\eta} +  c_9 + c_{10}.
\end{split}
\end{equation*}

Now, taking $\gamma_i > 0, i = 1, 2,$ sufficiently small such that
$$
\gamma_i < \min \left\{ \frac{1}{2\tilde{k}}, \frac{\eta}{16\tilde{\tilde{c}}^{2}} \frac{1}{1 + \eta^2}, \frac{\eta}{16\tilde{\tilde{c}}^{2}} \frac{1}{a_1^2} \right\}, \ i = 1, 2,
$$
where $\tilde{k} > 0$ is the embedding constant of $X^{\frac{1-\alpha_{1}}{2}} \hookrightarrow X^{-\frac{\alpha_{1}}{2}}$ and taking
$$M_1 = \min \left\{ \frac{1}{2}\gamma_1, \ \frac{\eta}{4\tilde{\tilde{c}}^{2}} - (1 + \eta^2)\gamma_1 - a_1^2\gamma_2, \ \frac{1}{2}\gamma_2, \ \frac{\eta}{4\tilde{\tilde{c}}^{2}} - a_1^2\gamma_1 - (1 + \eta^2)\gamma_2 \right\} > 0$$
and
$M_2 = \frac{\tilde{\tilde{c}}^{2}}{\eta}c_6^2 + \frac{1}{4}\gamma_1c_6^2 + \frac{b_0^2 c_7 }{2\eta} + \frac{b_0^2 c_8 }{3\eta} +  c_9 + c_{10} > 0,$
it follows that
\begin{equation}\label{estimate17}
\frac{d}{dt}\mathcal{L}_{\alpha_{1}}(t) \leq -M_1 \left( \Vert\phi\Vert_{X^{\frac{1-\alpha_{1}}{2}}}^2 + \Vert\phi_t\Vert_{X^{-\frac{\alpha_{1}}{2}}}^2 + \Vert\varphi\Vert_{X^{\frac{1-\alpha_{1}}{2}}}^2 + \Vert\varphi_t\Vert_{X^{-\frac{\alpha_{1}}{2}}}^2 \right) + M_2.
\end{equation}

Observe that
\[
\begin{split}
| \gamma_1\langle \phi, \phi_t \rangle_{X^{-\frac{\alpha_{1}}{2}}} + \gamma_2\langle \varphi, \varphi_t \rangle_{X^{-\frac{\alpha_{1}}{2}}}| \leq  \frac{1}{4}\left( \Vert\phi\Vert_{X^{\frac{1-\alpha_{1}}{2}}}^2 + \Vert\phi_t\Vert_{X^{-\frac{\alpha_{1}}{2}}}^2 + \Vert\varphi\Vert_{X^{\frac{1-\alpha_{1}}{2}}}^2 + \Vert\varphi_t\Vert_{X^{-\frac{\alpha_{1}}{2}}}^2 \right).
\end{split}
\]

In this way, using a similar argument as in \eqref{Eqneeded} and \eqref{Eqneeded1}, we get
\[
\begin{split}
&\frac{1}{4} \left( \Vert\phi\Vert_{X^{\frac{1-\alpha_{1}}{2}}}^2 + \Vert\phi_t\Vert_{X^{-\frac{\alpha_{1}}{2}}}^2 + \Vert\varphi\Vert_{X^{\frac{1-\alpha_{1}}{2}}}^2 + \Vert\varphi_t\Vert_{X^{-\frac{\alpha_{1}}{2}}}^2 \right) \\
& \leq \mathcal{L}_{\alpha_{1}}(t) \leq \frac{3(1 + \tilde{k}^2)}{4} \left( \Vert\phi\Vert_{X^{\frac{1-\alpha_{1}}{2}}}^2 + \Vert\phi_t\Vert_{X^{-\frac{\alpha_{1}}{2}}}^2 + \Vert\varphi\Vert_{X^{\frac{1-\alpha_{1}}{2}}}^2 + \Vert\varphi_t\Vert_{X^{-\frac{\alpha_{1}}{2}}}^2 \right).
\end{split}
\]

This estimate together with \eqref{estimate17} implies that
\[
\Vert\phi\Vert_{X^{\frac{1-\alpha_{1}}{2}}}^2 + \Vert\phi_t\Vert_{X^{-\frac{\alpha_{1}}{2}}}^2 + \Vert\varphi\Vert_{X^{\frac{1-\alpha_{1}}{2}}}^2 + \Vert\varphi_t\Vert_{X^{-\frac{\alpha_{1}}{2}}}^2 \leq 4\mathcal{L}_{\alpha_{1}}(\tau)e^{-M_3(t -\tau)} + M_4, 
\]
with positive constants $M_3$ and $M_4$. This assures that $(\phi, \phi_t, \varphi, \varphi_t)$ is bounded in the space $X^{\frac{1-\alpha_1}{2}} \times X^{-\frac{\alpha_1}{2}} \times X^{\frac{1-\alpha_1}{2}}\times  X^{-\frac{\alpha_1}{2}}$. But we want to conclude that $\bigcup\limits_{t \in \mathbb{R}}\mathbb{A}(t)$ is bounded in $X^{\frac{2-\alpha_1}{2}} \times X^{\frac{1-\alpha_1}{2}} \times X^{\frac{2-\alpha_1}{2}} \times X^{\frac{1-\alpha_1}{2}}$. We already know that $u_t$ and $v_t$ are bounded in $X^{\frac{1-\alpha_1}{2}}$. Now, to show that $ u\in X^{\frac{2-\alpha_1}{2}}$ and it is bounded, it is enough to show that $\Vert A u\Vert_{X^{-\frac{\alpha_1}{2}}} \leq C_1$ for some constant $C_1 > 0$,
since
$$
\Vert A u\Vert_{X^{-\frac{\alpha_1}{2}}} = \Vert A^{\frac{2-\alpha_1}{2}} u\Vert_X = \Vert u\Vert_{X^{\frac{2-\alpha_1}{2}}}.
$$

Indeed, note that
\[
\begin{split}
&\Vert -A u\Vert_{X^{-\frac{\alpha_1}{2}}} - \Vert  u + \eta A^{\frac{1}{2}} u_t + a_{\epsilon}(t)A^{\frac{1}{2}} v_t - f( u) \Vert_{X^{-\frac{\alpha_1}{2}}} \\
& \leq \Vert -A u -  u - \eta A^{\frac{1}{2}} u_t - a_{\epsilon}(t)A^{\frac{1}{2}} v_t + f( u) \Vert_{X^{-\frac{\alpha_1}{2}}} \\
& = \Vert  u_{tt} \Vert_{X^{-\frac{\alpha_1}{2}}} = \Vert \phi_t \Vert_{X^{-\frac{\alpha_1}{2}}} \leq k_1,
\end{split}
\]
which yields
$$
\Vert A u\Vert_{X^{-\frac{\alpha_1}{2}}} \leq k_1 + \Vert u\Vert_{X^{-\frac{\alpha_1}{2}}} + \eta\Vert A^{\frac{1}{2}} u_t\Vert_{X^{-\frac{\alpha_1}{2}}} + a_1\Vert A^{\frac{1}{2}} v_t\Vert_{X^{-\frac{\alpha_1}{2}}} + \Vert f( u)\Vert_{X^{-\frac{\alpha_1}{2}}}.
$$

Thus, we need to obtain estimates for the terms that are on the right hand side of the above inequality. Using the embedding $L^{\frac{2n}{n+2\alpha_1}}(\Omega) \hookrightarrow X^{-\frac{\alpha_1}{2}}$ and Lemma \ref{Lem_Aux_Int}, condition $(i)$, we have
\[
\begin{split}
\Vert f( u) \Vert_{X^{-\frac{\alpha_1}{2}}}  &\leq c_1\Vert f( u) \Vert_{L^{\frac{2n}{n+2\alpha_1}}(\Omega)} \leq c_1\left( \int_{\Omega} [c(1 + | u|^{\rho})]^{\frac{2n}{n+2\alpha_1}} dx \right)^{\frac{n+2\alpha_1}{2n}} \\
&\leq c_2\left( |\Omega| + \int_{\Omega} | u|^{\frac{2n\rho}{n+2\alpha_1}} dx \right)^{\frac{n+2\alpha_1}{2n}}  \leq c_3\left( 1 + \Vert u\Vert_{L^{\frac{2n\rho}{(n-2)\rho + 2}}(\Omega)}^{\rho} \right).
\end{split}
\]

Since the embedding $H^1(\Omega) \hookrightarrow L^p(\Omega)$ holds, if and only if $p \leq \frac{2n}{n-2}$, and
$$
(n-2)\rho + 2 > (n-2)\rho \implies \frac{2n\rho}{(n-2)\rho + 2} < \frac{2n\rho}{(n-2)\rho} = \frac{2n}{n-2},
$$
it follows that
$$
H^1(\Omega) \hookrightarrow L^{\frac{2n\rho}{(n-2)\rho + 2}}(\Omega)
$$
and, therefore,
\[
\begin{split}
\Vert f( u) \Vert_{X^{-\frac{\alpha_1}{2}}} & \leq c_3\left( 1 + \Vert u\Vert_{L^{\frac{2n\rho}{(n-2)\rho + 2}}(\Omega)}^{\rho} \right) \leq c_5\left( 1 + \Vert u\Vert_{X^{\frac{1}{2}}}^{\rho} \right) \leq k_2.
\end{split}
\]

For the remaining terms, note that
$$
\Vert u\Vert_{X^{-\frac{\alpha_1}{2}}} \leq \tilde{c}\Vert u\Vert_{X^{\frac{1}{2}}} \leq k_3,
$$
since $X^{\frac{1}{2}} \hookrightarrow X^{-\frac{\alpha_1}{2}},$ and, moreover,
$$
\eta\Vert A^{\frac{1}{2}} u_t\Vert_{X^{-\frac{\alpha_1}{2}}} = \eta\Vert\phi\Vert_{X^{\frac{1-\alpha_1}{2}}} \leq k_4,
$$
and
$$
a_1\Vert A^{\frac{1}{2}} v_t\Vert_{X^{-\frac{\alpha_1}{2}}} = a_1\Vert\varphi\Vert_{X^{\frac{1-\alpha_1}{2}}} \leq k_5.
$$

Therefore, we conclude that
$$
\Vert A u\Vert_{X^{-\frac{\alpha_1}{2}}} \leq k_1 + k_2 + k_3 + k_4 + k_5 = C_1,
$$
as desired.

Now, to show that $ v\in X^{\frac{2-\alpha_1}{2}}$ and it is bounded, the idea is similar, because
\[
\begin{split}
&\Vert -A v\Vert_{X^{-\frac{\alpha_1}{2}}} - \Vert \eta A^{\frac{1}{2}} v_t - a_{\epsilon}(t)A^{\frac{1}{2}} u_t \Vert_{X^{-\frac{\alpha_1}{2}}} \\
&\leq \Vert -A v - \eta A^{\frac{1}{2}} v_t + a_{\epsilon}(t)A^{\frac{1}{2}} u_t \Vert_{X^{-\frac{\alpha_1}{2}}} \\
&= \Vert  v_{tt} \Vert_{X^{-\frac{\alpha_1}{2}}} = \Vert \varphi_t \Vert_{X^{-\frac{\alpha_1}{2}}} \leq k_6,
\end{split}
\]
which implies
\[
\begin{split}
\Vert v\Vert_{X^{\frac{2-\alpha_1}{2}}} & = \Vert A v\Vert_{X^{-\frac{\alpha_1}{2}}} \leq k_6 + \eta\Vert A^{\frac{1}{2}} v_t\Vert_{X^{-\frac{\alpha_1}{2}}} + a_1\Vert A^{\frac{1}{2}} u_t\Vert_{X^{-\frac{\alpha_1}{2}}} \\
& = k_6 + \eta\Vert\varphi\Vert_{X^{\frac{1-\alpha_1}{2}}} + a_1\Vert\phi\Vert_{X^{\frac{1-\alpha_1}{2}}}  \leq C_2,
\end{split}
\]
with $C_2 > 0$ being constant.

From the previous observations and from the fact that
$$
\mathbb{A}(t) = \{ \xi(t)\colon  \xi(t) \ \text{is a bounded global solution} \},
$$
we conclude that
\begin{equation}\label{resultado passo1}
\bigcup\limits_{t \in \mathbb{R}}\mathbb{A}(t) \ \text{is bounded in} \ X^{\frac{2-\alpha_1}{2}} \times X^{\frac{1-\alpha_1}{2}} \times X^{\frac{2-\alpha_1}{2}} \times X^{\frac{1-\alpha_1}{2}}.
\end{equation}

Now, we turn our attention once again to the term $\Vert f^{\prime}( u)\phi\Vert_{X^{-\frac{\alpha}{2}}}$ that appears in $(\ref{estimate0})$. Note that the embedding $X^{\frac{1-\alpha_1}{2}} = H^{1-\alpha_1}(\Omega) \hookrightarrow L^p(\Omega)$ holds, if and only if $p \leq \frac{2n}{n - 2(1 - \alpha_1)}$. Hence, using $(\ref{resultado passo1})$, the Hölder's inequality and the growth condition \eqref{Gcondition}, we have
\[
\begin{split}
\Vert f^{\prime}( u)\phi\Vert_{X^{-\frac{\alpha}{2}}}  & \leq c_1\Vert f^{\prime}( u)\phi\Vert_{L^{\frac{2n}{n+2\alpha}}(\Omega)}  \leq c_1\Vert u_t\Vert_{L^{\frac{2n}{n - 2(1 - \alpha_1)}}(\Omega)} \Vert f^{\prime}( u) \Vert_{L^{\frac{n}{1 - \alpha_1 + \alpha}}(\Omega)} \\
& \leq c_2\Vert u_t\Vert_{X^{\frac{1-\alpha_1}{2}}} \left( \int_{\Omega} [c(1 + | u|^{\rho - 1})]^{\frac{n}{1 - \alpha_1 + \alpha}} dx \right)^{\frac{1 - \alpha_1 + \alpha}{n}} \\
& \leq c_3\Vert u_t\Vert_{X^{\frac{1-\alpha_1}{2}}} \left( |\Omega| + \int_{\Omega} | u|^{\frac{(\rho - 1)n}{1 - \alpha_1 + \alpha}} dx \right)^{\frac{1 - \alpha_1 + \alpha}{n}} \\
& \leq c_4\Vert u_t\Vert_{X^{\frac{1-\alpha_1}{2}}} \left( |\Omega|^{\frac{1 - \alpha_1 + \alpha}{n}} + \Vert u\Vert_{L^{\frac{(\rho - 1)n}{1 - \alpha_1 + \alpha}}(\Omega)}^{\rho - 1} \right) \\
& \leq c_5\Vert u_t\Vert_{X^{\frac{1-\alpha_1}{2}}} \left( 1 + \Vert u\Vert_{L^{\frac{(\rho - 1)n}{1 - \alpha_1 + \alpha}}(\Omega)}^{\rho - 1} \right).
\end{split}
\]

Now, note that the embedding $X^{\frac{2-\alpha_1}{2}} = H^{2 - \alpha_1}(\Omega) \hookrightarrow L^{\frac{(\rho - 1)n}{1 - \alpha_1 + \alpha}}(\Omega)$ holds, if and only if $(2-\alpha_1) - \frac{n}{2} \geq -\frac{1 - \alpha_1 + \alpha}{(\rho - 1)}$ and $\frac{(\rho-1)n}{1 - \alpha_1 + \alpha} \geq 2$, that is
\[
\frac{(\rho-1)(n-2)}{2} +\rho(\alpha_1-1) \leq \alpha \leq   \frac{(\rho-1)n}{2}  + \alpha_1 - 1.
\]
If $ \frac{(\rho-1)(n-2)}{2} +\rho(\alpha_1-1) =  \alpha_1 + \rho(\alpha_1 - 1) \geq 0$, then using \eqref{resultado passo1} and restarting the whole process from $(\ref{estimate0})$ with $\alpha_2 = \alpha_1 + \rho(\alpha_1 - 1)$, we will get
\begin{equation*}
\bigcup\limits_{t \in \mathbb{R}}\mathbb{A}(t) \ \text{is bounded in} \ X^{\frac{2-\alpha_2}{2}} \times X^{\frac{1-\alpha_2}{2}} \times X^{\frac{2-\alpha_2}{2}} \times X^{\frac{1-\alpha_2}{2}}.
\end{equation*}

We continue with this iterative process getting  $\alpha_{k+1} = \alpha_{1} + \rho(\alpha_{k} -1)$ for $k \geq 1$ while $\alpha_k \geq 0$.

There will be an integer $k_0 \geq 1$ such that $\alpha_{k_0} \geq 0$ and $\alpha_{k_0+1} < 0$. Thus, we obtain
\begin{equation*}
\bigcup\limits_{t \in \mathbb{R}}\mathbb{A}(t) \ \text{is bounded in} \ X^{\frac{2-\alpha_{k_0}}{2}} \times X^{\frac{1-\alpha_{k_0}}{2}} \times X^{\frac{2-\alpha_{k_0}}{2}} \times X^{\frac{1-\alpha_{k_0}}{2}},
\end{equation*}
but we cannot assure the boundedness in $X^{\frac{2-\alpha_{k_0+1}}{2}} \times X^{\frac{1-\alpha_{k_0+1}}{2}} \times X^{\frac{2-\alpha_{k_0+1}}{2}} \times X^{\frac{1-\alpha_{k_0+1}}{2}}$ because of the embeddings. Here, we set $\alpha = 0$ and we restart the whole process from $(\ref{estimate0})$, with the obvious adaptations
using the boundedness in $X^{\frac{2-\alpha_{k_0}}{2}} \times X^{\frac{1-\alpha_{k_0}}{2}} \times X^{\frac{2-\alpha_{k_0}}{2}} \times X^{\frac{1-\alpha_{k_0}}{2}}$, and we conclude that
\begin{equation*}
\bigcup\limits_{t \in \mathbb{R}}\mathbb{A}(t) \ \text{is bounded in} \ X^{1} \times X^{\frac{1}{2}} \times X^{1} \times X^{\frac{1}{2}},
\end{equation*}
and the proof is complete. \qed

\section{Upper semicontinuity of pullback attractors}\label{UpSem}

This last section is devoted to study the upper semicontinuity of pullback attractors with respect to the functional parameter $a_{\epsilon}$. To this end, we will use the regularity result obtained in the previous section. Let $\{a_{\epsilon}: \epsilon\in [0, 1]\}$ be a family of real valued functions of one real variable satisfying $(\ref{function a is bounded})$. For each $\epsilon\in [0, 1]$ denote by $\{ S_{(\epsilon)}(t, \tau): t\geq\tau \in \mathbb{R} \}$ and $\{ \mathbb{A}_{(\epsilon)}(t): t \in \mathbb{R} \}$, respectively, the evolution process and its pullback attractor associated with the problem \eqref{edp01}-\eqref{cond01}.

Moreover, we will assume that $\Vert a_{\epsilon} - a_0 \Vert_{L^{\infty}(\mathbb{R})} \rightarrow 0$ as $\epsilon \rightarrow 0^{+}$.

\medskip

\noindent {\bf Proof of Theorem \ref{T-UpSemi}:} 
Let $W = W^{(\epsilon)} - W^{(0)}$, where
\[
W^{(\epsilon)} = ( u^{(\epsilon)}, u_t^{(\epsilon)}, v^{(\epsilon)}, v_t^{(\epsilon)} ) \ \ \text{and} \ \ W^{(0)} = ( u^{(0)}, u_t^{(0)}, v^{(0)}, v_t^{(0)} ),
\]
with $u = u^{(\epsilon)} - u^{(0)}$ and $v = v^{(\epsilon)} - v^{(0)}$. From this, we have 
\[
\begin{cases}
u_{tt} - \Delta u + u + \eta(-\Delta)^{\frac{1}{2}}u_t + a_{\epsilon}(t)(-\Delta)^{\frac{1}{2}} v_t^{(\epsilon)} - a_0(t)(-\Delta)^{\frac{1}{2}} v_t^{(0)} = f(u^{(\epsilon)}) - f(u^{(0)}), \\
v_{tt} - \Delta v + \eta(-\Delta)^{\frac{1}{2}}v_t - a_{\epsilon}(t)(-\Delta)^{\frac{1}{2}} u_t^{(\epsilon)} + a_0(t)(-\Delta)^{\frac{1}{2}} u_t^{(0)} = 0,
\end{cases}
\]
for all $t > \tau$ and $x\in\Omega$. Taking the inner product of the first equation with $u_t$, and also the inner product of the second equation with $v_t$, we get
\[
\begin{split}
&\frac{1}{2}\frac{d}{dt}\int_{\Omega}|u_t|^2dx + \frac{1}{2}\frac{d}{dt}\int_{\Omega}|\nabla u|^2dx + \frac{1}{2}\frac{d}{dt}\int_{\Omega}|u|^2dx + \eta\Vert (-\Delta)^{\frac{1}{4}}u_t\Vert_X^2 \\
&+ a_{\epsilon}(t)\langle (-\Delta)^{\frac{1}{2}} v_t^{(\epsilon)}, u_t^{(\epsilon)} \rangle_X - a_{\epsilon}(t) \langle (-\Delta)^{\frac{1}{2}} v_t^{(\epsilon)}, u_t^{(0)} \rangle_X \\
& - a_0(t)\langle (-\Delta)^{\frac{1}{2}} v_t^{(0)}, u_t^{(\epsilon)} \rangle_X + a_0(t)\langle (-\Delta)^{\frac{1}{2}} v_t^{(0)}, u_t^{(0)} \rangle_X \\
& = \int_{\Omega} [ f(u^{(\epsilon)}) - f(u^{(0)}) ] u_t dx,
\end{split}
\]
and
\[
\begin{split}
&\frac{1}{2}\frac{d}{dt}\int_{\Omega}|v_t|^2dx + \frac{1}{2}\frac{d}{dt}\int_{\Omega}|\nabla v|^2dx + \eta\Vert (-\Delta)^{\frac{1}{4}}v_t\Vert_X^2 - a_{\epsilon}(t)\langle (-\Delta)^{\frac{1}{2}} u_t^{(\epsilon)}, v_t^{(\epsilon)} \rangle_X \\
& + a_{\epsilon}(t)\langle (-\Delta)^{\frac{1}{2}} u_t^{(\epsilon)}, v_t^{(0)} \rangle_X + a_0(t)\langle (-\Delta)^{\frac{1}{2}} u_t^{(0)}, v_t^{(\epsilon)} \rangle_X - a_0(t)\langle (-\Delta)^{\frac{1}{2}} u_t^{(0)}, v_t^{(0)} \rangle_X = 0,
\end{split}
\]
and combining these two last equations, it follows that
\[
\begin{split}
&\frac{d}{dt} \frac{1}{2} \left( \int_{\Omega}|\nabla u|^2dx + \int_{\Omega}|u|^2dx + \int_{\Omega}|u_t|^2dx + \int_{\Omega}|\nabla v|^2dx + \int_{\Omega}|v_t|^2dx \right) \\
&+ \eta\Vert (-\Delta)^{\frac{1}{4}}u_t\Vert_X^2 + \eta\Vert (-\Delta)^{\frac{1}{4}}v_t\Vert_X^2 + (a_0 - a_{\epsilon})(t)\langle (-\Delta)^{\frac{1}{4}} v_t^{(\epsilon)}, (-\Delta)^{\frac{1}{4}} u_t^{(0)} \rangle_X \\
&+ (a_{\epsilon} - a_0)(t)\langle (-\Delta)^{\frac{1}{4}} u_t^{(\epsilon)}, (-\Delta)^{\frac{1}{4}} v_t^{(0)} \rangle_X \\
&=  \int_{\Omega} [ f(u^{(\epsilon)}) - f(u^{(0)}) ] u_t dx.
\end{split}
\]

Now, using the Young's inequality, we have
\begin{equation}\label{estimativa semicontinuidade 1}
\begin{split}
&\frac{d}{dt} \left( \Vert u\Vert_{X^{\frac{1}{2}}}^2 + \Vert u\Vert_X^2 + \Vert u_t\Vert_X^2 + \Vert v\Vert_{X^{\frac{1}{2}}}^2 + \Vert v_t\Vert_X^2 \right) \\
&= - 2\eta\Vert A^{\frac{1}{4}}u_t\Vert_X^2 - 2\eta\Vert A^{\frac{1}{4}}v_t\Vert_X^2 + 2(a_{\epsilon} - a_0)(t)\langle A^{\frac{1}{4}} v_t^{(\epsilon)}, A^{\frac{1}{4}} u_t^{(0)} \rangle_X \\
&+ 2(a_0 - a_{\epsilon})(t)\langle A^{\frac{1}{4}} u_t^{(\epsilon)}, A^{\frac{1}{4}} v_t^{(0)} \rangle_X + 2\int_{\Omega} [ f(u^{(\epsilon)}) - f(u^{(0)}) ] u_t dx \\
&\leq  2|(a_{\epsilon} - a_0)(t)| \left( \frac{1}{2}\Vert v_t^{(\epsilon)} \Vert_{X^{\frac{1}{4}}}^2 + \frac{1}{2}\Vert u_t^{(0)} \Vert_{X^{\frac{1}{4}}}^2 \right) \\
&+ 2|(a_{\epsilon} - a_0)(t)| \left( \frac{1}{2}\Vert u_t^{(\epsilon)} \Vert_{X^{\frac{1}{4}}}^2 + \frac{1}{2}\Vert v_t^{(0)} \Vert_{X^{\frac{1}{4}}}^2 \right) \\
&+ 2\int_{\Omega} |[ f(u^{(\epsilon)}) - f(u^{(0)}) ] u_t| dx \\
&\leq \Vert a_{\epsilon} - a_0 \Vert_{L^{\infty}(\mathbb{R})} \left( \Vert u_t^{(\epsilon)} \Vert_{X^{\frac{1}{4}}}^2 + \Vert u_t^{(0)} \Vert_{X^{\frac{1}{4}}}^2 + \Vert v_t^{(\epsilon)} \Vert_{X^{\frac{1}{4}}}^2 + \Vert v_t^{(0)} \Vert_{X^{\frac{1}{4}}}^2 \right) \\
&+ 2\int_{\Omega} | [ f(u^{(\epsilon)}) - f(u^{(0)}) ] u_t| dx.
\end{split}
\end{equation}

On the other hand, from Theorem \ref{RegPA}, we know that $W^{(\epsilon)}$ and $W^{(0)}$ are bounded in $X^1 \times X^{\frac{1}{2}} \times X^1 \times X^{\frac{1}{2}}$. In particular, there exists a constant $C>0$, independent of $\epsilon$, such that
\begin{equation}\label{estimativa semicontinuidade 2}
\Vert u_t^{(\epsilon)} \Vert_{X^{\frac{1}{2}}}, \Vert u_t^{(0)} \Vert_{X^{\frac{1}{2}}}, \Vert v_t^{(\epsilon)} \Vert_{X^{\frac{1}{2}}}, \Vert v_t^{(0)} \Vert_{X^{\frac{1}{2}}} \leq C.
\end{equation}

Therefore, from $(\ref{estimativa semicontinuidade 1})$, $(\ref{estimativa semicontinuidade 2})$, and the embedding $X^{\frac{1}{2}} \hookrightarrow X^{\frac{1}{4}}$, we obtain
\begin{equation}\label{estimativa semicontinuidade 3}
\begin{split}
&\frac{d}{dt} \left( \Vert u\Vert_{X^{\frac{1}{2}}}^2 + \Vert u\Vert_X^2 + \Vert u_t\Vert_X^2 + \Vert v\Vert_{X^{\frac{1}{2}}}^2 + \Vert v_t\Vert_X^2 \right) \\
&\leq C^{\prime}\Vert a_{\epsilon} - a_0 \Vert_{L^{\infty}(\mathbb{R})} + 2\int_{\Omega} |[ f(u^{(\epsilon)}) - f(u^{(0)}) ] u_t| dx,
\end{split}
\end{equation}
where $C^{\prime} > 0$ is independent of $\epsilon$.

By the Mean Value Theorem, there exists $\sigma\in (0, 1)$ such that
$$
| f(u^{(\epsilon)}) - f(u^{(0)}) | = |f^{\prime}( \sigma u^{(\epsilon)} + (1 - \sigma)u^{(0)} )| |u^{(\epsilon)} - u^{(0)}| = |f^{\prime}( \sigma u^{(\epsilon)} + (1 - \sigma)u^{(0)} )| |u|,
$$
and so
$$
\int_{\Omega} |[ f(u^{(\epsilon)}) - f(u^{(0)}) ] u_t| dx = \int_{\Omega} |f^{\prime}( \sigma u^{(\epsilon)} + (1 - \sigma)u^{(0)} )| |u| |u_t| dx.
$$

As in the proof of Theorem \ref{RegPA}, the condition $1<\rho\leq\frac{n}{n-2}$ implies $X^{\frac{1}{2}} \hookrightarrow L^{2\rho}(\Omega)$. Since $\frac{(\rho - 1)}{2\rho} + \frac{1}{2\rho} + \frac{1}{2} = 1$, then Hölder's inequality gives us
\begin{equation}\label{estimativa semicontinuidade 4}
\int_{\Omega} |[ f(u^{(\epsilon)}) - f(u^{(0)}) ] u_t| dx \leq \Vert f^{\prime}( \sigma u^{(\epsilon)} + (1 - \sigma)u^{(0)} ) \Vert_{L^{\frac{2\rho}{\rho - 1}}(\Omega)} \Vert u\Vert_{L^{2\rho}(\Omega)} \Vert u_t\Vert_{L^{2}(\Omega)};
\end{equation}
but note that
\begin{equation}\label{estimativa semicontinuidade 5}
\begin{split}
&\Vert f^{\prime}( \sigma u^{(\epsilon)} + (1 - \sigma)u^{(0)} ) \Vert_{L^{\frac{2\rho}{\rho - 1}}(\Omega)} \leq \left( \int_{\Omega} [ c(1 + | \sigma u^{(\epsilon)} + (1 - \sigma)u^{(0)} |^{\rho - 1}) ]^{\frac{2\rho}{\rho - 1}} dx \right)^{\frac{\rho - 1}{2\rho}} \\
&\leq \tilde{c} \left( |\Omega| + \int_{\Omega} | \sigma u^{(\epsilon)} + (1 - \sigma)u^{(0)} |^{2\rho} dx \right)^{\frac{\rho - 1}{2\rho}} \\
&\leq \tilde{\tilde{c}} \left( |\Omega|^{\frac{\rho - 1}{2\rho}} + \Vert \sigma u^{(\epsilon)} + (1 - \sigma)u^{(0)} \Vert_{L^{2\rho}(\Omega)}^{\rho - 1} \right) \\
&\leq \tilde{\tilde{\tilde{c}}} \left[ 1 + \left( \Vert \sigma u^{(\epsilon)} \Vert_{L^{2\rho}(\Omega)} + \Vert (1 - \sigma)u^{(0)} \Vert_{L^{2\rho}(\Omega)} \right)^{\rho - 1} \right] \\
&\leq \tilde{\tilde{\tilde{\tilde{c}}}} \left( 1 + \Vert u^{(\epsilon)} \Vert_{X^{\frac{1}{2}}}^{\rho - 1} + \Vert u^{(0)} \Vert_{X^{\frac{1}{2}}}^{\rho - 1} \right)  \leq C_0,
\end{split}
\end{equation}
where $C_0 > 0$ is independent of $\epsilon$. Thus, combining $(\ref{estimativa semicontinuidade 4})$, $(\ref{estimativa semicontinuidade 5})$ and the Young's inequality, we obtain
\begin{equation}\label{estimativa semicontinuidade 6}
\begin{split}
&\int_{\Omega} |[ f(u^{(\epsilon)}) - f(u^{(0)}) ] u_t| dx \leq C_0\Vert u\Vert_{L^{2\rho}(\Omega)} \Vert u_t\Vert_{L^{2}(\Omega)} \leq \hat{c}\Vert u\Vert_{X^{\frac{1}{2}}} \Vert u_t\Vert_X \\
&\leq \frac{\hat{c}}{2} \left( \Vert u\Vert_{X^{\frac{1}{2}}}^2 + \Vert u_t\Vert_X^2 \right) \leq \frac{\hat{c}}{2} \left( \Vert u\Vert_{X^{\frac{1}{2}}}^2 + \Vert u\Vert_X^2 + \Vert u_t\Vert_X^2 + \Vert v\Vert_{X^{\frac{1}{2}}}^2 + \Vert v_t\Vert_X^2 \right).
\end{split}
\end{equation}

Now, denoting $G(t) = \Vert u(t)\Vert_{X^{\frac{1}{2}}}^2 + \Vert u(t)\Vert_X^2 + \Vert u_t(t)\Vert_X^2 + \Vert v(t)\Vert_{X^{\frac{1}{2}}}^2 + \Vert v_t(t)\Vert_X^2$, from $(\ref{estimativa semicontinuidade 3})$ and $(\ref{estimativa semicontinuidade 6})$, it follows that
\[
\frac{d}{dt} G(t)  \leq C^{\prime}\Vert a_{\epsilon} - a_0 \Vert_{L^{\infty}(\mathbb{R})} + \hat{c}G(t)  \leq \overline{C}\Vert a_{\epsilon} - a_0 \Vert_{L^{\infty}(\mathbb{R})} + \overline{C}G(t),
\]
where $\overline{C} = \max\{ C^{\prime}, \hat{c}\}$. Since this holds for all $t > \tau$, 
and noticing that $G(\tau) = 0$, we get
$$
G(t)e^{-\overline{C}(t - \tau)} \leq -\Vert a_{\epsilon} - a_0 \Vert_{L^{\infty}(\mathbb{R})}e^{-\overline{C}(t - \tau)} + \Vert a_{\epsilon} - a_0 \Vert_{L^{\infty}(\mathbb{R})}, \ t > \tau,
$$
that is,
$$
\Vert u\Vert_{X^{\frac{1}{2}}}^2 + \Vert u\Vert_X^2 + \Vert u_t\Vert_X^2 + \Vert v\Vert_{X^{\frac{1}{2}}}^2 + \Vert v_t\Vert_X^2 \leq e^{\overline{C}(t - \tau)} \Vert a_{\epsilon} - a_0 \Vert_{L^{\infty}(\mathbb{R})} \rightarrow 0
$$
as $\epsilon \rightarrow 0^{+}$ with $t, \tau$ in compact subsets of $\mathbb{R}$, and uniformly for $W_0$ in bounded subsets of $Y_0$. This proves the first part of the result.

In order to show the upper semicontinuity of the family of pullback attractors $\{ \mathbb{A}_{(\epsilon)}(t)\colon  t \in \mathbb{R} \}$ at $\epsilon = 0$, let $\delta > 0$ be given. Let $t \in \mathbb{R}$ be fixed but arbitrary and 
$$
B \supset \bigcup\limits_{s \leq t} \mathbb{A}_{(\epsilon)}(s)
$$
be a bounded set in $Y_0$, whose existence is guaranteed by Theorem \ref{existence of the pullback attractor}. Now, let $\tau \in\mathbb{R}$, $\tau < t$, be such that
$$
d_H (S_{(0)}(t, \tau)B, \mathbb{A}_{(0)}(t)) < \frac{\delta}{2}.
$$

Using the convergence obtained in the first part of this proof, there exists $\epsilon_0 > 0$ such that
$$
\sup\limits_{u_{\epsilon} \in \mathbb{A}_{(\epsilon)}(\tau)} \Vert S_{(\epsilon)}(t, \tau)u_{\epsilon} - S_{(0)}(t, \tau)u_{\epsilon} \Vert_{Y_0} < \frac{\delta}{2}
$$
for all $\epsilon < \epsilon_0$. Finally, for $\epsilon < \epsilon_0$, we have
\[
\begin{split}
&d_H (\mathbb{A}_{(\epsilon)}(t), \mathbb{A}_{(0)}(t))  \\
&\leq d_H ( S_{(\epsilon)}(t, \tau)\mathbb{A}_{(\epsilon)}(\tau), S_{(0)}(t, \tau)\mathbb{A}_{(\epsilon)}(\tau) ) + d_H ( S_{(0)}(t, \tau)\mathbb{A}_{(\epsilon)}(\tau), \mathbb{A}_{(0)}(t) ) \\
&= \sup\limits_{u_{\epsilon} \in \mathbb{A}_{(\epsilon)}(\tau)} \Vert S_{(\epsilon)}(t, \tau)u_{\epsilon} - S_{(0)}(t, \tau)u_{\epsilon} \Vert_{Y_0} + d_H ( S_{(0)}(t, \tau)\mathbb{A}_{(\epsilon)}(\tau), \mathbb{A}_{(0)}(t) ) \\
&< \frac{\delta}{2} + \frac{\delta}{2} = \delta,
\end{split}
\]
which proves the upper semicontinuity of the family of pullback attractors. \qed

\bibliographystyle{amsplain}

\end{document}